\documentclass[a4paper, 11pt]{article}
\usepackage [a4paper,left=2.5cm,bottom=2.5cm,right=2.5cm,top=2.5cm]{geometry}
\usepackage[english]{babel}
\usepackage{amssymb}
\usepackage{amsmath,amsthm, dsfont}
\usepackage{subcaption}
\usepackage{comment}
\usepackage{tabularx,array}
\usepackage{graphicx, url}
\usepackage{enumitem} 
\usepackage{bm}
\usepackage{bbm}
\usepackage{xcolor}
\usepackage{graphicx}
\usepackage{todonotes}
\usepackage{todonotes}
\usepackage{hyperref}
\hypersetup{colorlinks=true, linkcolor=blue,citecolor=gray}
\usepackage[nameinlink,capitalise]{cleveref}
\crefformat{equation}{#2(#1)#3}

\theoremstyle{plain}
\theoremstyle{plain}\newtheorem{assumption}{Assumption}
\crefname{assumption}{Assumption}{Assumptions}
\newtheorem{theorem}{Theorem}[section]
\newtheorem{lemma}[theorem]{Lemma}
\newtheorem{corollary}[theorem]{Corollary}
\newtheorem{proposition}[theorem]{Proposition}

\newtheorem{remark}[theorem]{Remark}
\newtheorem{definition}{Definition}

\newcommand{\E}{\mathbb{E}}

\newcommand{\R}{\mathbb{R}}
\newcommand{\N}{\mathbb{N}}

\renewcommand{\P}{\mathbb{P}}

\newcommand{\change}{\color{black}}
\newcommand{\rev}{\color{black}}

\title{Fractional interacting particle system: drift parameter estimation via Malliavin calculus}

\author{Chiara Amorino\thanks{Universitat Pompeu Fabra and Barcelona School of Economics, Department of Economics and Business, Ram\'on Trias Fargas 25-27, 08005, Barcelona, Spain. Corresponding author's email: chiara.amorino@upf.edu. CA gratefully acknowledges financial support of PID2022-138268NB-I00/AEI/10.13039/501100011033.}, Ivan Nourdin\thanks{Unit\'e de Recherche en Math\'ematiques, Universit\'e du Luxembourg, Maison du Nombre, 6 avenue de la Fonte, L-4364 Esch-sur-Alzette, Grand Duchy of Luxembourg. IN's research is supported by the Luxembourg National Research Fund (Grant: O22/17372844/FraMStA).}, Radomyra Shevchenko\thanks{Universit\'e Cote d’Azur, Laboratoire J.A. Dieudonn\'e, UMR CNRS 7351, Nice, 06108, France, and Centrale Méditerranée.}}

\date{\today}
\begin{document}
\maketitle

\begin{abstract}
We address the problem of estimating the drift parameter in a system of $N$ interacting particles driven by additive fractional Brownian motion of Hurst index \( H \geq 1/2 \). Considering continuous observation of the interacting particles over a fixed interval \([0, T]\), we examine the asymptotic regime as \( N \to \infty \). Our main tool is a random variable reminiscent of the least squares estimator but unobservable due to its reliance on the Skorohod integral. We demonstrate that this object is consistent and asymptotically normal by establishing a quantitative propagation of chaos for Malliavin derivatives, which holds for any \( H \in (0,1) \). Leveraging a connection between the divergence integral and the Young integral, we construct computable estimators of the drift parameter. These estimators are shown to be consistent and asymptotically Gaussian. Finally, a numerical study highlights the strong performance of the proposed estimators. 
 

\end{abstract}

\tableofcontents

\section{Introduction}
Diffusion phenomena, often modeled as solutions to stochastic differential equations (SDEs), are widely studied in fields ranging from probability theory to functional analysis and differential geometry. The study of statistical properties for such models has gained prominence due to extensive applications in areas such as finance, biology \cite{46Est}, neurology \cite{29Est}, and economics \cite{11Est}, as well as classical fields like physics \cite{44Est} and mechanics \cite{36Est}. In pharmacology, SDEs model variability in biomedical experiments, as reviewed in \cite{22Wang, 21Wang, 23Wang,24Wang, Wang}.

Interacting particle systems (IPS) extend the applications mentioned for classical SDEs by modeling the dynamics of multiple interacting entities, which are crucial in applications ranging from mathematical biology and social sciences to data science and optimization. Boltzmann's seminal work \cite{24poc} introduced the kinetic theory of gases based on interacting particle system and molecular chaos, i.e. statistical independence of particles in the limit. Later, Kac \cite{104poc} established the notion of the propagation of chaos, further explored by McKean \cite{61imp} through diffusion models.

Recent decades have seen IPS applied to self-organization models in biology and social sciences \cite{2poc,59poc,132poc,134poc}, mean-field games \cite{30poc,CarDel18}, and data science, particularly in neural networks \cite{47poc,56poc,142poc,147poc}, optimization \cite{39poc,90poc,155poc}, and MCMC methods \cite{65poc}.

In the classical SDE framework, replacing standard Brownian motion with fractional Brownian motion has proven highly effective in better modeling certain real-world phenomena. Indeed, fractional Brownian motion (fBm) has emerged as a powerful tool for modeling processes driven by a long-range dependent and self-similar noise. It has been used in diverse fields, including finance \cite{Chr12, Com12, Fle11, Gra04, Ros09}, geophysics \cite{Earthquakes, Hurst}, traffic patterns \cite{Cag04, Nor95, Will}, and medical research \cite{Lai00}. These developments have naturally driven interest in statistical inference for fBm.

Naturally, one might wonder whether this extension is equally fruitful to the interacting case. It is thus fitting to enhance the interacting particle systems  by incorporating fractional noise, aiming for a closer alignment with reality. This leads to the focus of our work: IPS driven by fBm. This model offers a robust framework for capturing interactions, such as those between financial assets or neurons in neuroscience.

Recently, IPS have also been applied to opinion dynamics (see \cite{op1, op2}), modeling how interactions among individuals influence opinion shifts. This has been particularly relevant in understanding public opinion on COVID-19 vaccinations, as evidenced by several studies, including \cite{cov3, cov1, cov2}. Notably, as new interactions disseminate information with varying memory effects, triggering shifts that may be either abrupt or gradual, IPS driven by fBm appear well-suited to capture these dynamics. The Hurst parameter \( H \in (0, 1) \) governs the extent of memory in the system, making it a powerful tool for analyzing such phenomena. Building on these insights, we focus on IPS driven by fBm, which provide a richer framework for capturing memory effects and long-range dependencies in opinion dynamics.

{\change In the geophysical context, IPS-based models have recently begun to emerge for processes related to turbulent motion, such as turbulent kinetic energy (see \cite{Bossy2022}). At the same time, empirical observations suggest that non-Markovian driving noise may offer a more accurate representation of such phenomena (see \cite{Franzke2015, Lilly2017}), thereby motivating the study of IPS driven by fBm.

A further application arises in a recent study \cite{House2025}, which examines the growth of $N$ axon cells in vertebrate brains. While it is typically modeled using independent fBm paths $B^{H,i}$, the authors argue that a self-repulsion effect observed in these cells indicates the need for a more realistic modeling approach, one that accounts for interactions between particles, as in the IPS framework. A simple version of such a model is captured by the SDE
$$ dX_t^{i,N} =  \theta^0 \left(X^{i,N}_t-\frac{1}{N}\sum_{j=1}^N X^{j,N}_t\right)dt +dB^{H,i}_t,$$
where $\theta^0\neq 0$ quantifies the degree of attraction or repulsion of each particle $X^{i,N}$ relative to the empirical mean. A dedicated simulation study for this model will be presented in the article, alongside the (much more general) theoretical analysis.
}

Our aim consists in studying a drift parameter estimation problem for the following IPS, driven by a fBm
\begin{equation}
\begin{cases}
 d X_t^{\theta,i, N} =  \, \sum_{m=1}^p \theta_m b_m \big(X_t^{\theta,i, N}, \mu_t^{\theta, N} \big) dt +  \sigma d B_t^{H,i},  \qquad i = 1, ... , N, \quad t\in [0, T], \\[1.5 ex]
 \mathcal{L} \big( X_0^{{\theta,} 1, N}, ... , X_0^{{\theta,} N, N} \big) : = \mu_0 \times ... \times \mu_0.
\end{cases}
 \label{eq: model}
\end{equation}
Here the unknown parameter vector $\theta:= (\theta_1,\dots ,\theta_p)^T$ belongs to some compact and convex set $\Theta \subset \R^p$, the processes $(B^{H,i}_t)_{t \in [0, T]}$, $i=1,\dots,N$, are independent $\mathbb{R}$-valued fractional Brownian motions of Hurst index $H \in (0, 1)$, independent of the initial value $(X^{\theta,1,N}_0, \dots, X^{\theta,N,N}_0)$ of the system, $\sigma >0$ is a volatility parameter, and $\mu_t^{\theta, N}$ is the empirical measure of the system at time $t$, i.e.\
\[
\mu_t^{\theta, N} := \frac{1}{N} \sum_{i = 1}^N \delta_{X_t^{\theta, i, N}}.
\]
The drift coefficient $b=(b_1,\dots, b_p)$ consists of functions $b_m: \R \times \mathcal{P}_2(\R) \rightarrow \R$, $m=1,\dots , p$, where $\mathcal{P}_l$ denotes the set of probability measures on $\R$ with a finite $l$-th moment, endowed with the Wasserstein $l$-metric
\begin{equation}
W_l(\mu, \nu) := \Big( \inf_{m \in \Gamma (\mu, \nu)} \int_{\R^2} |x - y|^l m(dx, dy) \Big)^{\frac 1 l},
\label{eq: wass}
\end{equation}
where $\Gamma(\mu, \nu)$ denotes the set of probability measures on $\R^2$ with marginals $\mu$ and $\nu$. The observations considered in this work are given by  
\begin{equation}{\label{eq: observations}}
\big(X_{t}^{\theta,i,N}\big)_{t \in [0,T]}^{i=1, \ldots, N},
\end{equation}  
where the time horizon \( T \) is fixed, and the number of particles \( N \) tends to infinity. In this asymptotic framework, the central idea for analysing the properties of the system is the propagation of chaos phenomenon, that is, weak convergence of the measure $\mu_t^{\theta, N}$ to the measure $\bar{\mu}_t^{\theta}=\mathcal{L}(\bar{X}_t^{\theta})$ in the McKean-Vlasov SDE
\begin{equation}
d \bar{X}_t^{\theta} =  \sum_{m=1}^p \theta_m b_m \big(\bar{X}_t^{\theta}, \bar{\mu}_t^{\theta} \big) dt + \sigma d B_t^H,  \quad t\in [0, T],
\label{eq: McK}
\end{equation}
driven by an fBm $B^H$, see also Section \ref{s: poc}. This result allows for a type of averaging, similarly to the ergodic theorem in the asymptotic regime where \( N = 1 \) and \( T \to \infty \).


Even when the system is driven by Brownian noise, the literature on statistical inference for interacting particle systems remains limited. Early work, such as Kasonga’s contribution \cite{Kas}, marks the field’s inception. However, initial research primarily focused on microscopic particle systems arising from statistical physics, which were not directly observable. Recent years have seen a shift, fueled by the emergence of diverse applications (mentioned above) generating new data, sparking interest in this field among statisticians.  

Significant advancements in nonparametric and semiparametric statistical inference for these systems can be found in works like \cite{Pol, Vyt, DelHof}, and \cite{Richard}. Parameter estimation, based on observations of interacting particle systems and their associated McKean-Vlasov equations, has been investigated in various asymptotic regimes; see, for example, \cite{Amo23, Bis, GenLar1, GenLar2, Giesecke, Liu, Imp, Wen}, and references therein. Notably, most of these contributions have emerged in the past five years, reflecting the novelty and growing interest in this field.  

Despite this progress, no work currently addresses statistical inference for interacting particle systems driven by fBm — a gap our study seeks to fill. The absence of statistical results in this framework is not surprising since the probabilistic tools central to this endeavor, such as the propagation of chaos (detailed in Section \ref{s: poc}), are still under development. This motivates our investigation of this promising area, which we believe has substantial potential for future applications.  

It is important to note that statistical inference for IPS driven by standard Brownian motion heavily relies on the Markovian and semimartingale properties of the process. These properties cannot be directly extended to the fractional Brownian framework due to the non-Markovian and non-semimartingale nature of fBm, which introduces unique challenges. Several foundational contributions have been made in the context of parameter estimation for ergodic SDEs driven by fBm, which we summarize below to provide context for our study.

The most common approaches in the literature are the maximum likelihood estimator (MLE) and the least squares estimator (LSE). Notably, in the case of Brownian motion, the least squares estimator coincides with the maximum likelihood estimator; however, this equivalence no longer holds when considering fractional Brownian motion (see \cite{HuNua10, HuNuaZho19, KleLeB02, TudVie07}).

For continuous observations, Kleptsyna and Le Breton \cite{KleLeB02} proved the consistency of the MLE for a fractional Ornstein-Uhlenbeck process when \( H > 1/2 \). Later, Brouste and Kleptsyna demonstrated a central limit theorem for the MLE, which was also independently established by Bercu et al. \cite{Bercu}, using a different approach. Both studies focus on the case \( H > 1/2 \). The consistency of the MLE for any \( H \in (0, 1) \) was established in the seminal work by Tudor and Viens \cite{TudVie07}. Using Malliavin calculus, they extended their results to include a linear drift in the parameter, demonstrating the consistency for both the continuous MLE and its discrete approximation.

On the other hand, the LSE was proposed in \cite{HuNua10} for estimating the drift parameter in a fractional Ornstein-Uhlenbeck process. Consistency and asymptotic normality were proven in this setting for \( H \in [1/2, 3/4) \). These results were later extended in \cite{HuNuaZho SISP}, using a novel method that established validity for any \( H \in (0,1) \). Their approach relies on Malliavin calculus, and it is worth highlighting that the asymptotic Gaussianity of the LSE crucially depends on the process being an Ornstein-Uhlenbeck process. This specific structure allows the explicit expression of the estimator, showing that it belongs to the second Wiener chaos, thereby enabling the application of a central limit theorem for multiple stochastic integrals.

A more general setting, involving multidimensional but linear drifts, is considered in \cite{HuNuaZho19}. The authors proved the consistency of the LSE for \( H > 1/4 \). However, this LSE involves Skorohod integrals, making it impractical for computation. To address this, Neuenkirch and Tindel \cite{NeuTin} demonstrated the strong consistency of a version of LSE based on discrete observations. Their approach employs Young’s inequality from rough path theory to manage Skorohod integrals, and it is valid for \( H > 1/2 \).

An alternative approach to parameter estimation based on identifying the invariant measure was also proposed and extensively studied in \cite{PTV20,HarHu21,HarRic23}.


In this work, we focus on estimating \(\theta\) from the observations specified in \eqref{eq: observations}. At the heart of our approach lies an unobservable quantity, \(\tilde{\theta}_N\), inspired by a least-squares-type estimator. However, this quantity cannot be considered a true estimator, as it depends on the data through a Skorohod integral, rendering it inherently unobservable. Consequently, we will refer to this random variable as a 'fake-stimator' or 'estim-actor' in the sequel. To address this problem, we exploit a well-known relationship between the Skorohod integral and the Stratonovich integral, which holds for \(H > 1/2\). This allows us to construct computable estimators, as detailed in Section \ref{s: ass}.

As one might wish to compare our results with those existing in the literature, primarily for single SDEs with \(T \to \infty\), it is worth noting that, in recent studies on interacting particle systems driven by standard Brownian motion, the roles of \(T\) and \(N\) have often been observed to be analogous. This analogy facilitates nonparametric estimation of the drift function in an IPS over a fixed time horizon, as demonstrated in \cite{DelHof}, and parametric estimation in \cite{Amo23}. Other works highlighting this observation include \cite{AmoPil, Shiwi, Imp}, among others.  

In the case of IPS driven by independent Brownian motions, the interaction between particles is weak and manifests only in the coefficients, rendering the particles nearly independent. For \(H = \frac{1}{2}\), this independence makes it intuitive why the asymptotic behavior over \(T\) can be replaced by that over \(N\): observations from the second particle can be conceptually concatened onto the end of the first, and the system's Markovianity ensures the validity of such inference. However, this Markovianity (and consequently, this argument) no longer holds in our setting with \(H \neq \frac{1}{2}\). It is therefore more surprising that we can still recover information about the drift coefficient, even with observations limited to a finite time horizon.  

Our main results establish that the proposed fake-stimator for parameter estimation in the drift is both consistent and asymptotically Gaussian. Furthermore, we demonstrate that the errors introduced when transitioning from this theoretical estim-actor to the computable estimators are negligible. As a result, the desirable asymptotic properties of \(\tilde{\theta}_N\) extend to the actual estimators.

To the best of our knowledge, in all the aforementioned literature on parameter estimation for classical SDEs driven by fBm, the only results that establish a central limit theorem for the proposed estimators pertain to fractional Ornstein-Uhlenbeck processes. These results heavily rely on the specific structure of the drift, making it somewhat surprising that we succeed in developing a central limit theorem for our estimator in a more complex setting. By employing empirical projections, one can demonstrate that our system of \(N\) equations in \(\mathbb{R}\) can be reformulated as a single equation in \(\mathbb{R}^N\). This transformation reveals that our fake-stimator has the same form as the least squares estimator found in the existing literature (see Remark \ref{rk: projection} for details). However, there are notable differences in the asymptotic regimes. On the one hand, our fixed time horizon \(T\) simplifies certain computations; on the other hand, the increasing dimensionality of the fBm, as \(N \to \infty\), introduces substantial additional challenges. This makes the parallelism between \(T\) and \(N\), discussed earlier, less apparent in our setting, where the absence of Markovianity adds further complexity.  

In a framework closely related to ours, Comte and Marie \cite{ComMar19} analyze \(N\) i.i.d. copies of fractional SDEs over a fixed time horizon. They propose a nonparametric estimator for the drift based on Skorohod integrals and establish its consistency, along with rates of convergence for both the estimator and its computable approximation. Separately, Marie \cite{Mar23} examines \(N\) i.i.d. copies of our model with a one-dimensional parameter \(\theta\). He establishes convergence results for a computable approximation of the LSE, based on a fixed point argument, valid for \(H > 1/3\).

Comparing our results to those mentioned above, it is important to note that the interactions among the particles introduce several additional challenges. In particular, due to the presence of the empirical measure in \eqref{eq: model}, the Malliavin derivative of particle \(i\) with respect to the fractional Brownian motion of particle \(j\) is non-zero, which contrasts with the case of independent particles where this derivative is zero. This factor significantly complicates our analysis. For example, in the absence of interactions, one can derive explicit expressions for the Malliavin derivative (see \eqref{eq: expression DX}). However, when interactions are present, such equalities no longer hold, and errors arise when transitioning from the Malliavin derivatives to other quantities. We must carefully bound these errors in order to successfully transfer the asymptotic properties of the estim-actor to the computable estimator.

The challenges introduced by the interactions are even more pronounced when proving the asymptotic properties of the estim-actor. To establish its consistency, the main tool we rely on is the propagation of chaos, which enables us to approximate interacting particles by independent ones. However, since our analysis heavily depends on stochastic calculus, the mere propagation of chaos for the particles is insufficient to derive our results. It is also crucial to prove that the Malliavin derivatives of the interacting particles converge, as \(N \to \infty\), to the Malliavin derivatives of the independent particles.

More precisely, the core tool that allows us to derive our main results is the propagation of chaos for the Malliavin derivatives of the process (see Theorem \ref{cor: PoC-deriv}), which holds for any \(H \in (0,1)\). Furthermore, the convergence rate for this propagation of chaos, off the diagonal, is faster than the rate for the particles themselves. This difference is the key reason we are able (quite surprisingly) to recover the asymptotic normality of the fake-stimator (see Section \ref{s: hunt} for details).


Passing from the fake-stimator to computable estimators leads to the appearance of the Malliavin derivative of the drift, which is also unobservable and requires approximation. We propose two different versions of this approximation. The first, denoted \(\hat{\theta}_{N, \epsilon}\), is based on the increments of the process with respect to the initial condition. The second, \(\hat{\theta}_N^{(fp)}\), is based on a fixed-point argument. On one hand, we prove that, by choosing \(\epsilon\) small enough, \(\hat{\theta}_{N, \epsilon}\) retains the same properties as the fake-stimator (see Theorem \ref{th: estim increments}). On the other hand, the approximation provided by \(\hat{\theta}_N^{(fp)}\) incurs some loss in precision. Specifically, we show that this computable estimator is both consistent and asymptotically Gaussian, though with a larger variance compared to that of the estim-actor (see Theorem \ref{th: fp}). Additionally, we introduce an iterative estimator that converges to the fixed-point estimator and demonstrate that it behaves similarly to its limit. Finally, we present numerical results that highlight the strong performance of the computable estimators, with practical results that align well with our theoretical findings. {\rev In addition, in our numerical study we also consider an estimator obtained via a contrast function, which provides a natural framework when only discrete observations of the system are available. This approach, inspired by the classical contrast-based methods for drift estimation in diffusion models, performs well in practice and can be seen as a promising direction toward the extension of our results to the discrete-observation setting. A rigorous theoretical analysis of this estimator, as well as a comprehensive study of the impact of discretization errors, is left for future research.}\\
\\
The structure of the paper is as follows. In Section \ref{s: ass}, we introduce the estimators to be analyzed throughout the paper and outline the assumptions on the model. Under these assumptions, we derive our results on the propagation of chaos, which are presented in Section \ref{s: poc}. Section \ref{s: main} contains the statements of our main results, divided into those concerning the fake-stimator (in Section \ref{s: fake}) and those pertaining to the computable estimators (in Section \ref{s: computable}). Section \ref{s: numerical} presents the numerical results, while Section \ref{s: preliminaries} provides the preliminary concepts and tools necessary for the proofs of our main results. The proof of all key results is given in Section \ref{s: proof main}, and the details required for the preliminary results are provided in Section \ref{s: proof technical}.

\subsection*{Notation}
Throughout the paper all positive constants are denoted by $c$. All vectors are row vectors, $\| \cdot \|$ denotes the Euclidean norm for vectors and $\langle \cdot, \cdot \rangle$ the associated scalar product. For any $f: \R \times \mathcal{P}_l \rightarrow \R$, we denote by $\partial_x f$ the partial derivative of a function $f(x, \mu )$ with respect to $x$ and by $\partial_\mu f$ the partial derivative of a function $f(x, \mu )$ with respect to the measure $\mu$ in the sense of Lions (see Section \ref{s: lions}). We say that a function $f:\R \times \mathcal{P}_l \to \R$ has polynomial growth if 
\begin{align} \label{eq: pol growth}
|f(x,\mu)| \le C (1 +|x|^k +W_2^l(\mu,\delta_0) )
\end{align}
for some {$k,l =0,1,\dots$ and all $(x,\mu)\in \R\times \mathcal{P}_l$. Moreover, we say that $f:\R \times \mathcal{P}_l(\R) \to \R$ is locally Lipschitz if for all $(x_1,\mu_1), (x_2, \mu_2) \in \R\times \mathcal{P}_l(\R)$ and 
for some $k,l =0,1,\dots$
\begin{align}{\label{eq: def loc lip}}
|f(x_1,\mu_1) - f(x_2,\mu_2)| \le C (|x_1 - x_2| + W_2(\mu_1, \mu_2)) (1 +|x_1|^k + |x_2|^k +W_2^l(\mu_1,\delta_0) +W_2^l(\mu_2,\delta_0) ).
\end{align} We denote by $\theta^0$ the true value of the parameter vector and we suppress the dependence of several objects on the true parameter $\theta^0$. In particular, we write $\P := \P^{\theta^0}$, $\E := \E^{\theta^0}$  
$X_t^{i,N}:=X_t^{\theta^0, i, N}$, $\bar X_t^{i} := \bar X^{\theta^0, i}_t$, $\mu_t^{N}:=\mu_t^{\theta^0, N}$ and $\bar \mu_t := \bar \mu^{\theta^0}_t$. Furthermore, we denote by $\xrightarrow{\mathbb{P}}$, $\xrightarrow{\mathcal{L}}$, 
$\xrightarrow{L^p}$ the convergence in probability, in law and in $L^p$, respectively.

\section{Estimators and assumptions}{\label{s: ass}}
This section introduces the concept of what we have referred to in the introduction as the 'fake-stimator' or 'estim-actor', which is based on the least-squares-type estimator, along with the computable estimators that we use to approximate it. \\
\\
Consider a classical SDE with \(N=1\) and \(T \to \infty\):
\[
dX_t^\theta = \langle \theta, b(X_t^\theta) \rangle dt + \sigma dB_t^H.
\]
The least-squares-type estimator is derived by minimizing the error term \(\int_0^T |\dot{X}_t - \langle \theta, b(X_t) \rangle|^2 dt\) with respect to \(\theta\), where \(X_t = X_t^{\theta^0}\).

In the literature on parameter estimation of interacting particle system, for \(H = \frac{1}{2}\), the maximum likelihood estimator for the classical SDE has been extensively replaced by an estimator that maximizes the sum of the likelihood functions computed for each particle. Therefore, it seems natural to replace the error expression above with the sum of the errors across all particles, aiming to minimize the quantity:
$$\sum_{i=1}^N \int_0^T |\dot{X}_t^{i,N} - \langle \theta ,\,b(X_t^{i,N}, \mu_t^N) \rangle |^2 dt.$$
This is equivalent to minimizing the expression:
$$\sum_{i=1}^N \int_0^T |\dot{X}_t^{i,N}|^2 dt - 2\sum_{i=1}^N \int_0^T \langle \theta ,\, b(X_t^{i,N}, \mu_t^N)\rangle dX_t^{i,N} +  \sum_{i=1}^N \int_0^T \langle \theta ,\, b(X_t^{i,N}, \mu_t^N)\rangle^2 dt,$$
which forms a quadratic function of \(\theta\), although the term \(\int_0^T |\dot{X}_t^{i,N}|^2 dt\) does not exist. The formal minimizer of this quadratic form is given by the $\R^p$ vector
\begin{equation}\label{eq: def estimator}
\tilde{\theta}_N := \Psi_N^{-1}\cdot \sum_{i=1}^N \int_0^T b(X_t^{i,N}, \mu_t^N) dX_t^{i,N} = \theta^0 +  \Psi_N^{-1}\cdot \sum_{i=1}^N \int_0^T b(X_t^{i,N}, \mu_t^N) \sigma dB_t^{i,H},
\end{equation}
where the integrals are taken componentwise, and $\Psi_N$ is the $p\times p$ matrix given by
$$\left(\Psi_N\right)_{lj}:=\left(\sum_{i=1}^N \int_0^T b_l(X_t^{i,N}, \mu_t^N)b_j(X_t^{i,N}, \mu_t^N) dt\right)_{lj}.$$
To obtain this, it is enough to observe that
\begin{align*}
    &\int_0^T b(X_t^{i,N}, \mu_t^N) dX_t^{i,N} = \int_0^T b(X_t^{i,N}, \mu_t^N) ((\theta^0 )^T b (X_t^{i,N}, \mu_t^N) )^Tdt + \int_0^T b(X_t^{i,N}, \mu_t^N) \sigma dB_t^{i,H}\\
    &= \int_0^T b(X_t^{i,N}, \mu_t^N) b (X_t^{i,N}, \mu_t^N)^T \theta^0 dt + \int_0^T b(X_t^{i,N}, \mu_t^N) \sigma dB_t^{i,H}= \Psi_N \theta^0 + \int_0^T b(X_t^{i,N}, \mu_t^N) \sigma dB_t^{i,H}.\\
\end{align*}
In this context, we interpret the stochastic integral \(\int_0^T b(X_t^{i,N}, \mu_t^N) \sigma dB_t^{i,H}\) as a divergence-type (or Skorohod) integral.
Recall that when \(H = \frac{1}{2}\), the divergence-type integral simplifies to the standard Itô stochastic integral, which can be approximated using forward Riemann sums. However, for \(H > \frac{1}{2}\), the Skorohod integral is defined as the limit of Riemann sums involving the Wick product (see \cite{Duncan2000}). 
This Wick product-based approximation, however, is impractical for numerical simulations since the Wick product cannot be directly computed from the values of the factors involved. Fortunately, a well-known relationship between the Skorohod integral and the Stratonovich integral, involving the Malliavin derivatives, offers a pathway to develop a computationally feasible estimator (see Proposition 5.2.3 in \cite{NuaBook} for details):  

\begin{equation}{\label{eq: link divergence Young}}
\int_0^T b(X_t^{i,N}, \mu_t^N) \circ dB_t^{i,H} = \int_0^T b(X_t^{i,N}, \mu_t^N) dB_t^{i,H} + \int_0^T \int_0^T D^i_s b(X_t^{i,N}, \mu_t^N) \phi(t,s) \, ds \, dt,
\end{equation} 
where we have denoted as $\int_0^T b(X_t^{i,N}, \mu_t^N) \circ  dB_t^{i,H}$ the Stratonovich integral and we have introduced the notation
 \begin{equation}{\label{eq: def phi}}
\phi(t,s) := H(2H -1) |t-s|^{2H-2}.
 \end{equation}
Moreover, \( D^i_s b(X_t^{i,N}, \mu_t^N) \) represents the Malliavin derivative, rigorously computed below, in \eqref{eq: DFx}. However, the Malliavin derivatives of the particles are not directly observable, so they must be approximated.  

It is worth noting that, in the absence of interactions, one can easily obtain the following result {\rev (see Proposition~2.7 in \cite{ComMar19} for a proof of the equality below, based on the dynamics of the processes under consideration)}: 
\begin{equation}{\label{eq: expression DX}}
D_s X_t = \sigma \frac{\partial_{x_0} X_t}{\partial_{x_0} X_s} 1_{s \leq t}  
= \sigma \exp\left(\int_s^t \langle \theta^0, \partial_x b(X_r) \rangle \, dr\right) 1_{s \leq t},
\end{equation}
where \( x_0 \) is the initial condition. 

However, when interactions are introduced, the situation becomes significantly more complex. Controlling the error produced in the approximation poses a considerable challenge. Nevertheless, this argument ultimately justifies the introduction of the following (computable) estimators. The first equality in \eqref{eq: expression DX} leads us to
\begin{align}{\label{eq: def estimator increments}}
&\hat{\theta}_{N, {\epsilon}}^{(2)} := \Psi_N^{-1}{\sum_{i=1}^N \int_0^T b(X_t^{i,N}, \mu_t^N) \circ dX_t^{i,N}}\nonumber \\
&\quad- \Psi_N^{-1}{\sum_{i=1}^N \int_0^T \int_0^t \partial_x b(X_t^{i,N}, \mu_t^N) \frac{\frac{1}{\epsilon}(X_t^{i, x_0^i + \epsilon}- X_t^{i, x_0^i} )}{\frac{1}{\epsilon}(X_s^{i, x_0^i + \epsilon}- X_s^{i, x_0^i} ) \lor 1} \sigma \phi(t,s) ds \, dt}.
\end{align}
This estimator is practical in cases where one can observe two different paths of \( X^i \) for each individual \( i \), with different but close initial conditions. The proximity of the initial conditions is determined by \( \epsilon \). 

The second approximation of the Malliavin derivative as in \eqref{eq: expression DX}, using the exponential representation outlined above, motivates the introduction of the following fixed-point estimator, provided that such an object exists:  
\begin{align}\label{eq: def fp estimator}
    \hat{\theta}_{N}^{(fp)} & :=  \Psi_N^{-1} \sum_{i=1}^N \int_0^T b(X_t^{i,N}, \mu_t^N) \circ dX_t^{i,N} \\
    & \quad - \Psi_N^{-1} \sum_{i=1}^N \int_0^T \int_0^t \partial_x b(X_t^{i,N}, \mu_t^N)
    \exp\!\left(\int_s^t \langle \hat{\theta}_{N}^{(fp)}, \partial_x b(X_r^{i,N}, \mu_r^N) \rangle \, dr\right)
    \sigma \phi(t, s) \, ds \, dt. \nonumber
\end{align}
We denote this expression compactly as  
\[
\hat{\theta}_{N}^{(fp)} =: F_N(\hat{\theta}_{N}^{(fp)}),
\]
{\rev where $F_N : \R^p \to \R^p$ is the mapping defined by
\begin{align*}
F_N(\theta)
&= \Psi_N^{-1} \sum_{i=1}^N \int_0^T b(X_t^{i,N}, \mu_t^N) \circ dX_t^{i,N} \\
&\quad - \Psi_N^{-1} \sum_{i=1}^N \int_0^T \int_0^t \partial_x b(X_t^{i,N}, \mu_t^N)
\exp\!\left(\int_s^t \langle \theta, \partial_x b(X_r^{i,N}, \mu_r^N) \rangle \, dr\right)
\sigma \phi(t, s) \, ds \, dt.
\end{align*}
}

Furthermore, we introduce a sequence of iterative estimators designed to converge to the fixed-point estimator. 

Our primary goal is to analyze these estimators and establish their desirable asymptotic properties, including consistency and asymptotic Gaussianity. These properties are derived from the corresponding results obtained for the fake-stimator defined in \eqref{eq: def estimator}, combined with the observation that the errors introduced in transitioning between the two are negligible. {\rev Note that, in order to estimate the parameter in the drift, the proposed estimator is based on observations of the interacting particle system. However, the analysis of its asymptotic properties requires moving from the interacting particles to an independent particle system, so that classical limit results such as the law of large numbers and the central limit theorem for averages of i.i.d. random variables can be applied. The key tool that allows this transition is the propagation of chaos. In particular, we need to establish this property not only for the particles themselves but also for their Malliavin derivatives, as both play a fundamental role in the derivation of our theoretical results.}}


Notice that our system can be interpreted as a single equation in \(\mathbb{R}^N\), rather than as \(N\) separate equations in \(\mathbb{R}\). This is formalized through the concept of \textit{empirical projection}, as defined in Definition \ref{def: empirical projection} below. To this end, we introduce the function \(B^N : \mathbb{R}^N \to \mathbb{R}^{N\times p}\), defined as:
$$B^N({x}) := (b^{1,N}({x}), \dots, b^{N,N}({x}))^T := \left(b(x^1, \mu^N), \dots, b(x^N, \mu^N)\right)^T,$$
where \(\mu^N\) represents the empirical distribution and $M^T$ denotes the transpose of the matrix $M$.

Using the empirical projection, we can reformulate the interacting particle system in \eqref{eq: model} as the following high-dimensional SDE in \(\mathbb{R}^N\):
$$d{X}^\theta_t = B^N({X}^\theta_t)  \theta dt + \sigma  d {B}_t^H,$$
where \({B}^H = (B^{1,H}, \dots, B^{N,H})^T\) is an \(N\)-dimensional fractional Brownian motion.

\begin{remark}{\label{rk: projection}}
It is worth noting that in this context, the least squares fake-stimator takes the same form as in \cite{HuNuaZho19}:
$$\tilde{\theta}_N := \theta^0 + \left( \int_0^T (B^N)^T B^N({X}_s) ds \right)^{-1} \int_0^T B^N({X}_s) \sigma  d{B}_s^H.$$
At first glance, one might (mistakenly) assume that our result is similar or comparable to the one in \cite{HuNuaZho19}. However, this is not the case due to the significant differences in the asymptotic regimes. 

In our framework, the time horizon \(T\) is fixed, which simplifies some computations. On the other hand, the dimension of the fractional Brownian motion increases to infinity under our assumptions, introducing many new challenges.
\end{remark}

In the sequel, we will always assume that the following assumptions hold true. 
\begin{assumption} \label{as: lip}
(Lipschitz condition) The drift coefficient is Lipschitz continuous in $(x,\mu)$, i.e.\ there exists a constant $c > 0$ such that for all $(x,\mu), (y,\nu) \in \R \times {\cal P}_2(\R)$, and for each $m=1,\dots , p$
$$|b_m(x, \mu) - b_m (y, \nu)| \le c ( |x - y| + W_2(\mu, \nu) ).$$
\end{assumption}

\begin{assumption} \label{as: bound}
There exists a constant $c > 0$ such that for each $m=1,\dots , p$, we have $|b_m(0, \delta_0)| \le c$, where $\delta_0$ denotes the Kronecker delta.
\end{assumption}

Observe that Assumptions \ref{as: lip} and \ref{as: bound} together imply the linear growth of $b_m$, $m=1,\dots , p$, in the following sense: $\forall x \in \R$, $\forall \mu \in \mathcal{P}_2(\R)$; there exists a constant $c > 0$ such that
\begin{equation}{\label{eq: linear b}}
    |b_m(x, u)| \le c(1 + |x| + W_2(\mu, \delta_0)). 
\end{equation}

\begin{assumption} \label{as: moments}
(Boundedness moments) For all $k \ge 1$
$$C_k:=\int_{\R} |x|^k \mu_0(dx)<\infty.$$
\end{assumption}

Under the Assumptions \ref{as: lip}, \ref{as: bound} and \ref{as: moments}, both Equations \eqref{eq: model} and \eqref{eq: McK} are well-posed, meaning that their solutions exist and are unique. The well-posedness of \eqref{eq: model} is relatively standard and becomes intuitively clear when considering its reformulation through empirical projection. However, the well-posedness of \eqref{eq: McK} is more intricate and can be found in Theorem 3.1 of \cite{FanHuang}.

{\rev Observe that the assumptions stated above are quite standard. In \cite[Section 2.2]{Amo23}, several examples of drift functions satisfying the same assumptions are discussed in the context of drift and diffusion coefficient estimation (with a more general diffusion coefficient) for particles driven by standard Brownian motion. Specifically, the drift functions in their Examples (i), (ii), and (iii) satisfy our assumptions. While this is straightforward for Examples (i) and (iii), Example (ii) also satisfies our assumptions provided that the parameter \(\theta_{0,1,2}\) is assumed known and only \(\theta_{0,1,1}\) is estimated. The general case
$$b_m(x,\mu)=\int_{\R}\tilde{b}_m(x,y)d\mu(y),$$
where $\tilde{b}_m$ is Lipschitz continuous in both variables, is also covered by Assumption \ref{as: lip}, as explained in \cite{DelHof}.
}

\begin{assumption} \label{as: identif}
(Identifiability) For almost every $(x, \mu) \in\R \times \mathcal{P}_2(\R)$, for $\theta^1$, $\theta^2\in\Theta$
$$\langle \theta^1,\,b(x,\mu)\rangle = \langle \theta^2,\,b(x,\mu)\rangle \text{ implies }\theta^1=\theta^2.$$

\end{assumption}

While the identifiability condition assumed above is quite standard in our context (see, for example, A5 in \cite{Amo23} for the case \( H = \frac{1}{2} \)), verifying this condition in practice can be challenging. We refer to Section 2.4 of \cite{Hof2}, where the authors provide a detailed analysis of the identifiability condition. Specifically, for estimating the drift from continuous observations in the case of interacting particle systems driven by standard Brownian motion, they identify explicit criteria that ensure both identifiability and the non-degeneracy of the Fisher information matrix.

Furthermore, to establish our main results, we must assume that the denominator in the estim-actor does not diverge in probability, as formalized in the following assumption.

\begin{assumption} \label{as: denominator}
Assume that $\mathbb{P} \left( \int_0^T b_m^2(\bar{X}_s, \bar{\mu}_s) ds > 0 \right) > 0$ for all $m=1,\dots , p$, where $\bar{X}_s$ is the solution of the McKean-Vlasov equation \eqref{eq: McK}, and $\bar{\mu}_s$ is its law. 
\end{assumption}

{\change Note that this assumption is very mild. For instance, under the continuity conditions on \( b \), it is satisfied as long as the functions \( b_m \) are not identically zero in a neighborhood of the initial conditions.}

\begin{remark}
    Let us briefly discuss the  estimation question of the parameters in the diffusion term, $\sigma$ and $H$. In order to obtain these parameters, we do not need to use the techniques related to the propagation of chaos, since one can construct consistent and asymptotically normal estimators of $\sigma$ and $H$ from high-frequency observations of a single trajectory for a fixed $N$ following the quadratic variations ansatz. This is in analogy with the estimation of $\sigma$ and $H$ from observations over a bounded interval in the classical SDE case (see, for example, \cite{HuNuaZho SISP} and \cite{GaiImkSheTud}), and, indeed, the same method can be used for interacting particle. The crucial idea is that the drift term
    $$Y^{i,N}:=\left(\int_0^t \langle \theta^0,\, b(X^{i,N}_s, \mu^N_s)\rangle ds\right)_{t\in [0,\,T] }$$
    is sufficiently smooth and does not impact the limit of the empirical quadratic variations of the trajectories. More precisely, let us define empirical quadratic variations of order $2$ of a process $(R_t)_{t\in [0,\,T]} $ as
    $$V^M(R):= \sum_{i=2}^{M} \left(R_{\frac{iT}{M}}-2R_{\frac{(i-1)T}{M}}+R_{\frac{(i-2)T}{M}} \right)^2.$$
    Then under the Assumptions \ref{as: lip}, \ref{as: bound} and \ref{as: moments}, it is straightforward to show that
    $$M^{-\frac{1}{2}+2H}V^M(Y^{i,N})\to 0 $$
    in probability, and it follows by Minkowski inequality that $V^M(X^{i,N})$ behaves asymptotically like $V^M(\sigma B^{i,H})$, namely, one can obtain
    $$\sqrt{M}( M^{-1+2H} V^M(X^{i,N}) - c\sigma^2 T)\stackrel{d}{\to} \mathcal N(0, T\sigma^4)$$
    with a known, explicit constant $c$, see Corollary 3.4 in \cite{HuNuaZho SISP}. From this convergence result one can construct consistent and asymptotically normal estimators for $\sigma$ and $H$, and from a multivariate version for $(V^{M_k}(X^{i,N}))$, ${k=1,\dots , p}$, joint estimation of $(\sigma, H)$ is possible following the approach in \cite{IstLan} and \cite{Coeur}. Therefore, as we are considering continuous observations, $\sigma$ and $H$ can be assumed known in our manuscript.
\end{remark}

To establish our main results, we will primarily rely on Malliavin calculus and the propagation of chaos. A brief introduction to Malliavin calculus is provided in Section \ref{s: preliminaries}, while the specifics of the propagation of chaos are discussed in the subsequent subsection.

\subsection{Propagation of chaos}{\label{s: poc}}
The interacting particle system is naturally associated to its
mean field equation as $N \rightarrow \infty$. The latter is described by the 1-dimensional McKean-Vlasov SDE
\begin{equation}
d \bar{X}_t^{\theta} =  \sum_{m=1}^p \theta_m b_m \big(\bar{X}_t^{\theta}, \bar{\mu}_t^{\theta} \big) dt + \sigma d B_t^H,  \quad t\in [0, T],
\end{equation}
where $\bar{\mu}_t^{\theta}$ is the law of $\bar{X}_t^{\theta}$ and $(B_t^H)_{t \in [0, T]}$ is a fractional Brownian motion with Hurst index $H$, independent of the initial value having the law $\bar \mu_0^\theta := \mu_0$.
This equation is non-linear in the sense of McKean, see e.g.\ \cite{60imp,61imp,79imp}. It means, in particular, that the drift coefficient depends not only on the current state but also on the current distribution of the solution. {\rev If we now consider $N$ independent fractional Brownian motions $(B_t^{1,H})_{t \in [0,T]}, \dots, (B_t^{N,H})_{t \in [0,T]}$, each with the same Hurst index $H$, we can define a corresponding system of independent particles $(\bar{X}_t^{1,\theta}, \dots, \bar{X}_t^{N,\theta})$ satisfying
\[
d \bar{X}_t^{i,\theta} =  \sum_{m=1}^p \theta_m b_m \big(\bar{X}_t^{i,\theta}, \bar{\mu}_t^{\theta} \big) dt + \sigma d B_t^{i,H}, \quad i = 1, \dots, N.
\]
In this formulation, each particle $\bar{X}_t^{i,\theta}$ is driven by its own fractional Brownian motion $B_t^{i,H}$, and the collection $(\bar{X}_t^{i,\theta})_{i=1}^N$ consists of i.i.d.\ copies of the solution to the McKean–Vlasov equation. We denote by $\bar{X}_t^{i}$ the process corresponding to the true parameter value, that is, $\bar{X}_t^{i} := \bar{X}_t^{i,\theta_0}$.}

When \( H = \frac{1}{2} \), it is well known that, even with more general models and under appropriate assumptions on the coefficients, one can observe a phenomenon commonly referred to as \textit{propagation of chaos} (see, for example, \cite{79imp}). This means that the empirical measure \( \mu_t^{\theta, N} \) converges weakly to \( \bar{\mu}_t^{\theta} \) as \( N \rightarrow \infty \). Among the key works, we highlight \cite{poc, poc2}, where the propagation of chaos is proven in very general settings, including those with singular coefficients.

For \( H \neq \frac{1}{2} \), however, the literature is significantly more sparse, and additional assumptions on the coefficients are required, even to guarantee the well-posedness of the limiting McKean-Vlasov equation.
It is worth noting that \cite{CDFM20} demonstrates, by revisiting an old proof by Tanaka \cite{Tan84}, that the connection between interacting and independent particle systems holds for a wide variety of driving noises. Some of the ideas in \cite{CDFM20} build upon the work in \cite{BCD20}, which provides a robust solution theory for random rough SDEs of mean field type. The paper \cite{BCD21} is devoted to the propagation of chaos in the same setting.

In a context closer to ours, \cite{Galeati} addresses the case where the interacting particle system is driven by an additive fBm for $H \in (0,1)$, $H \neq \frac{1}{2}$. Although the main contribution of that paper is the proof of propagation of chaos in the presence of a singular drift, their result also applies to the case of Lipschitz drift (with the regular drift case detailed in Appendix A). Due to the complexity of their problem, the drift in their framework is assumed to be linear in the measure component, which contrasts with our setting. {\change A similar linear structure is also considered in \cite{HuRam}, where the authors apply a mimicking theorem approach to establish sharp convergence rates. In light of these differences, we provide a dedicated proof of propagation of chaos tailored to our specific model. }

\begin{theorem}{\label{th: poc}}
{\rev Let $H \in (0,1)$.} Assume Assumptions \ref{as: lip}, \ref{as: bound}, and \ref{as: moments} hold. Then, for each $q \ge 2$, there exists a constant $c > 0$ such that, for any $i \in \{1, ... , N \}$:
$$ \E \Big[ \sup_{t \in [0, T]}  |X_t^{i,N} - \bar{X}_t^{i}|^q \Big] \le c \Big( \frac{1}{\sqrt{N}}\Big), \qquad \E \Big[ \sup_{t \in [0, T]}  W_2^q(\mu_t^N, \bar{\mu}_t) \Big] \le c \Big( \frac{1}{\sqrt{N}}\Big).$$
\end{theorem}

It is important to highlight that Malliavin calculus for fractional Brownian motion will play a central role in establishing our main results. Consequently, we will frequently rely on the fact that the Malliavin derivative of an interacting particle is well approximated by that of the independent particle, similar to the propagation of chaos discussed earlier. This finding, which extends the propagation of chaos to Malliavin derivatives, is summarized in Proposition \ref{cor: PoC-deriv} below. To establish this result, it is essential to introduce an assumption concerning the derivatives of $b$. Observe that it involves the derivative with respect to the measure, known as the $L$-derivative, where the $L$ stands for "Lions". The role and properties of this derivative are explained in Section \ref{s: lions}.

\begin{assumption}[$f$] \label{as: deriv}
The function \( f(x, \mu) \) is continuously differentiable in \( x \), with \( \partial_x f(x, \mu) \) uniformly bounded in both \( x \) and \( \mu \). Furthermore, the map \( \mu \mapsto f(x, \mu) \) is \( \mathbb{P} \)-a.s. continuous in the topology induced by \( W_2 \), and \( \mathbb{P} \)-a.s. \( L \)-differentiable at every \( \mu \in \mathcal{P}_2(\mathbb{R}) \). The derivative \( \partial_\mu f(x, \mu)(v) \) has a \( \mathbb{P} \)-a.s. continuous and uniformly bounded version. Additionally, the functions \( \partial_x f \) and \( \partial_\mu f \) are Lipschitz continuous in both variables.
\end{assumption}

To clearly present the propagation of chaos for Malliavin derivatives, we begin by understanding how the Malliavin derivatives for both the independent and interacting particle systems depend on $\bar{X}^i$ and ${X}^{i,N}$. For any $i, j \in \{1, ... , N \}$ and \( s \leq t \), they satisfy
\begin{equation}{\label{eq: Ds barX}}
D_s^j \bar{X}^i_t = 1_{\{i=j\}}\left(\sigma  + \int_s^t \langle\theta^0,\, \partial_x b(\bar{X}^i_r, \bar{\mu}_r)\rangle D^i_s \bar{X}^i_r \, dr\right),  
\end{equation}
\begin{equation}{\label{eq: DsX}}
D_s^j {X}^{i,N}_t = 1_{\{i=j\}} \sigma  + \int_s^t  \Big( \langle\theta^0,\,\partial_x b({X}^{i,N}_r, {\mu}^N_r)\rangle D_s^j {X}^{i,N}_r \, + \frac{1}{N}\sum_{k=1}^N \langle\theta^0,\,(\partial_{\mu}b)({X}^{i,N}_r, \mu^N_r)(X^{k, N}_r)\rangle D_s^j X^k_r \Big) \, dr,    
\end{equation}
and for \( s > t \), these derivatives are zero.

For more details on Malliavin calculus for fractional Brownian motion, see Section \ref{s: malliavin}, and for explicit computations of the Malliavin derivatives above, refer to Point 1 of Lemma \ref{l: deriv-bds}.

Observe that, when addressing the problem of the convergence of the Malliavin derivatives of interacting particle systems towards those of the independent particles, the only result we are aware of is Proposition 4.8 in \cite{DosWil}. In that work, for particles driven by standard Brownian motions, the authors prove that \( D^j X_t^{i,N} \to D^j \bar{X}_t^i \) in \( L^2 \) as \( N \to \infty \). However, no quantitative estimates are provided.

We are now prepared to extend the propagation of chaos result, presented in Theorem \ref{th: poc}, to include the Malliavin derivatives in a quantitative manner. The detailed proof of this extension is provided in Section \ref{s: proof main}.
Our quantitative result applies also when the particles are driven by a standard Brownian motion.

\begin{theorem}{\label{cor: PoC-deriv}}
Let $H \in (0, 1)$. Under Assumptions \ref{as: lip}, \ref{as: bound}, \ref{as: moments}, and \ref{as: deriv}(b), there exists a constant \( c > 0 \) such that, for any \( i, j \in \{1, \dots, N \} \), the following holds, for any $s \in [0, T]$:
\begin{enumerate}
\item For each $q \ge 1$ $$\sup_{t \in [s, T]} |D_s^j {X}^{i,N}_t - D_s^j \bar{X}^i_t|^q = \sup_{t \in [s, T]} |D_s^j {X}^{i,N}_t |^q  \le c \left( \frac{1}{N} \right)^q \qquad \text{for} \ j \neq i.$$
    \item For each $q \ge 2$, $$ \E \Big[ \sup_{t \in [s, T]} |D_s^i {X}^{i,N}_t - D_s^i \bar{X}^i_t|^q \Big] \le c \left( \frac{1}{\sqrt{N}} \right).$$
\end{enumerate}
\end{theorem}

\begin{remark}{\label{rk: Malliavin}}
The following observations stem from the propagation of chaos for Malliavin derivatives outlined above:

Firstly, the almost sure bound in Point 1 for the off-diagonal elements of the Malliavin matrix is feasible due to the boundedness of the derivatives of \( b \). If this condition were to be relaxed, only \( L^p \) results would be achievable, and even those would require additional work.

Secondly, it is worth noting that in Point 1, we achieve a convergence rate of \( N^{-q} \), which surpasses the classical rate typically associated with propagation of chaos. {A reader familiar with Malliavin calculus might be reminded of the regularizing effect of Malliavin derivatives: while the $L^p$ bounds on the increments of a fractional SDE are governed by the noise, the Malliavin derivative eliminates the noise contribution and improves such bounds (see, e.g., Proposition 4.1 in \cite{HuNuaZho19}). Indeed, one can think that a similar effect could facilitate a faster convergence in the propagation of chaos than in the convergence of interacting particles to independent ones.} While Point 1 seems to support this perspective, Point 2 does not.

The slower convergence rate in Point 2, compared to Point 1, stems from the fact that the diagonal elements of the Malliavin matrix are estimated using the classical propagation of chaos result from Theorem \ref{th: poc}, which only provides the standard rate of \( N^{-\frac{1}{2}} \). This also explains why the result in Point 2 is restricted to \( q \ge 2 \). Nevertheless, the proof clearly shows that any improvement in the convergence rate of the original propagation of chaos would directly yield a corresponding improvement in the rate for the Malliavin derivatives in Point 2.

\end{remark}


Let us note one can easily verify that 
for any \( i,j \in \{1, \ldots, N \} \) and any \( t \in [s, T] \), we have
\begin{equation}{\label{eq: Df barX}}
D_s^j f(\bar{X}^i_t, \bar{\mu}_t) = 1_{\{i = j\}} \partial_x f(\bar{X}^i_t, \bar{\mu}_t) D_s^j \bar{X}^i_t,
\end{equation}
\begin{equation}{\label{eq: DFx}}
D_s^j f(X^{i,N}_t, \mu_t^N) = \partial_x f(X^{i,N}_t, \mu_t^N) D_s^j X^{i,N}_t + \frac{1}{N} \sum_{k = 1}^N (\partial_\mu f)(X^{i,N}_t, \mu_t^N)(X_t^{k,N}) D_s^j X^{k,N}_t.
\end{equation}
Then, from the propagation of chaos for Malliavin derivatives as in Theorem \ref{cor: PoC-deriv}, one can obtain the following corollary, whose proof is given in Section \ref{s: proof main}. We state these bounds here for simplicity, as they will be repeatedly used in the proofs of our main results.

\begin{corollary}{\label{cor: Df}}
Let $f: \R \times \mathcal{P}_2(\R)$ satisfy Assumption \ref{as: deriv}(f). Then, under Assumptions \ref{as: lip}, \ref{as: bound}, \ref{as: moments}, and \ref{as: deriv}(b), there exists a constant \( c > 0 \) such that, for any \( i, j \in \{1, \dots, N \} \), the following holds for any $s \in [0, T]$:
\begin{enumerate}
    \item For any $q \ge 1$, $$\sup_{t \in [s, T]} |D_s^j f(X^{i,N}_t, \mu_t^N) - D_s^j f(\bar{X}^i_t, \bar{\mu}_t)|^q \le c \left( \frac{1}{N} \right)^q \qquad \text{for} \ j \neq i.$$
    \item For any $q \ge 2$,  $$  \E \Big[ \sup_{t \in [s, T]} |D^i_s f({X}^{i,N}_t, \mu_t^N) - D^i_s f(\bar{X}^i_t, \bar{\mu}_t)|^q \Big] \le c \left( \frac{1}{\sqrt{N}} \right).$$
\end{enumerate}
\end{corollary}

\section{Main results}{\label{s: main}}

\subsection{Asymptotic properties of the fake-stimator $\tilde{\theta}_N$}{\label{s: fake}}

We begin by establishing the consistency of \(\tilde{\theta}_N\), as stated in the following theorem. The proof, detailed in Section \ref{s: proof main}, relies on Propositions \ref{prop: denominator} and \ref{prop: conv num} presented below.

\begin{theorem}{(Consistency)}{\label{th: consistency}}
Let $H \ge \frac{1}{2}$.
Assume that Assumptions \ref{as: lip}, \ref{as: bound}, \ref{as: moments}, \ref{as: identif}, \ref{as: denominator} and \ref{as: deriv}(b) hold. Then, $\tilde{\theta}_N$ defined in \eqref{eq: def estimator} is consistent in probability, meaning that
$$\tilde{\theta}_N \xrightarrow{\mathbb{P}} \theta^0 \quad \text{as} \quad N \rightarrow \infty.$$
\end{theorem}

\begin{remark}{\label{rk: H>1/4}}
Although technical, the consistency of the fake-stimator can also be established for \( H \in (\frac{1}{4}, \frac{1}{2}) \) by leveraging the linear operator \( K_H \), in a similar way as in \cite{HuNuaZho19}. The condition \( H > \frac{1}{4} \) is nevertheless a technical requirement necessary to demonstrate that the Skorohod integral appearing in the error term lies in \( L^2 \) (as proven in our case in Lemma \ref{l: L2-bound-loclip} below). This condition also facilitates the application of the dominated convergence theorem in Proposition \ref{prop: conv num}. Whether this result can be extended to any \( H \in (0, 1) \) remains an open question.
\end{remark}

The proof of the consistency stated above relies heavily on the explicit form of $\tilde{\theta}_N$, allowing us to express it as
\begin{equation}{\label{eq: est for consistency}}
\tilde{\theta}_N - \theta^0 = \Psi_N^{-1}\cdot{\sum_{i=1}^N \int_0^T b(X_t^{i,N}, \mu_t^N) \sigma dB_t^{i,H}}. 
\end{equation}
Due to the propagation of chaos, we can approximate the interacting particle system with independent particles, leading to the convergence in probability of $\Psi_N$ at a rate of $\frac{1}{N}$, as detailed in the following proposition.

\begin{proposition}{\label{prop: denominator}}
Assume that Assumptions \ref{as: lip}-\ref{as: moments} hold and that $g: \R \times \mathcal{P}_2(\R) \rightarrow \R$ is a locally Lipschitz function according to \eqref{eq: def loc lip}. Then, the following convergence in probability holds:
$$\frac{1}{N} \sum_{i = 1}^N \int_0^T g(X_s^{i,N}, \mu_s^N) ds \xrightarrow{\mathbb{P}} \int_0^T \E[g(\bar{X}_s, \bar{\mu}_s)] ds.$$
\end{proposition}

Thanks to Proposition \ref{prop: denominator}, proving consistency becomes straightforward once we show that 
$$\frac{1}{N} \sum_{i=1}^N \int_0^T b(X_t^{i,N}, \mu_t^N) \sigma  dB_t^{i,H}$$
converges to zero in probability, for which the propagation of chaos plays a crucial role. 
Specifically, let us define:
$$\sum_{i=1}^N \int_0^T b(X_t^{i,N}, \mu_t^N) \sigma  dB_t^{i,H} =: Z^N =: \sum_{i=1}^N Z^{i,N},$$
and similarly,
$$\sum_{i=1}^N \int_0^T b(\bar{X}_t^{i}, \bar{\mu}_t) \sigma dB_t^{i,H} =: \bar{Z}^N =: \sum_{i=1}^N \bar{Z}^{i,N}.$$
{\change Recall that the function  
\[
b: \R \times \mathcal{P}_2(\R) \rightarrow \R^p, \quad (x, \mu) \mapsto (b_1(x, \mu), \dots, b_p(x, \mu))
\]
defines a vector-valued drift. Accordingly, we denote \( Z^N = (Z_1^N, \dots, Z_p^N) \), where for each \( m \in \{1, \dots, p\} \),
\[
Z_m^N = \sum_{i=1}^N \int_0^T b_m(X_t^{i,N}, \mu_t^N)\, \sigma \, dB_t^{i,H}.
\]
An analogous notation applies to \( \bar{Z}^N \).
}

Our results will then follow by showing that both \( \frac{1}{N} \bar{Z}^N \) and \( \frac{1}{N}(Z^N - \bar{Z}^N) \) converge to 0 in \( L^2 \), and thus in probability, as \( N \to \infty \). The detailed proofs will be provided in Section \ref{s: proof main}.

\begin{proposition}{\label{prop: conv num}}
Let \( H \in [1/2, 1) \). Suppose that Assumptions \ref{as: lip}, \ref{as: bound}, \ref{as: moments}, \ref{as: deriv}(b) hold. Then  
$$\frac{1}{N} Z^N \xrightarrow{L^2} 0 \qquad \text{as } N \to \infty.$$
\end{proposition}

With Proposition \ref{prop: conv num}, we conclude the section dedicated to the consistency of our estim-actor. As we will see in the next subsection, proving a central limit theorem for the fake-stimator presents greater challenges. The process unfolds as an (almost) unending sequence of hope and despair, complete with numerous plot twists.

\subsubsection{The hunt for a central limit theorem}{\label{s: hunt}}

We now introduce some additional notation that will be useful in characterizing the variance of the limiting Gaussian distribution obtained in our main results. Specifically, let:  
\begin{align}{\label{eq: def vari}}
\Sigma^2 &:= \sigma ^2 \Big(\int_0^T \int_0^T \mathbb{E}[b_i(\bar{X}_s, \bar{\mu}_s) b_j(\bar{X}_t, \bar{\mu}_t)] \phi(s,t) \, ds \, dt \\
& \qquad + \int_0^T \int_0^T \int_0^T \int_0^T \mathbb{E}[D_v b_i(\bar{X}_s, \bar{\mu}_s) D_u b_j(\bar{X}_t, \bar{\mu}_t)] \phi(t,v) \phi(s,u) \, dv \, du \, ds \, dt\Big)_{i, j=1,\dots , p}, \nonumber 
\end{align}
\begin{equation}{\label{eq: def tilde vari}}
\tilde{\Sigma}^2 := \Psi^{-2}\Sigma^2, \text{ where}
\end{equation}
\begin{equation*}
\Psi=\left(\int_0^T \E [b_i(\bar{X}_s,\bar{\mu}_s)b_j(\bar{X}_s,\bar{\mu}_s)]ds\right)_{i,j=1,\dots , p},
\end{equation*}
with \(\phi(\cdot,\cdot)\) as introduced in \eqref{eq: def phi}. Additionally, since the particles are i.i.d., we simplify notation by writing \(\bar{X}\) for \(\bar{X}^1\) and \(D\bar{X}\) for \(D^1 \bar{X}^1\). Note that under Assumptions \ref{as: identif} and \ref{as: denominator} the matrix $\Psi$ is invertible. \\
\\
We begin our hunt by observing that, if \(\tilde{\theta}_N\) were constructed using independent particles, the central limit theorem for sums of independent and identically distributed random variables would directly ensure the asymptotic Gaussianity of the numerator. Moreover, the convergence in probability stated in Proposition \ref{prop: denominator} would naturally follow from the law of large numbers for i.i.d. random variables. In this scenario, Slutsky’s theorem could be applied straightforwardly, immediately establishing the asymptotic normality of the least squares fake-stimator constructed from independent particles.

In the case of interacting particle systems, however, the situation is more intricate. While Proposition \ref{prop: denominator} ensures that interactions have a negligible effect on the denominator, the same cannot be readily said for the numerator. Specifically, while we have  
\[ \frac{1}{\sqrt{N}} \bar{Z}^N \xrightarrow{\mathcal{L}} N(0, {\Sigma}^2), \]  
this is insufficient to conclude that  
\[ \frac{1}{\sqrt{N}} Z^N \xrightarrow{\mathcal{L}} N(0, {\Sigma}^2), \]  
{\change 
 as the propagation of chaos ensures that the error incurred when replacing \( \bar{Z}^N \) with \( Z^N \) vanishes as \( N \to \infty \), but not sufficiently fast for our purposes. As a result, the presence of interactions complicates the analysis of the asymptotic behavior of the numerator.

This naturally prompts the question of whether the convergence rate in the propagation of chaos can be improved. Recent work in \cite{Lac23} offers some optimism: the author shows that, in the Brownian setting and for interaction functions \( b \) that are linear in the measure component, an accelerated rate of convergence can indeed be achieved. This result has since been extended to general additive Gaussian noise in \cite{HuRam}.
 \\
\\
We were quite surprised (and excited!) to discover that the improved convergence rate of the off-diagonal Malliavin derivatives, established in Point 1 of Theorem \ref{cor: PoC-deriv}, is already sufficient to neutralize the effect of interactions in our setting.} As a result, \(\frac{1}{\sqrt{N}} Z^N\) converges to a zero-mean Gaussian with the same variance as that of the independent particles.  

Specifically, for any pair of random variables \((X, Y)\) with laws \((\mathcal{L}(X), \mathcal{L}(Y))\), we use the shorthand \(W_1(X, Y)\) to denote \(W_1(\mathcal{L}(X), \mathcal{L}(Y))\). Now, consider \(Z := \mathcal{N}(0, \operatorname{Id}_p)\). The following theorem holds. 

\begin{theorem}{\label{th: fluctuations}}
Let \( H \in (1/2, 1) \). Assume Assumptions \ref{as: lip}, \ref{as: bound}, \ref{as: moments}, \ref{as: deriv}(b) hold. Then, there exist constant $c> 0$ and $N_0 > 0$ such that, for any $N \ge N_0$, 
$$W_1\left(\frac{1}{\sqrt{N}}Z^N,\, \Sigma\cdot Z\right) \le c N^{-\frac{1}{8}},$$
with $\Sigma$ a square root matrix of $\Sigma^2$ from Equation \eqref{eq: def vari}.
\end{theorem}

As previously mentioned, the propagation of chaos for Malliavin derivatives enables us to establish the central limit theorem for the fake-stimator at the usual rate of \(\frac{1}{\sqrt{N}}\). In fact, the result in Theorem \ref{th: fluctuations} goes a step further by providing the convergence rate of the fluctuations. From the proof of Theorem \ref{th: fluctuations}, it becomes clear that the fluctuations comprise several terms. Some of these terms are shown to decay at a rate of \(\frac{1}{\sqrt{N}}\), leveraging the propagation of chaos for Malliavin derivatives. However, the slowest terms arise from the classical propagation of chaos, leading to a slower rate of \(N^{-\frac{1}{8}}\).  


From Theorem \ref{th: fluctuations}, a central limit theorem for \(\tilde{\theta}_N\) follows naturally as a direct application of Slutsky's theorem, combined with Proposition \ref{prop: denominator}. The detailed proof is deferred to Section \ref{s: proof main}.

\begin{theorem}{\label{th: CLT fakestimator}}
Let \( H \in (1/2, 1) \). Assume Assumptions \ref{as: lip}, \ref{as: bound}, \ref{as: moments}, \ref{as: identif}, \ref{as: denominator} and \ref{as: deriv}(b),  hold. Then,
\[
\sqrt{N}(\tilde{\theta}_N - \theta^0) \xrightarrow{\mathcal{L}} \mathcal{N}(0, \tilde{\Sigma}^2), \qquad N\rightarrow \infty
\]
with $\tilde{\Sigma}^2$ as in \eqref{eq: def tilde vari}.
\end{theorem}

\subsection{Alternative estimators}{\label{s: computable}}
As previously highlighted, the fake-stimator \(\tilde{\theta}_N\) introduced in \eqref{eq: def estimator} interprets the stochastic integral \(\int_0^T b(X_t^{i,N}, \mu_t^N) \sigma \, dB_t^{i,H}\) as a divergence-type (or Itô-Skorohod) integral, rendering it unobservable from the continuous paths of the particles. This observation motivated the introduction of the computable estimators \(\hat{\theta}_{N, \epsilon}\) and \(\hat{\theta}_N^{(fp)}\), as defined in \eqref{eq: def estimator increments} and \eqref{eq: def fp estimator}, respectively, along with the study of their asymptotic properties. Both estimators are derived using the relationship between the divergence integral and the Young integral, as established in \eqref{eq: link divergence Young} (see Proposition 5.2.3 of \cite{NuaBook}), which we recall below:
\begin{equation*}
\int_0^T b(X_t^{i,N}, \mu_t^N) \circ  dB_t^{i,H} = \int_0^T b(X_t^{i,N}, \mu_t^N) dB_t^{i,H} + \int_0^T \int_0^T D^i_s b(X_t^{i,N}, \mu_t^N) \phi(t,s) \, ds \, dt,
\end{equation*}
with \(\int_0^T b(X_t^{i,N}, \mu_t^N) \circ dB_t^{i,H}\) denoting the Stratonovich integral. Then, one can rewrite the fake-stimator in \eqref{eq: est for consistency} as
\begin{equation}{\label{eq: fake with strato}}
\tilde{\theta}_N = \Psi_N^{-1}\cdot{\sum_{i=1}^N \int_0^T b(X_t^{i,N}, \mu_t^N) \sigma \circ dX_t^{i,N}} - \Psi_N^{-1}\cdot \sum_{i=1}^N \int_0^T \int_0^T D^i_s b(X_t^{i,N}, \mu_t^N) \phi(t,s) \, ds \, dt.
\end{equation}

Thus, the favorable asymptotic results established for the estim-actor can be extended to practical, computable estimators, provided we can approximate \(D^i_s b(X_t^{i,N}, \mu_t^N)\).  

A first step toward addressing this question is presented in the following proposition, whose proof is detailed in Section \ref{s: proof main}.  

\begin{proposition}{\label{prop: approx Db}}  
Let \( H \geq \frac{1}{2} \). Assume that Assumptions \ref{as: lip}, \ref{as: bound}, \ref{as: moments}, and \ref{as: deriv}(b) hold. Then, there exist constants \(c, N_0 > 0\) such that, for any \(N \geq N_0\),  
\[
\big|D^i_s b(X_t^{i,N}, \mu_t^N) - \sigma \partial_x b(X_t^{i,N}, \mu_t^N) \exp\big(\int_s^t \langle \theta^0, \partial_x b(X_r^{i,N}, \mu_r^N) \rangle dr\big)1_{s \leq t} \big| \leq \frac{c}{N}.
\]  
\end{proposition}  

However, the proposition above does not fully resolve the issue, as the proposed approximation of the Malliavin derivative depends on \(\theta^0\), the unknown parameter we aim to estimate, making it inaccessible. Thus, an alternative approach to approximate the exponential term is required.  

A first idea involves the derivatives with respect to the diffusion's initial condition. A similar method, albeit without interaction, was employed in Proposition 2.7 of \cite{ComMar19}. As we will see, incorporating interaction introduces several additional challenges.

Before proceeding, let us introduce some notation. For any \( i \in \{1, \ldots, N \} \), let \( X^{i, x_0^i, N} \) denote a solution of \eqref{eq: model} with initial condition \( x_0^i \in \mathbb{R} \). We also denote by \( X^{k, x_0^i, N} \), for \( k \in \{1, \ldots, N \} \), the particle \( k \), where the particle \( i \) started at \( x_0^i \). In the following, we omit the dependence on \( N \) in \( X^{k, x_0^i, N} \) for simplicity.
For any \( t \in [0, T] \), \( i, j \in \{1, \ldots, N \} \), and \( x_0^j \in \mathbb{R} \), we have
\begin{align}{\label{eq: dx0 X 7}}
&\partial_{x_0^j} X^{i,x_0^j}_t = 1_{\{i=j\}} \nonumber \\
&+\int_0^t \left( \langle\theta^0,\,\partial_x b(X^{i,x_0^j}_r, \mu^N_r)\rangle \partial_{x_0^j} X^{i,x_0^j}_r + \frac{1}{N} \sum_{k=1}^N \langle\theta^0,\,(\partial_{\mu} b)(X^{i,x_0^j}_r, \mu^N_r)(X^{k, x_0^j}_r)\rangle \partial_{x_0^j} X^{k, x_0^j}_r \right) dr.
\end{align}
The formal derivation of this follows from the definition of the \( L \)-derivative in Section \ref{s: lions} and Proposition \ref{prop: deriv empirical measure}.

Comparing Equations \eqref{eq: dx0 X 7} and \eqref{eq: DsX} reveals striking similarities. These similarities suggest an approximation of the Malliavin derivative \( D_s^i X_t^{i,N} \) by the ratio $\frac{\partial_{x_0^i} X_t^{i,x_0^i}}{\partial_{x_0^i} X_s^{i,x_0^i}}\sigma ,$
for \( s < t \). We can show that, up to some technicalities involving a threshold in the denominator of this ratio to avoid potential blow-up, this is a good approximation, with an error of magnitude \( \frac{1}{N} \). However, when observing the particles, we do not have access to such a ratio, and it is therefore important to devise a practical estimator that does not rely on it. This motivates the idea of replacing the ratio with something that can be computed directly. 

As previously introduced in \eqref{eq: def estimator increments}, we now recall the definition of the estimator \(\hat{\theta}_{N, {\epsilon}}^{(2)}\): 
\begin{align*}
&\hat{\theta}_{N, {\epsilon}}^{(2)} := \Psi_N^{-1}{\sum_{i=1}^N \int_0^T b(X_t^{i,N}, \mu_t^N) \circ dX_t^{i,N}}\nonumber \\
&\quad- \Psi_N^{-1}{\sum_{i=1}^N \int_0^T \int_0^t \partial_x b(X_t^{i,N}, \mu_t^N) \frac{\frac{1}{\epsilon}(X_t^{i, x_0^i + \epsilon}- X_t^{i, x_0^i} )}{\frac{1}{\epsilon}(X_s^{i, x_0^i + \epsilon}- X_s^{i, x_0^i} ) \lor 1} \sigma ^2 \phi(t,s) ds \, dt}.
\end{align*}

It is interesting to note that, when there is no interaction, the ratio
\[
\frac{\frac{1}{\epsilon}(X_t^{i, x_0^i + \epsilon}- X_t^{i, x_0^i} )}{\frac{1}{\epsilon}(X_s^{i, x_0^i + \epsilon}- X_s^{i, x_0^i} )}
\]
approximates the ratio \( \frac{\partial_{x_0^i} X_t^{i,x_0^i}}{\partial_{x_0^i} X_s^{i,x_0^i} } \) very well, with their difference almost surely bounded by \( c \epsilon \). The impact of the interaction is not very significant, as we can show that \( \hat{\theta}_{N, {\epsilon}}^{(2)} \) approximates \( \tilde{\theta}_N \) in \( L^1 \), with an error of order \( \epsilon + \frac{1}{N} \). Therefore, choosing \( \epsilon = o(\frac{1}{\sqrt{N}}) \) ensures both the consistency and asymptotic normality of \( \hat{\theta}_{N, {\epsilon}}^{(2)} \), as formalized in Theorem \ref{th: estim increments} below, provided that the following two additional assumptions hold.

\begin{assumption}{\label{as: lb deriv drift}}
There exists a constant $M > 0$ such that, for all $(x, \mu) \in \R \times \mathcal{P}_2(\R)$, $|\partial_x b_m(x, \mu)| \ge M$ for all $m=1,\dots ,p$.
\end{assumption}

\begin{assumption}{\label{as: second derivatives}}
The second derivatives of \( b \) are uniformly bounded.
\end{assumption}

{\change 
We approximate the ratio in the denominator using the exponential term appearing in Proposition \ref{prop: approx Db}. Thanks to the lower bound in Assumption \ref{as: lb deriv drift}, this quantity is guaranteed to be bounded away from zero. Moreover, Assumption \ref{as: second derivatives} is required to control the second derivatives of \(X_s^{j, x_0^i}\), as detailed in Lemma \ref{l: x0} below.
}

\begin{theorem}{\label{th: estim increments}}
Let \( H \geq \frac{1}{2} \). Assume that Assumptions \ref{as: lip}, \ref{as: bound}, \ref{as: moments}, \ref{as: identif}, \ref{as: denominator}, \ref{as: deriv}(b), \ref{as: lb deriv drift}, and \ref{as: second derivatives} hold. Then, the estimator \( \hat{\theta}^{(2)}_{N, \varepsilon} \) satisfies the following properties:
\begin{enumerate}
    \item If \( \varepsilon = o(1) \), then \( \hat{\theta}^{(2)}_{N, \varepsilon} \) is consistent in probability:
    \[
    \hat{\theta}^{(2)}_{N, \varepsilon} \xrightarrow{\mathbb{P}} \theta^0 \quad \text{as} \quad N \to \infty.
    \]
    
    \item If \( \varepsilon = o\left( \frac{1}{\sqrt{N}} \right) \), then \( \hat{\theta}^{(2)}_{N, \varepsilon} \) is asymptotically normal:
    \[
    \sqrt{N} \left( \hat{\theta}^{(2)}_{N, \varepsilon} - \theta^0 \right) \xrightarrow{\mathcal{L}} \mathcal{N}(0, \tilde{\Sigma}^2),
    \]
    with \( \tilde{\Sigma}^2 \) as defined in \eqref{eq: def tilde vari}.
\end{enumerate} 
\end{theorem}

The proof of Theorem \ref{th: estim increments} is presented in Section \ref{s: proof main}. It builds on the asymptotic properties of the fake-stimator established earlier, combined with the fact that the error introduced in transitioning from \(\tilde{\theta}_N\) to \(\hat{\theta}^{(2)}_{N, \varepsilon}\) is negligible.  \\
\\
An alternative approach to making the exponential in Proposition \ref{prop: approx Db} more accessible involves employing a fixed-point argument, leading to the introduction of the estimator \(\hat{\theta}_N^{(fp)} = F_N(\hat{\theta}_N^{(fp)})\), as defined in \eqref{eq: def fp estimator}. To ensure the effectiveness of this estimator, it is crucial to demonstrate that \(F_N\) is a contraction. This, in turn, requires a precise bound on the constants appearing in our model (see Equation \eqref{eq: condition CT} below).  

In particular, a condition involving the inverse matrix of \(\Psi_N\) arises. For \(p=1\), it is straightforward to verify that assuming \(b\) is lower bounded away from zero guarantees an upper bound on \((\Psi_N)^{-1}\). However, for \(p > 1\), establishing similar assumptions on the model to bound \((\Psi_N)^{-1}\) is significantly more challenging. To avoid cumbersome computations involving the determinant, trace, and eigenvalues of this matrix, we restrict our analysis of the fixed-point estimator to the case \(p=1\). The extension to a multidimensional parameter setting, along with the related technical challenges, is left for future work.

\begin{assumption}{\label{as: fp}}  
There exists a constant \(l > 0\) such that, for all \((x, \mu) \in \R \times \mathcal{P}_2(\R)\), \(\partial_x b (x, \mu) \leq 0\) and \(|b(x, \mu)| \geq l\).  
\end{assumption}  

Under Assumption \ref{as: fp}, it follows that \((\frac{1}{N}\Psi_N)^{-1}\leq \frac{1}{l^2 T}\). To establish that \(F_N\) is a contraction, the following condition must also hold:  
\begin{equation}{\label{eq: condition CT}}  
C_T := \frac{\left\| \partial_x b \right\|_\infty^2}{l^2} \frac{(2H - 1)}{(2H + 1)} T^{2H} \sigma < 1.  
\end{equation}  
It is worth noting that we are operating in an asymptotic regime where \(N \to \infty\), but the time horizon \(T\) is fixed. Consequently, it suffices to focus on observations within a sufficiently small time frame. Specifically, we select  
\[
T \leq T_{\text{max}} := \left(\frac{l^2}{\sigma \left\| \partial_x b \right\|_\infty^2} \frac{(2H + 1)}{(2H - 1)}\right)^{\frac{1}{2H}},  
\]  
ensuring that \eqref{eq: condition CT} is satisfied. This leads to the following theorem, which summarizes the asymptotic properties of the fixed-point estimator. The proof can be found in Section \ref{s: proof main}.  

\begin{theorem}{\label{th: fp}}
Let \( H \geq \frac{1}{2} \). Assume that Assumptions \ref{as: lip}, \ref{as: bound}, \ref{as: moments}, \ref{as: identif}, \ref{as: denominator}, \ref{as: deriv}(b),  \ref{as: fp} hold. Assume moreover that $\Theta$ is a compact space in $\R^+$  and that $T \le T_{max}$. Then, the estimator \( \hat{\theta}^{(fp)}_{N} \) satisfies the following properties:
\begin{enumerate}
    \item It is consistent in probability:
    \[
    \hat{\theta}^{(fp)}_{N} \xrightarrow{\mathbb{P}} \theta^0 \quad \text{as} \quad N \to \infty.
    \]
    
    \item It is asymptotically normal:
    \[
    \sqrt{N} \left(\hat{\theta}^{(fp)}_{N} - \theta^0 \right) \xrightarrow{\mathcal{L}} \mathcal{N}(0, \bar{\Sigma}^2),
    \]
    with \( \bar{\Sigma}^2 = ( \frac{\Sigma}{\Psi - \tilde{V}})^2 \) and
    \begin{equation}{\label{eq: Vtilde 24.5}}
    \tilde{V}:= \int_0^T \int_0^t \int_s^t \E[\partial_x b(\bar{X}_t, \bar{\mu}_t)\partial_x b(\bar{X}_r, \bar{\mu}_r) \exp(\int_s^t \theta^0 \partial_x b(\bar{X}_r, \bar{\mu}_r)  dr) ] \phi(t,s) dr \, ds \, dt.    
    \end{equation}
\end{enumerate} 
\end{theorem}

\begin{remark}
Observe that Theorem \ref{th: fp} assures us that consistency and asymptotic Gaussianity can still be achieved for the fixed-point estimator, even though an approximation error is introduced. However, the cost of this error becomes evident in the second point of the theorem, which concerns the asymptotic Gaussianity of the estimator, specifically in its variance.  

Indeed, under our assumptions (\(\partial_x b \leq 0\)), it is straightforward to verify that \(\tilde{V}\) is positive. This implies that \({\Psi} - \tilde{V} \leq \Psi\), leading to the variance \(\bar{\Sigma}^2\) being larger than that of the estim-actor, \(\tilde{\Sigma}^2\), as established in Theorem \ref{th: CLT fakestimator}.  
\end{remark}

Although the fixed-point estimator performs well, as established in the main results above, it is often more practical to replace it with an iterative estimator that converges to the fixed-point solution. Specifically, we define the iterative estimator as follows:  
\[
\hat{\theta}^{(it)}_{N, n} := F_N (\hat{\theta}^{(it)}_{N, n-1}).  
\]  
The conditions ensuring that \(F_N\) is a contraction also guarantee the convergence of the iterative estimator to the fixed-point solution as \(n \to \infty\). Moreover, the favorable properties of the iterative estimator are summarized in the following corollary.

\begin{corollary}{\label{cor: it}}
Let \( H \geq \frac{1}{2} \). Assume that Assumptions \ref{as: lip}, \ref{as: bound}, \ref{as: moments}, \ref{as: identif}, \ref{as: denominator}, \ref{as: deriv}(b),  \ref{as: fp} hold. Assume moreover that $\Theta$ is a compact space in $\R^+$  and that $T \le T_{max}$. Then, the estimator \( \hat{\theta}^{(it)}_{N, n} \) satisfies the following properties:
\begin{enumerate}
    \item It is consistent in probability:
    \[
    \hat{\theta}^{(it)}_{N, n}  \xrightarrow{\mathbb{P}} \theta^0 \quad \text{as} \quad N, n \to \infty.
    \]
    
    \item Assume moreover that $\sqrt{N}(C_T)^n \rightarrow 0 $ for $N, n \to \infty$. Then, $\hat{\theta}^{(it)}_{N, n}$ is asymptotically normal:
    \[
    \sqrt{N} \left(\hat{\theta}^{(it)}_{N,n} - \theta^0 \right) \xrightarrow{\mathcal{L}} \mathcal{N}(0, \bar{\Sigma}^2) \quad \text{as} \quad N, n \to \infty,
    \]
    with \( \bar{\Sigma}^2\) as in Theorem \ref{th: fp}.
\end{enumerate}   
\end{corollary}
\begin{remark}
Note that since \(C_T < 1\), the assumption \(\sqrt{N}(C_T)^n \to 0\) is not restrictive. In fact, we can even allow \(n\) to grow logarithmically with \(N\). Specifically, if \(n = a \log N\) for some \(a > 0\), the constraint simplifies to \(C_T < 2a\). Consequently, a new time horizon \(\tilde{T}_{\text{max}}\) can be chosen to ensure that this condition is satisfied.  
\end{remark}

Before moving forward, let us highlight that in the next section, we will take a practical step further by considering a discretized version of the estimators proposed here, in the high-frequency setting. This is feasible because we are working with Stratonovich integrals, allowing us to apply the Riemann sum approximation to all quantities appearing in the estimators we have defined. Although it seems clear that the discretization error does not play a significant role, one should study in detail the implications of having only discrete observations. To maintain the article within bounds we have decided to restrict our analysis to the continuous case, leaving the case of discrete observations for future investigation. 

We note that when only discrete observations \( (X_{t_j}^{i,N})_{i \in \{1, \dots, N \}}^{ j \in \{1, \dots, n\}} \) are available, with \( 0 = t_0 < t_1 < \dots < t_n = T \) as the discretization grid and step size \( \Delta_n := \frac{T}{n} \), a natural estimator of the true parameter \( \theta^0 \) can be obtained considering

\begin{equation}{\label{eq: QnN}}
 Q_n^N (\theta) := \sum_{j=0}^{n-1} \sum_{i=1}^N \left[ \left( X_{t_{j+1}}^{i,N} - X_{t_j}^{i,N} - \Delta_n \tilde{b}(\theta, X_{t_j}^{i,N}, \mu_{t_j}^N) \right)^2 - \Delta_n^{2H} \right],   
\end{equation}
where $\tilde{b}$ is the drift function depending on $\theta$.

This quadratic statistic is inspired by the one studied in \cite{NeuTin} for drift parameter estimation with additive noise. Readers familiar with parameter estimation in the classical Brownian motion setting will recognize its similarity to the classical contrast function (e.g., \cite{Amo23}, \cite{Imp} or  \cite{Yoshida}). However, in our case, an additional correction of order \( \Delta_n^{2H} \) is required to remove the contribution of the second moments of the fractional Brownian motion.

From this contrast function \( Q_n^N(\theta) \), we can define the estimator for \( \theta^0 \) as:
\begin{equation}\label{eq: est discrete}
    \theta_n^N := \arg \min_{\theta \in \Theta} Q_n^N(\theta).
\end{equation}

Although we use this estimator in our simulations and demonstrate that it performs well numerically, its theoretical analysis remains an open question. It is important to note that our analysis heavily relies on the fact that the drift function is linear in the parameter vector \( \theta \), a condition not required when considering the contrast function \( Q_n^N(\theta) \). We believe this is another reason why the estimator \( \theta_n^N \) is both interesting and warrants further investigation.

\section{Numerical results}{\label{s: numerical}}

In this section we illustrate the performance of the estimators by evaluating them from simulated sample paths. Note that in practice, we need to consider discretised versions of the estimators, obtained by replacing the integrals by the corresponding Riemann Stieltjes type sums.\\

For all simulations, we utilise a Monte Carlo procedure with $100$ iterations to approximate the sample bias and RMSE. To simulate the trajectories on the interval $[0,1]$, we generate independent fractional Brownian motions, and then use the Euler scheme.\\

The first model we consider is a simple linear model, namely
\begin{equation}\label{eq: lin}
   X_t^{i,N} = \int_0^t \theta^0 \left(X^{i,N}_s-\frac{1}{N}\sum_{j=1}^N X^{j,N}_s\right)ds +B^{H,i}_t. 
\end{equation}
For this model, we investigate the behaviour of the ratio estimator for $\theta^0=5$ and $\theta^0=12$. To approximately compute the value of the ratio estimator considered in Theorem \ref{th: estim increments}, we run the Euler scheme and discretize the pathwise integrals with a mesh size $0.001$ and pick $\epsilon = 0.15$.\\

We also study the discrete estimator given by formula \eqref{eq: est discrete} in order to motivate future work on its asymptotic properties, and present a first empirical comparison of these estimators. To ensure that the setups are comparable in terms of observations, we set $n=1000$. For the discrete estimator, another important parameter is the mesh size for the points on which the function $Q^N_n$ is evaluated. Here, we pick this quantity as $0.05$. Also, since the minimum is determined on a discrete grid, it makes sense to speak about the exact recovery of the parameter if it happens to be on the grid. In a theoretical study, an approach for error quantification could be to increase the grid size with growing $N$.\\

\begin{table}[h!]
\centering
\begin{tabular}{|c|c|c|c|c|}
\hline
& $H = 0.6$ & $H = 0.6$ & $H = 0.8$ & $H = 0.8$ \\
\hline
&$N = 30$ & $N = 60$ & $N = 30$ & $N = 60$ \\
\hline
Ratio, $\theta^0=5$ & $0.099 \,(-0.008)$ & $0.018 \,( -0.015)$ & $0.129\,(-0.006)$ & $0.078 \,(0.005)$ \\
Discrete, $\theta^0=5$ & $0.062\,( -0.001)$ & $0.059\,(0.004)$ & $0.042\,(0.036)$ & $0.040\,(0.033)$ \\
\hline
Ratio, $\theta^0=12$ & $5 \cdot 10^{-5} \,( -6 \cdot 10^{-6})$ & $4 \cdot 10^{-5}\,(6 \cdot 10^{-6})$ & $6 \cdot 10^{-4}\,(9 \cdot 10^{-5})$ & $1 \cdot 10^{-4}\,(3 \cdot 10^{-5})$ \\
Discrete, $\theta^0=12$ & $0\,(0)$ & $0\,(0)$ & $0\,(0)$ & $0\,(0)$ \\
\hline
\end{tabular}\\
\caption{Simulation results for the model \eqref{eq: lin} in the format RMSE (Bias).}
\label{table:1}
\end{table}

As mentioned before, the discrete estimator can recover the correct value exactly. However, its performance, and also the speed of computation, by definition depends heavily on prior knowledge of bounds on the space of parameters $\Theta$.

The next model is defined as follows:
\begin{equation}\label{eq: arctan}
  X_t^{i,N} = \int_0^t \theta^0 \left(2- \arctan \left(X^{i,N}_s-\frac{1}{N}\sum_{j=1}^N X^{j,N}_s\right )\right )ds +B^{H,i}_t.   
\end{equation}

With this setup, the individual particles are pushed into linear growth with fluctuations that get stronger as the particle moves further away from the mean. For this drift, the fixed point estimator satisfies the conditions of Theorem \ref{th: fp}, and we can illustrate its behaviour in a simulation and compare to the ratio estimator. Here, we choose to run the Euler scheme and discretise the integrals with the step size $0.005$. To approximate the fixed point values, we pick $n=\lfloor \log (N)\rfloor$, however, in simulations we can see that the estimator is already close to the true value after the first iteration and improves only slightly after that.\\

\begin{table}[h!]
\centering
\begin{tabular}{|c|c|c|c|c|}
\hline
& $H = 0.6$ & $H = 0.6$ & $H = 0.8$ & $H = 0.8$ \\
\hline
&$N = 30$ & $N = 60$ & $N = 30$ & $N = 60$ \\
\hline
Ratio, $\theta^0=5$ & $ 1.910 \,(-0.017)$ & $ 0.919\,( -0.083)$ & $2.607\,(0.22)$ & $ 0.646\,(0.024)$ \\
Fixed point, $\theta^0=5$ & $ 0.091\,(-0.002)$ & $ 0.061\,( -0.007)$ & $0.09\,(0.002)$ & $ 0.063\,(0.006)$ \\
Discrete, $\theta^0=5$ & $ 0.339 \,(-0.067)$ & $ 0.247\,(0.004)$ & $0.106\,(-0.045)$ & $ 0.078\,(-0.05)$ \\
\hline
Ratio, $\theta^0=12$ & $ 0.623\,(-0.044)$ & $ 0.23\,(0.005)$ & $0.302\,(0.082)$ & $ 0.262\,(0.07)$ \\
Fixed point, $\theta^0=12$ & $ 0.093\,(-0.001)$ & $0.066 \,(-0.001 )$ & $0.085\,(-0.001)$ & $ 0.065\,(-0.003)$ \\
Discrete, $\theta^0=12$ & $ 0.284\,(-0.036)$ & $ 0.208\,( -0.059)$ & $0.11\,(-0.034)$ & $ 0.073\,(-0.024)$ \\
\hline
\end{tabular}\\
\caption{Simulation results for the model \eqref{eq: arctan} in the format RMSE (Bias). For $H=0.8$ the interval on which the fixed point estimator is considered is chosen as $[0, 0.79]$ in order to comply with the condition on $T_{max}$.}
\label{table:2}
\end{table}

The ratio estimator is significantly faster than the fixed-point estimator, as it requires one fewer integral approximation. However, in our simulations, the fixed-point estimator provides more precise results: in particular, the RMSE improves by a factor of 10 compared to the ratio estimator. The relatively poor performance of the ratio estimator can also be attributed to the fact that the drift in \eqref{eq: arctan} does not satisfy Assumption \ref{as: lb deriv drift}. However, the performance still improves with growing N, hinting at a possible relaxation of the hypotheses. Additionally, the fixed-point estimator outperforms the discrete estimator in terms of variance, even though the true values lie precisely on the grid over which $Q_n^N$ is minimised.\\

Finally, we illustrate the performance of the ratio estimator in the two-dimensional case given by the equation
\begin{equation}\label{eq: 2d param}
    X_t^{i,N} = \int_0^t \theta^0_1 \left(X^{i,N}_s-\frac{1}{N}\sum_{j=1}^N X^{j,N}_s\right) + \theta^0_2 X^{i,N}_s ds +B^{H,i}_t
\end{equation}
with the true parameters $\theta^0_1 = 2,\, \theta^0_2=11$. Also in this case, we can define the discrete estimator by minimising $Q_n^N$ with respect to both parameters. Like in the second model, we choose $0.005$ as discretisation step. The performance of both estimators is summarised in Table \ref{table:3}.\\

\begin{table}[h!]
\centering
\begin{tabular}{|c|c|c|c|c|}
\hline
& $H = 0.6$ & $H = 0.6$ & $H = 0.8$ & $H = 0.8$ \\
\hline
&$N = 30$ & $N = 60$ & $N = 30$ & $N = 60$ \\
\hline
Ratio, $\theta_1$ & $0.006\,( \sim 10^{-4})$ & $0.002\,(\sim 10^{-4})$ & $0.023\,(0.002)$ & $0.003\,(2\cdot 10^{-4})$ \\
Discrete, $\theta_1$ & $0.01\,(0.001)$ & $0.005\,(0.001)$ & $0.01\,(-0.001)$ & $0\,(0)$ \\
\hline
Ratio, $\theta_2$ & $0.006\,( \sim 10^{-4})$ & $0.002\,(\sim 10^{-4})$ & $0.023\,(-0.002)$ & $0.003\,(1\cdot 10^{-4})$ \\
Discrete, $\theta_2$ & $0.01\,(-0.001)$ & $0.005\,(-0.001)$ & $0.01\,(0.001)$ & $0\,(0)$ \\
\hline
\end{tabular}
\caption{Simulation results for the model \eqref{eq: 2d param} in the format RMSE (Bias).}
\label{table:3}
\end{table}

{\change
\section{Conclusion and future perspectives}
This section reflects on the insights gained throughout the paper and outlines directions for future research. As previously discussed, we studied the estimation of a drift parameter in an interacting particle system driven by fractional Brownian motion. Our main contributions include proving the consistency and asymptotic normality of the so-called fake-stimator, as well as that of computable estimators.

The complexity of the problem arises not only from the involvement of fBm, which is well known to pose analytical challenges, although widely studied in the literature, but also from the novelty of our statistical framework. In particular, establishing asymptotic normality in this context is rare. In the existing literature, such results are typically achieved for specific models like the Ornstein–Uhlenbeck process, where the estimator lies in the second Wiener chaos, allowing for the application of CLTs for multiple stochastic integrals.

In our case, the CLT for the fake-stimator is instead based on a new propagation of chaos result for Malliavin derivatives, which we establish here. Extending this result to the computable estimators is especially challenging due to particle interactions. Standard propagation of chaos is insufficient to show that these interactions can be neglected when approximating Malliavin derivatives. This motivated us to derive sharper estimates, see Proposition \ref{prop: approx Db}, which we believe have independent interest.

That said, as this is the first study of its kind, our main results come with some limitations, opening several directions for future investigation.

First, as already mentioned, it is crucial for applications to base inference on discrete-time, high-frequency observations. A natural way forward is to analyze the contrast function \( Q_{n}^N \), as introduced in \eqref{eq: QnN}. This would not only address the discrete observation setting but also allow for a generalization beyond linearity in the parameter, overcoming another limitation of our current framework.

Second, the efficiency of the proposed estimators remains an open question. While the fixed-point estimator exhibits higher variance, it is unclear whether the variances of the fake-stimator or the ratio-based estimator are optimal. The recent work \cite{Hof2} proves the LAN property for drift estimation in McKean–Vlasov equations with standard Brownian noise under continuous observation. The Fisher information matrix they obtain mirrors the classical case, with expectations taken with respect to the time-evolving law \(\bar{\mu}_t\) instead of a stationary distribution. Exploring whether a similar structure emerges in our setting—and whether the results of \cite{LAN Eul}, which prove the LAN property for SDEs with additive fractional noise, can be extended to our regime—is a compelling direction. Although their approach relies heavily on ergodicity and a Girsanov-type transformation, we believe that the tools developed here for analyzing Malliavin derivatives in the interacting setting could help overcome the challenges of our asymptotic framework.

Third, we aim to explore nonparametric approaches to drift estimation, offering greater modeling flexibility when the drift's functional form is unknown. We plan to build on the projection-based methodology proposed in \cite{CM}, developed for independent particles. However, extending it to the interacting case introduces new complexities, even in the limiting regime, especially in choosing appropriate functional spaces for \( b_m(x, \mu) \) and managing the measure dependence.

Lastly, all our results are restricted to the additive noise setting. Extending the analysis to the multiplicative noise case poses significant challenges, and currently remains an open problem. \\
\\
In conclusion, to the best of our knowledge, this work represents the first attempt at statistical inference for interacting particle systems driven by fBm. We hope it provides a useful starting point for statisticians interested in the rich and challenging structure of such models and that it offers insights and tools to inspire further research in this emerging area.
}

\section{Preliminaries}{\label{s: preliminaries}}

In this section, we present some preliminary concepts that will be useful throughout the paper. We begin with an introduction to fractional Brownian motion and Malliavin calculus for this process, with Section 5 of \cite{Nua} serving as the primary reference. Following this, we provide background on differentiating a function of a probability measure, focusing in particular on the \(L\)-derivative; the main reference here is Section 5.2 of \cite{CarDel18}. To conclude, we outline several technical results, whose proofs can be found in Section \ref{s: proof technical}, that are essential for establishing our main findings.

\subsection{Malliavin calculus for fractional Brownian motion}{\label{s: malliavin}}

Let us recall that the $N$-dimensional fractional Brownian motion (fBm), denoted by \( B = \{ (B_t^{1,H}, \dots, B_t^{N,H}), t \geq 0 \} \), with Hurst parameter \( H \in (0,1) \), is a zero-mean Gaussian process. Its components are independent and share the covariance function

\[
\mathbb{E}(B_t^{i,H} B_s^{i,H}) = R_H(t, s) := \frac{1}{2} \left(t^{2H} + s^{2H} - |t - s|^{2H}\right),
\]
for \( i = 1, \dots, N \), and for all \( s, t \geq 0 \). The probability space \( (\Omega, \mathcal{F}, \mathbb{P}) \) considered here is the canonical space of fBm. Specifically, \( \Omega = C_0(\mathbb{R}_+ ; \mathbb{R}^N) \) is the set of continuous functions from \( \mathbb{R}_+ \) to \( \mathbb{R}^N \), equipped with the uniform topology on compact intervals. \( \mathcal{F} \) is the Borel \( \sigma \)-algebra, and \( \mathbb{P} \) is the probability measure on \( (\Omega, \mathcal{F}) \), such that the process \( B_t^H(\omega) \) is an fBm with Hurst parameter \( H \in (0,1) \).


Next, we recall some background material on Malliavin calculus for the fBm \( B \). Let \( \mathcal{E}^N \) denote the set of \( \mathbb{R}^N \)-valued step functions on \( [0, T] \) with compact support. The Hilbert space \( \mathcal{H}^N \) is defined as the closure of \( \mathcal{E}^N \), endowed with the inner product

\[
\left\langle \left( \mathbf{1}_{[0, s_1]}, \dots, \mathbf{1}_{[0, s_N]} \right), \left( \mathbf{1}_{[0, t_1]}, \dots, \mathbf{1}_{[0, t_N]} \right) \right\rangle_{\mathcal{H}^N} = \mathbb{E} \left[ \left( \sum_{j=1}^N B_{s_j}^{j,H} \right) \left( \sum_{j=1}^N B_{t_j}^{j,H} \right) \right]
= \sum_{i=1}^N R_H(s_i, t_i),
\]
for every \( s_i, t_i \geq 0 \). The mapping \( \left( \mathbf{1}_{[0, t_1]}, \dots, \mathbf{1}_{[0, t_N]} \right) \mapsto \sum_{j=1}^N B_{t_j}^{j,H} \) can be extended to a linear isometry between \( \mathcal{H}^N \) and the Gaussian space spanned by \( B \). We denote this isometry by \( \varphi \in \mathcal{H}^N \mapsto B(\varphi) \). For \( N = 1 \), we simply write \( \mathcal{E} = \mathcal{E}^1 \) and \( \mathcal{H} = \mathcal{H}^1 \) (see \cite{Nua} for details).

When \( H = \frac{1}{2} \), \( B \) is just an \( N \)-dimensional Brownian motion, and \( \mathcal{H}^N = L^2([0, T]; \mathbb{R}^N) \). When \( H \in \left( \frac{1}{2}, 1 \right) \), let \( |\mathcal{H}|^N \) be the linear space of \( \mathbb{R}^N \)-valued measurable functions \( \varphi \) on \( [0, T] \) such that
\[
\|\varphi\|_{|\mathcal{H}|^N}^2 = \alpha_H \sum_{j=1}^N \int_{[0, T]^2} | \varphi_r^j| |\varphi_s^j| |r - s|^{2H - 2} \, \mathrm{d}r \, \mathrm{d}s < \infty,
\]
where \( \alpha_H = H(2H - 1) \). Then \( |\mathcal{H}|^N \) is a Banach space with the norm \( \|\cdot\|_{\mathcal{H}^N} \), and \( \mathcal{E}^N \) is dense in \( |\mathcal{H}|^N \). Furthermore, for any \( \varphi \in L^{1/H}([0, T]; \mathbb{R}^N) \), we have
\[
\|\varphi\|_{|\mathcal{H}|^N} \leq b_{H,d} \|\varphi\|_{L^{1/H}([0, T]; \mathbb{R}^N)},
\]
for some constant \( b_{H,d} > 0 \) (again, see \cite{Nua} for details). Thus, we have continuous embeddings \( L^{1/H}([0, T]; \mathbb{R}^N) \subset |\mathcal{H}|^N \subset \mathcal{H}^N \) for \( H > \frac{1}{2} \).

Next, we introduce the derivative operator and its adjoint, the divergence operator. Consider a smooth and cylindrical random variable of the form \( F = f\left(B^H_{t_1}, \ldots, B^H_{t_n}\right) \), where \( f \in C_b^{\infty}\left(\mathbb{R}^{N \times n}\right) \), meaning \( f \) and its partial derivatives are all bounded. We define its Malliavin derivative as the \( \mathcal{H}^N \)-valued random variable given by \( D F = \left( D^1 F, \ldots, D^N F \right) \), whose \( j \)-th component is

\[
D_s^j F = \sum_{i=1}^n \frac{\partial f}{\partial x_i^j}\left(B^H_{t_1}, \ldots, B^H_{t_n}\right) 1_{\left[0, t_i\right]}(s).
\]

By iteration, one can define higher-order derivatives \( D^{j_1, \ldots, j_i} F \) that take values in \( (\mathcal{H}^N)^{\otimes i} \). For any natural number \( p \) and any real number \( q \geq 1 \), we define the Sobolev space \( \mathbb{D}^{p,q} \) as the closure of the space of smooth and cylindrical random variables with respect to the norm \( \|\cdot\|_{p,q} \), given by

\[
\|F\|_{p,q}^q = \mathbb{E}[|F|^q] + \sum_{i=1}^p \mathbb{E} \left[ \left( \sum_{j_1, \ldots, j_i = 1}^N \|D^{j_1, \ldots, j_i} F\|^2_{(\mathcal{H}^N)^{\otimes i}} \right)^{\frac{q}{2}} \right].
\]

Similarly, if \( \mathbb{W} \) is a general Hilbert space, we can define the Sobolev space of \( \mathbb{W} \)-valued random variables, denoted \( \mathbb{D}^{p,q}(\mathbb{W}) \).

For \( j = 1, \ldots, N \), the adjoint of the Malliavin derivative operator \( D^j \), denoted by \( \delta^j \), is called the divergence operator or Skorohod integral. A random element \( u \) belongs to the domain of \( \delta^j \), denoted \( \operatorname{Dom}\left(\delta^j\right) \), if there exists a positive constant \( c_u \) depending only on \( u \) such that

\[
\mathbb{E}\left( \left\langle D^j F, u \right\rangle_{\mathcal{H}} \right) \leq c_u \|F\|_{L^2(\Omega)}
\]
for any \( F \in \mathbb{D}^{1,2} \). If \( u \in \operatorname{Dom}\left(\delta^j\right) \), the random variable \( \delta^j(u) \) is defined by the duality relationship

\[
\mathbb{E}\left( F \delta^j(u) \right) = \mathbb{E}\left( \left\langle D^j F, u \right\rangle_{\mathcal{H}} \right),
\]
for any \( F \in \mathbb{D}^{1,2} \). In a similar way, we define the divergence operator on \( \mathcal{H}^N \), where {\( \delta(u) = \sum_{j=1}^N \delta^j(u_j) \)}, with \( u_j \in \operatorname{Dom}\left(\delta^j\right) \) for all \( j = 1, \ldots, N \). We also use the notation {\( \delta(u) = \int_0^T u_t \, dB^H_t \)}, referring to {\( \delta(u) \)} as the divergence integral of \( u \) with respect to the fBm \( B^H \).

For \( p > 1 \), as a consequence of Meyer's inequality, the divergence operator \( \delta \) is continuous from \( \mathbb{D}^{1,p}(\mathcal{H}^N) \) into \( L^p(\Omega) \). This means
\begin{equation}{\label{eq: Meyer}}
\mathbb{E}\left[|\delta(u)|^p\right] \leq C_p \left( \mathbb{E}\left[\|u_\cdot\|_{\mathcal{H}^N}^p\right] + \mathbb{E}\left[\|D_\cdot u_\cdot\|_{\mathcal{H}^N \otimes \mathcal{H}^N}^p\right] \right),   
\end{equation}
for some constant \( C_p \) depending on \( p \).

\subsection{Differentiability of functions of probability measures}{\label{s: lions}}
We study a stochastic differential equation dependent on a measure, and to derive our result, we focus on the differentiability of the associated stochastic flow. This requires a concept of differentiation for functions defined on spaces of probability measures. The notion of differentiability we employ is one introduced by P.-L. Lions in his lectures at the Collège de France, as compiled by Cardaliaguet in his notes \cite{Car10}. The key ideas are thoroughly explained in \cite{CarDel18}, which serves as our primary reference. We specifically refer to Section 5.2 of \cite{CarDel18} for the preliminaries relevant to our discussion.

There are several notions of differentiability for functions defined on spaces of probability measures, and recent advances in optimal transport theory have highlighted some of their geometric characteristics (see Section 5.4 of \cite{CarDel18} for a review). However, our approach is more functional-analytic in nature rather than geometric, focusing on controlling infinitesimal perturbations of probability measures induced by small variations in a linear space of random variables. Consequently, our differentiation framework relies on the \textit{lifting} of functions $\mathcal{P}_2(\R) \ni \mu \mapsto u(\mu)$ to functions $\tilde{u}$ defined on a Hilbert space $L^2(\Omega, \mathcal{F}, \mathbb{P}; \R)$ over some probability space $(\Omega, \mathcal{F}, \mathbb{P})$, by setting $\tilde{u}(X) = u(\mathcal{L}(X))$, for $X \in L^2(\Omega, \mathcal{F}, \mathbb{P}; \R)$. Here, $\Omega$ is a Polish space, $\mathcal{F}$ its Borel $\sigma$-field, and $\mathbb{P}$ an atomless probability measure (and since $\Omega$ is Polish, $\mathbb{P}$ is atomless if and only if every singleton has measure zero).

In our analysis of the differentiability of probability measures, we frequently use the fact that over an atomless probability space $(\Omega, \mathcal{F}, \mathbb{P})$, for any probability distribution $\mu$ on a Polish space $E$, one can construct an $E$-valued random variable on $\Omega$ with $\mu$ as its distribution. For more details on the properties of the lifting $\tilde{u}$ over general spaces that are neither Polish nor atomless, we refer to Remark 5.26 in \cite{CarDel18}.

\begin{definition}{\label{def: L-diff}}
A function $u$ on $\mathcal{P}_2(\R)$ is said to be L-differentiable at $\mu_0 \in \mathcal{P}_2(\R)$ if there exists a random variable $X_0$ with law $\mu_0$ (i.e., $\mathcal{L}(X_0) = \mu_0$) such that the lifted function $\tilde{u}$ is Fréchet differentiable at $X_0$.
\end{definition}

The Fréchet derivative of $\tilde{u}$ at $X_0$ can be viewed as an element of $L^2(\Omega, \mathcal{F}, \mathbb{P}; \R)$ by identifying $L^2(\Omega, \mathcal{F}, \mathbb{P}; \R)$ with its dual. When studying this form of differentiation, the first task is to demonstrate that this notion is intrinsic, meaning the law of $D\tilde{u}(X_0)$ does not depend on the particular choice of the random variable $X_0$ satisfying $\mathcal{L}(X_0) = \mu_0$. This is ensured by Proposition 5.24 in \cite{CarDel18}, which we restate here for completeness:

\begin{proposition}{\label{prop: 5.24 CarDel18}}
Let $u$ be a real-valued function on $\mathcal{P}_2(\R)$ and $\tilde{u}$ its lifting to $L^2(\Omega, \mathcal{F}, \mathbb{P}; \R)$. If $u$ is L-differentiable at $\mu_0 \in \mathcal{P}_2(\R)$ in the sense of Definition \ref{def: L-diff}, then the lifting $\tilde{u}$ is differentiable at any $X \in L^2(\Omega, \mathcal{F}, \mathbb{P}; \R)$ such that $\mu_0 = \mathcal{L}(X)$, and the law of the pair $(X, D\tilde{u}(X))$ does not depend on the choice of $X$ as long as $\mu_0 = \mathcal{L}(X)$. 
\end{proposition}

The following result elucidates the structure of the \( L \)-derivative of a function defined on probability measures and provides the precise form in which it will be utilized. Consequently, we present its statement here; readers interested in the proof can refer to Proposition 5.25 in \cite{CarDel18}.

\begin{proposition}{\label{prop: 5.25 CarDel18}}
Let \( u \) be a real-valued, continuously \( L \)-differentiable function on \( \mathcal{P}_2(\mathbb{R}) \), and let \( \tilde{u} \) denote its lifting to \( L^2(\Omega, \mathcal{F}, \mathbb{P}; \mathbb{R}) \). Then, for any \( \mu \in \mathcal{P}_2(\mathbb{R}) \), there exists a measurable function \( \xi: \mathbb{R} \rightarrow \mathbb{R} \) such that for all \( X \in L^2(\Omega, \mathcal{F}, \mathbb{P}; \mathbb{R}) \) with \( \mathcal{L}(X) = \mu \), it holds that \( D \tilde{u}(X) = \xi(X) \) almost surely.
\end{proposition}

By stating that \( u \) is continuously \( L \)-differentiable, we mean that the Fréchet derivative \( D \tilde{u}(X) \) of the lifting \( \tilde{u} \) is a continuous function of \( X \), mapping from the space \( L^2(\Omega, \mathcal{F}, \mathbb{P}; \mathbb{R}) \) into itself. It is important to note that the function \( \xi \) defined in the proposition is uniquely determined \( \mu \)-almost everywhere on \( \mathbb{R} \) and satisfies the condition \( \int_{\mathbb{R}} |\xi(x)|^2 d\mu(x) < \infty \). Furthermore, both sides of the equation \( D\tilde{u}(X) = \xi(X) \) are evaluated at \( X \). However, the interpretation of these evaluations differs: \( D\tilde{u} \) is considered a mapping from \( L^2(\Omega, \mathcal{F}, \mathbb{P}; \mathbb{R}) \) evaluated at the random variable \( X \) (yielding another \( \mathbb{R} \)-valued random variable on \( (\Omega, \mathcal{F}, \mathbb{P}) \)), whereas \( \xi \) is a mapping from \( \mathbb{R} \) into itself, evaluated at each realization of the random variable \( X \). Therefore, for almost every \( \omega \in \Omega \), it holds that \( [D\tilde{u}(X)](\omega) = \xi(X(\omega)) \).

It is noteworthy that Proposition \ref{prop: 5.24 CarDel18} implies that the distribution of the \( L \)-derivative of \( u \) at \( \mu_0 \), when considered as a random variable, depends solely on the law \( \mu_0 \) and not on the specific random variable \( X_0 \) with law \( \mu_0 \). The Fréchet derivative \( [D \tilde{u}](X_0) \) is referred to as the representation of the \( L \)-derivative of \( u \) at \( \mu_0 \) along the variable \( X_0 \). Since it is considered an element of \( L^2(\Omega, \mathcal{F}, \mathbb{P}; \mathbb{R}) \), whenever \( X \) and \( X_0 \) are random variables with distributions \( \mu \) and \( \mu_0 \), respectively, we have the following expression:

\[
u(\mu) = u(\mu_0) + [D \tilde{u}](X_0) \cdot (X - X_0) + o(\|X - X_0\|_2),
\]
where \( \cdot \) denotes the \( L^2 \)-inner product in \( L^2(\Omega, \mathcal{F}, \mathbb{P}; \mathbb{R}) \), and \( \|\cdot\|_2 \) represents the associated norm.

Proposition \ref{prop: 5.25 CarDel18} indicates that, as a random variable, this Fréchet derivative takes the form \( \xi(X_0) \) for some deterministic function \( \xi: \mathbb{R} \rightarrow \mathbb{R} \), uniquely defined \( \mu_0 \)-almost everywhere on \( \mathbb{R} \). Given that the equivalence class of \( \xi \) in \( L^2(\mathbb{R}, \mu_0; \mathbb{R}) \) is uniquely determined, we can denote this Fréchet derivative as \( \partial_\mu u (\mu_0) \) (or simply \( \partial u (\mu_0) \) when there is no risk of confusion). We refer to \( \partial_\mu u (\mu_0) \) as the \( L \)-derivative of \( u \) at \( \mu_0 \) and identify it with the function \( \partial_\mu u(\mu_0)(\cdot): \mathbb{R} \ni x \mapsto \partial_\mu u(\mu_0)(x) \in \mathbb{R} \). With this notation, we can rewrite the earlier equation as 
\[
u(\mu) = u(\mu_0) + \mathbb{E}[\partial_\mu u(\mathcal{L}(X_0))(X_0) \cdot (X - X_0)] + o(\|X - X_0\|_2).
\]

This construction allows us to express \( [D \tilde{u}](X_0) \) as a function of any random variable \( X_0 \) with distribution \( \mu_0 \), irrespective of where this random variable is defined, thereby attributing meaning to the \( L \)-derivative of \( u \) at \( \mu_0 \) independently of the chosen lifting.

To gain a better understanding of this distinctive form of differentiation, we refer to Section 5.2.2 of \cite{CarDel18}, where the authors illustrate the behavior of such differentiation by computing the \( L \)-derivative in several fundamental examples. We will now turn our attention to the \( L \)-differentiability of functions of empirical measures, which is the primary focus of our preliminaries.

The somewhat intricate notion of differentiability for functions of empirical measures can be better grasped through the concept of empirical projection, defined as follows:

\begin{definition}{\label{def: empirical projection}}
Given a function \( u : \mathcal{P}_2(\mathbb{R}) \rightarrow \mathbb{R} \) and an integer \( N \geq 1 \), we define the empirical projection of \( u \) onto \( \mathbb{R} \) as:
\[
u^N: \mathbb{R}^N \ni (x_1, \dots, x_N) \mapsto u\left(\frac{1}{N} \sum_{i=1}^N \delta_{x_i}\right).
\]
\end{definition}

The following result establishes a connection between the \( L \)-derivative of a function on probability measures and the standard partial derivatives of its empirical projection. For the proof, the reader may refer to \cite{CarDel18}, where this result is presented as Proposition 5.35.

\begin{proposition}{\label{prop: deriv empirical measure}}
If \( u : \mathcal{P}_2 (\mathbb{R}) \rightarrow \mathbb{R} \) is continuously \( L \)-differentiable, then its empirical projection \( u^N \) is differentiable on \( \mathbb{R}^N \), and for all \( i \in \{1, \dots, N\} \),
\[
\partial_{x_i} u^N (x_1, \dots, x_N) = \frac{1}{N} \partial_\mu u\left(\frac{1}{N} \sum_{j=1}^N \delta_{x_j}\right)(x_i).
\]
\end{proposition}

This result will prove particularly useful when we compute the Malliavin derivative of the interacting particle system (see Point 1 of Lemma \ref{l: deriv-bds} below).

\subsection{Technical results}
In order to prove our main results, we need to ensure some technical bounds such as the ones gathered in next lemma.
\begin{lemma}{\label{l: moments}}
Grant Assumptions \ref{as: lip}-\ref{as: moments}. Then, for all $q \ge 2$, $0 \le s < t \le T$, $i \in \{ 1, ..., N \}$, $N \in \mathbb{N}$, the following hold true.
\begin{enumerate}
    \item $\sup_{t \in [0,T]} \E[|X_t^{ i, N}|^q] + \sup_{t \in [0,T]} \E[|\bar{X}_t^{ i}|^q] + \sup_{t \in [0,T]} \E[W_2^q(\mu_t^{N}, \delta_0)] < c$.
    \item 
    $\E[|X_t^{i, N} - X_s^{i, N}|^q] + \E[W_2^q(\mu_t^{N}, \mu_s^{N})]  \le c [(t - s)^{Hq}\land 1] $.
    \item  $\E[|\bar{X}_t^{i} - \bar{X}_s^{i}|^q] + \E[W_2^q(\bar{\mu}_t, \bar{\mu}_s)]  \le c [(t - s)^{Hq}\land 1]$.
\end{enumerate}
\end{lemma}

As previously mentioned, our analysis heavily depends on the use of Malliavin derivatives for both processes involved in our study: the interacting particle system and the independent particle system. The following lemma provides some identities characterising them, along with bounds that will prove useful in the subsequent sections. The detailed proof of this lemma can be found in Section \ref{s: proof technical}.


\begin{lemma}\label{l: deriv-bds}
    Grant Assumptions \ref{as: lip}, \ref{as: bound}, \ref{as: moments}, and \ref{as: deriv}(b). {\rev Then, the Malliavin derivatives of \(X^i\) and $\bar{X}^i$ exist, are unique, and square-integrable. Moreover,} for each $i, j=1,\dots ,N$, let $D^j X^{i, N}$ and $D^j\bar{X}^i$ denote the Malliavin derivatives of the solution processes $X^{i, N}$ and $\bar{X}^i$ with respect to $B^{H,j}$. The following hold:
    \begin{enumerate}
        \item For $s\leq t$:
        \begin{align}{\label{eq: Malliavin indep}}
           D_s^j\bar{X}^i_t &=1_{\{i=j\}}\left(\sigma +\int_s^t \langle\theta^0,\, \partial_x b(\bar{X}^i_r, \bar{\mu}_r)\rangle D^i_s\bar{X}^i_r dr\right),\\ 
           D_s^j{X}^{i,N}_t &=\sigma 1_{\{i=j\}}+\int_s^t \langle\theta^0,\,  \Big(\partial_x b({X}^{i,N}_r, {\mu}^N_r)D^{j}_s {X}^{i,N}_r + \frac{1}{N}\sum_{k=1}^N (\partial_{\mu}b)({X}^{i,N}_r, \mu^N_r)(X^{k,N}_r)D_s^{j} X^{k,N}_r \Big)\rangle dr, \nonumber
        \end{align}
        and for $s>t$, these derivatives are zero.
        \item Both \( D_s^i \bar{X}^i_t \) and \( D_s^i {X}^{i,N}_t \) are uniformly bounded, with the bound depending on \( T \) but independent of \( N \), \( t \), or \( s \). Moreover, for $j \neq i$, $|D_s^j X_t^{i,N}| \le \frac{c}{N}$.
        \item For $u\leq v$, $s\leq t$:
        \begin{align*}
            &|D^i_u\bar X^i_t-D^i_v\bar X^i_t| \leq c |u-v|,\\
            & |D^i_u\bar X^i_t-D^i_u\bar X^i_s| \leq c |t-s|,\\
            & |D^i_u\bar X^i_t-D^i_v\bar X^i_t- D^i_u\bar X^i_s + D^i_v\bar X^i_s| \leq c (1\wedge |u-v|)(1\wedge |t-s|).
        \end{align*}
   \item For any $i, j \in \{1, ... , N \}$, $u\leq v$, $s\leq t$:
        \begin{align*}
            &|D^j_u X^{i,N}_t-D^j_vX^{i, N}_t| \leq c |u-v|(1_{i = j} + \frac{1}{N}),\\
            & |D^j_u X^{i, N}_t-D^j_u X^{i,N}_s| \leq c |t-s|(1_{i = j} + \frac{1}{N}),\\
            & |D^j_u X^{i,N}_t-D^j_v X^{i,N}_t- D^j_u X^{i,N}_s + D^j_v X^{i,N}_s| \leq c (1\wedge |u-v|)(1\wedge |t-s|)(1_{i = j} + \frac{1}{N}).
        \end{align*}
    \end{enumerate}
\end{lemma}

\begin{remark}
It is worth noting that, in the case of a classical ergodic SDE $(X_t)_{t \in [0, T]}$ driven by a fractional Brownian motion as studied in \cite{HuNuaZho19}, one can establish the explicit bound \( |D_s X_t| \le c_1 e^{-c_2 |t - s|} \). For independent particles, a similar result could be obtained by modifying the drift with the addition of a restoring force—necessary in the classical diffusion scenario to achieve ergodicity. However, such a bound does not appear achievable for the Malliavin derivative of interacting particles, where instead we find \( |D^j_s X_t^{i,N}| \le c_1 e^{c_2 |t - s|} \). 

Importantly, this bound is not a requirement in our context, as we work with a fixed time horizon \( T \) and will only need to analyze the integrability of functions of Malliavin derivatives for \( |t - s| \) close to zero. Thus, for our purposes, the exponential term in our bound—as well as in \cite{HuNuaZho19}—effectively behaves simply as $1$.

\end{remark}



Before proceeding further, let us introduce a function \( g:\mathbb{R} \times \mathcal{P}_l(\mathbb{R}) \to \mathbb{R} \) and use the shorthand notation \( g_t \) to represent either \( g(\bar{X}^i_t, \bar{\mu}^i_t) \) or \( g({X}^{i,N}_t, {\mu}^{i,N}_t) \), for any \( i \in \{1, \dots, N \} \). Additionally, we will simply denote \( D^i \) by \( D \). Finally, let us recall that, according to the preliminaries on Malliavin calculus provided in Section \ref{s: malliavin}, we can express the following:  
        \begin{align}
           & \E [ \|g_\cdot \|^2_{\mathcal H}] = \int_0^T\int_0^T \E [g_s g_r]\phi(s,r)ds \,dr, \label{eq: g >}\\
           & \E \|D_\cdot g_\cdot\|^2_{\mathcal H\otimes \mathcal{H}} =\int_0^T\int_0^T\int_0^T\int_0^T \E [D_ug_s D_v g_t] \phi(u,v) \phi(s,t)du\,dv \,dt \,ds, \label{eq: Dg >}
        \end{align}
        with $\phi$ as in \eqref{eq: def phi}.
This recall will be useful for the following technical result, whose proof is provided in Section \ref{s: proof technical}. We require this result because, in the proof of consistency, we will frequently apply the dominated convergence theorem. The justification for this application is given in the following lemma.

\begin{lemma}\label{l: L2-bound-loclip}
{\change Grant Assumptions \ref{as: lip}, \ref{as: bound} and \ref{as: moments}.} Let 
$H \in \left(\frac{1}{2},1\right)$ and let \( g:\mathbb{R} \times \mathcal{P}_l(\mathbb{R}) \to \mathbb{R} \) be a locally Lipschitz function in the sense of \eqref{eq: def loc lip}, satisfying Assumption \ref{as: deriv}(g) {\change and $|g(0, \delta_0)| < \infty$}. Then, in the integrals \eqref{eq: g >} and \eqref{eq: Dg >}, 
the integrands are uniformly bounded in \( N \) by integrable functions. As a consequence, both \( \mathbb{E} \left[ \|g_\cdot \|^2_{\mathcal{H}} \right] \) and \( \mathbb{E} [\|D_\cdot g_\cdot\|^2_{\mathcal{H} \otimes \mathcal{H}}] \) are finite. 


\end{lemma}
\begin{corollary}\label{cor: L2-bound}
{\change Grant Assumptions \ref{as: lip}, \ref{as: bound}, \ref{as: moments}, and \ref{as: deriv}(b) and} let $H \ge \frac{1}{2}$. Then $\int_0^T b(\bar{X}^{i}_r, \bar{\mu}_r)\sigma dB^{H,i}_r$ as well as $\int_0^T b({X}^{i,N}_r, {\mu}^{N}_r)\sigma dB^{H,i}_r$ are well defined as $L^2$-random vectors. 
\end{corollary}
\begin{proof}
    The case $H\in (\frac{1}{2},1)$ is a direct consequence of Lemma \ref{l: L2-bound-loclip} and Theorem 2.5.5 in \cite{NPBook}. The case $H=\frac{1}{2}$ follows by It\^o isometry.
\end{proof}

To conclude this section, we present some bounds on the derivative with respect to the initial condition, the proof of which can be found in Section \ref{s: proof technical}. These bounds are crucial for establishing the asymptotic results of the estimator defined in \eqref{eq: def estimator increments}. Since the derivatives with respect to the initial conditions behave similarly to the Malliavin derivatives, the resulting bounds are to be expected.

\begin{lemma}{\label{l: x0}}
Grant Assumptions \ref{as: lip}, \ref{as: bound}, \ref{as: moments}, and \ref{as: deriv}(b). Then, for each $i,j \in \{1, ... , N \}$ and for any $t \in [0, T]$, the following hold true.
\begin{enumerate}
    \item $|\partial_{x_0^j} X_t^{i,x_0^j}| \le c(1_{\{i = j\}} + \frac{1}{N})$, 
    \item Assume moreover Assumption \ref{as: second derivatives}. Then, for any $\tau \in [0, 1]$, $|\partial^2_{x_0^j} X_t^{i,x_0^j + \tau}| \le c(1_{\{i = j\}} + \frac{1}{N})$.
\end{enumerate} 
\end{lemma}

\section{Proof main results}{\label{s: proof main}}
This section is devoted to proving the propagation of chaos for the particles and their derivatives, as established in Theorem \ref{th: poc} and Theorem \ref{cor: PoC-deriv}/Corollary \ref{cor: Df}, respectively, as well as the statistical main results stated in Section \ref{s: main}.


\subsection{About the propagation of chaos: proofs}
We begin by proving the propagation of chaos as described in Theorem \ref{th: poc}. It follows Sznitman’s direct approach \cite{79imp}, which is based on comparing the dynamics of interacting particles with their limiting counterparts. 
\subsubsection{Proof of Theorem \ref{th: poc}}
\begin{proof}
From Equations \eqref{eq: model} and \eqref{eq: McK}, we can write, for any \(i \in \{1, \dots, N\}\),  
\[
\E\left[\sup_{t \in [0, T]} |X_t^{i,N} - \bar{X}_t^i|^q\right] \leq c T^{q-1} \int_0^T \E\left[\|b(X_t^{i,N}, \mu_t^N) - b(\bar{X}_t^i, \bar{\mu}_t)\|^q\right] dt,
\]
which, by the Lipschitz continuity of \(b\) from Assumption \ref{as: lip}, leads to
\[
\leq c T^{q-1} \int_0^T \left(\E[|X_t^{i,N} - \bar{X}_t^i|^q] + \E[W_2^q(\mu_t^N, \bar{\mu}_t)]\right) dt.
\]
To apply Grönwall’s lemma, the first term is already in the desired form, but we need to further analyze the Wasserstein distance. To that end, we introduce the empirical measure over the independent particle system as:
\[
\bar{\mu}_t^N := \frac{1}{N} \sum_{j=1}^N \delta_{\bar{X}_t^j}.
\]
Clearly, we have:
\[
\E[W_2^q(\mu_t^N, \bar{\mu}_t)] \leq c \E[W_2^q(\mu_t^N, \bar{\mu}_t^N)] + c \E[W_2^q(\bar{\mu}_t^N, \bar{\mu}_t)].
\]
For the first term, since \(q \geq 2\), Jensen’s inequality gives:
\[
W_2^q(\mu_t^N, \bar{\mu}_t^N) \leq \left(\frac{1}{N} \sum_{j=1}^N |X_t^{j,N} - \bar{X}_t^j|^2\right)^\frac{q}{2} \leq \frac{1}{N} \sum_{j=1}^N |X_t^{j,N} - \bar{X}_t^j|^q.
\]
Thus, it follows that:
\[
\E[W_2^q(\mu_t^N, \bar{\mu}_t^N)] \leq \frac{1}{N} \sum_{j=1}^N \E[|X_t^{j,N} - \bar{X}_t^j|^q] = \E[|X_t^{i,N} - \bar{X}_t^i|^q]
\]
for any \(i \in \{1, \dots, N\}\), by symmetry. 

Next, we turn to the analysis of \(\E[W_2^q(\bar{\mu}_t^N, \bar{\mu}_t)]\). By Theorem 1 in \cite{FouGui} (for \(q = 2\) and \(d = 1\) in our case), we have:
\[
\E[W_2^q(\bar{\mu}_t^N, \bar{\mu}_t)] \leq c N^{-\frac{1}{2}} + c N^{-\frac{2-q}{2}} \leq c N^{-\frac{1}{2}},
\]
where the last inequality follows from \(q \geq 2\).

Combining everything, we obtain:
\begin{align*}
\E\left[\sup_{t \in [0, T]} |X_t^{i,N} - \bar{X}_t^i|^q\right] & \leq c \int_0^T \E\left[|X_t^{i,N} - \bar{X}_t^i|^q\right] dt + c N^{-\frac{1}{2}}, \\
& \leq c \int_0^T \E\left[ \sup_{s \le t}|X_s^{i,N} - \bar{X}_s^i|^q\right] dt + c N^{-\frac{1}{2}}
\end{align*}
which, after applying Grönwall’s lemma and incorporating \(T\) into the constant \(c\) (as we are working over a fixed time horizon), completes the proof. 

The proof of \( \E \left[ \sup_{t \in [0, T]} W_2^q(\mu_t^N, \bar{\mu}_t) \right] \leq c \left( \frac{1}{\sqrt{N}} \right) \) naturally follows as a direct consequence of the arguments presented earlier.

\end{proof}

The same approach used to prove the propagation of chaos earlier can be applied to establish the propagation of chaos for the Malliavin derivatives as well. Specifically, our method focuses on directly bounding the difference between the Malliavin derivatives of the interacting particle system and those of the independent particle system. The Malliavin derivatives for both types of particles are computed in detail in the proof of Point 1 of Lemma \ref{l: deriv-bds}.

\subsubsection{Proof of Theorem \ref{cor: PoC-deriv}}
\begin{proof}
Point 1. The proof will focus on the case $t > s$, as otherwise the Malliavin derivatives are simply zero. We begin by proving the almost surely propagation of chaos in the case where \( j \neq i \). From Point 1 of Lemma \ref{l: deriv-bds}, we know that \( D_s^j\bar{X}^i_t = 0 \) in this case, so that
\[
\sup_{t \in [s, T]} |D_s^j{X}^{i,N}_t - D_s^j\bar{X}^i_t| = \sup_{t \in [s, T]} |D_s^j{X}^{i,N}_t|.
\]
Then, the second point of Lemma \ref{l: deriv-bds} concludes the proof.

Point 2. Let us now consider the case where \( i = j \). From the dynamics of the Malliavin derivatives for both the independent and interacting particle systems, we have for any $t \in [s, T]$
\begin{align*}
D_s^i{X}^{i,N}_t - D_s^i\bar{X}^i_t & = \int_s^t \langle\theta^0,\,  \left( \partial_x b({X}^{i,N}_r, {\mu}^N_r) D_s^i {X}^{i,N}_r - \partial_x b(\bar{X}^{i}_r, \bar{\mu}_r) D_s^i \bar{X}^i_r \right)\rangle dr \\
& + \int_s^t ( \langle\theta^0,\, \frac{1}{N} \sum_{k = 1, k \neq i}^N (\partial_\mu b)({X}^{i,N}_r, \mu^N_r)(X^{k,N}_r) D_s^i X^k_r + \frac{1}{N} (\partial_\mu b)({X}^{i,N}_r, \mu^N_r)(X^{i,N}_r) D_s^i X^i_r)\rangle \, dr.
\end{align*}
{\rev Thanks to Point~2 of Lemma~\ref{l: deriv-bds}, we have 
\(|D_s^i X_r^k| \le \frac{c}{N}\) for \(k \neq i\) and \(|D_s^i X_r^i| \le c\).
Moreover, the term \(D_s^i X_r^i\) is multiplied by an additional factor \(\frac{1}{N}\), 
so that both contributions are of order \(\frac{1}{N}\).
Using the boundedness of \(\partial_\mu b\), we then deduce that the second integral above is bounded by 
\(c\,\frac{|t-s|}{N}\).
It follows that}
\[
|D_s^i{X}^{i,N}_t - D_s^i\bar{X}^i_t|^q \leq c \left( \int_s^t \langle\theta^0,\, \left( \partial_x b({X}^{i,N}_r, {\mu}^N_r) D_s^i {X}^{i,N}_r - \partial_x b(\bar{X}^{i}_r, \bar{\mu}_r) D_s^i \bar{X}^i_r \right)\rangle dr \right)^q + \frac{c T^q}{N^q},
\]
where the constant \( c \) is uniform in \( N \). This yields the inequality
\begin{equation}{\label{eq: A}}
\E[\sup_{t \in [s, T]}|D_s^i{X}^{i,N}_t - D_s^i\bar{X}^i_t|^q] \leq A + \frac{c}{N^q},
\end{equation}
where \( A = A_1 + A_2 \), and
\[
A_1 := c T^{q-1} \int_s^T \E[(\langle\theta^0,\, \partial_x b({X}^{i,N}_r, {\mu}^N_r) D_s^i {X}^{i,N}_r - \partial_x b(\bar{X}^{i}_r, \bar{\mu}_r) D_s^i {X}^{i,N}_r\rangle)^q] dr,
\]
\[
A_2 := c T^{q-1} \int_s^T \E[(\langle\theta^0,\, \partial_x b(\bar{X}^{i}_r, \bar{\mu}_r) D_s^i {X}^{i,N}_r - \partial_x b(\bar{X}^{i}_r, \bar{\mu}_r) D_s^i \bar{X}^i_r \rangle)^q] dr.
\]
For \( A_1 \), we apply the Lipschitz continuity of \( \partial_x b \), to obtain
\[
|\langle\theta^0,\, \partial_x b({X}^{i,N}_r, {\mu}^N_r) - \partial_x b(\bar{X}^{i}_r, \bar{\mu}_r)\rangle| |D_s^i {X}^{i,N}_r| \leq c (|{X}^{i,N}_r - \bar{X}^{i}_r| + W_2({\mu}^N_r, \bar{\mu}_r)) |D_s^i {X}^{i,N}_r|.
\]
The uniform boundedness of \( D_s^i {X}^{i,N}_r \), as shown in Point 2 of Lemma \ref{l: deriv-bds}, together with the propagation of chaos in Theorem \ref{th: poc}, implies
\begin{equation}{\label{eq: bound A1}}
A_1 \leq c \sup_{r\in [s,T]}\E[|{X}^{i,N}_r - \bar{X}^{i}_r|^q + W_2^q({\mu}^N_r, \bar{\mu}_r)] \leq c N^{-\frac{1}{2}}.
\end{equation}
Now let us consider \( A_2 \). Since \( \|\partial_x b \|_\infty < c \), this term is already in the right form to apply Gronwall's lemma:
\begin{equation}{\label{eq: bound A2}}
A_2 \leq c T^{q-1} \int_s^T \E[|D_s^i{X}^{i,N}_r - D_s^i \bar{X}^i_r|^q] dr.
\end{equation}
Combining \eqref{eq: A}, \eqref{eq: bound A1}, and \eqref{eq: bound A2}, we get
\begin{align*}
\E[\sup_{t \in [s, T]}|D_s^i{X}^{i,N}_t - D_s^i \bar{X}^i_t|^q] & \leq \frac{c}{N^q} + \frac{c}{N^{\frac{1}{2}}} + c T^{q-1} \int_s^T \E[|D_s^i{X}^{i,N}_r - D_s^i \bar{X}^i_r|^q] dr \\
& \leq \frac{c}{N^\frac{1}{2}} + c T^{q-1} \int_s^T \E[\sup_{u \in [s, r]}|D_s^i{X}^{i,N}_u - D_s^i \bar{X}^i_u|^q] dr.
\end{align*}
It concludes the proof via an application of Gronwall's lemma.

\end{proof}

From Theorem \ref{cor: PoC-deriv}, it is straightforward to derive Corollary \ref{cor: Df}. The proof relies on a formal computation of the Malliavin derivatives, combined with bounds similar in nature to those used in the previous proof.

\subsubsection{Proof of Corollary \ref{cor: Df}}
\begin{proof}
Point 1. From \eqref{eq: Df barX} we know that, for $j \neq i$, $D_s^j f (\bar{X}_t^i, \bar{\mu}_t) = 0$, so that the result follows from the boundedness of the derivatives of $f$, as in Assumption \ref{as: deriv}($f$), and Point 2 of Lemma \ref{l: deriv-bds}. 

Point 2. In the following, in order to lighten the notation, we will write simply $D_s$ for $D_s^i$. It follows clearly from Assumption \ref{as: deriv}($f$) and Point 2 of Lemma \ref{l: deriv-bds} that, for any \( t,s \in [0, T] \) and any \( q \geq 2 \),
\begin{align*}
&\mathbb{E}[\sup_{t \in [s, T]}|D_s f(X^{i,N}_t, \mu_t^N) - D_s f(\bar{X}^i_t, \bar{\mu}_t)|^q] \\
& \leq \mathbb{E}[\sup_{t \in [s, T]}|\partial_x f(X^{i,N}_t, \mu_t^N) D_s X^{i,N}_t - \partial_x f(\bar{X}^i_t, \bar{\mu}_t) D_s \bar{X}^i_t|^q] + \frac{c}{N^q} \\
& =: \tilde{A} + \frac{c}{N^q}.
\end{align*}
We recognize the same pattern as in \eqref{eq: A}, and thus, exactly as for \( A \), we split \( \tilde{A} \) into \( \tilde{A}_1 \) and \( \tilde{A}_2 \), where:
\[
\tilde{A}_1 \leq c \mathbb{E}[\sup_{t \in [s, T]}|\partial_x f(X^{i,N}_t, \mu_t^N) - \partial_x f(\bar{X}^i_t, \bar{\mu}_t)|^q | D_s X^{i,N}_t|^q] \leq c \frac{1}{N^\frac{1}{2}},
\]
due to the boundedness of \( D X^{i,N}_t \), the Lipschitz continuity of \( \partial_x f \), and the propagation of chaos from Theorem \ref{th: poc}. Moreover:
\[
\tilde{A}_2 \leq c \mathbb{E}[\sup_{t \in [s, T]}|D_s X^{i,N}_t - D_s \bar{X}^i_t|^q] \leq \frac{c}{N^\frac{1}{2}},
\]
utilizing the boundedness of \( \partial_x f \) and the propagation of chaos for Malliavin derivatives from Proposition \ref{cor: PoC-deriv}. Therefore, the proof is concluded.

\end{proof}


\subsection{About the asymptotic properties of the fake-stimator $\tilde{\theta}_N$: proofs}
We now proceed with the proof of the asymptotic properties of \(\tilde{\theta}_N\). Specifically, we begin by establishing the consistency stated in Theorem \ref{th: consistency}, along with the two supporting propositions upon which it is based.

\subsubsection{Proof of Theorem \ref{th: consistency}}
\begin{proof}
By \eqref{eq: est for consistency}, the consistency of $\tilde{\theta}_N$ follows directly from Propositions \ref{prop: denominator} and \ref{prop: conv num}, combined with the straightforward observation that convergence in \( L^2 \) implies convergence in probability.
\end{proof}

\subsubsection{Proof of Proposition \ref{prop: denominator}}
\begin{proof}
To establish the desired convergence, we first show that, as $N \rightarrow \infty$,
\[
\frac{1}{N} \sum_{i=1}^N \int_0^T g(X_s^{i,N}, \mu_s^N) \, ds - \frac{1}{N} \sum_{i=1}^N \int_0^T g(\bar{X}_s^i, \bar{\mu}_s) \, ds \xrightarrow{L^1} 0.
\]
From here, the proof follows by invoking the law of large numbers, that is applicable as we have finite second moments. It guarantees that:
\[
\frac{1}{N} \sum_{i=1}^N \int_0^T g(\bar{X}_s^i, \bar{\mu}_s) \, ds \rightarrow \int_0^T \E[g(\bar{X}_s^1, \bar{\mu}_s)] \, ds.
\]
Thus, it suffices to prove that:
\[
\int_0^T \E[|g(X_s^{i,N}, \mu_s^N) - g(\bar{X}_s^i, \bar{\mu}_s)|] \, ds \rightarrow 0,
\]
where \( i \) is fixed and the integral is taken over a bounded domain.
Since \( g \) is locally Lipschitz, we can apply the following bound:
\begin{align*}
&\E[|g(X_s^{i,N}, \mu_s^N) - g(\bar{X}_s^i, \bar{\mu}_s)|] \\
& \leq c\E\left[ \left( |X_s^{i,N} - \bar{X}_s^i| + W_2(\mu_s^N, \bar{\mu}_s) \right) \left( 1 + |X_s^{i,N}|^k + |\bar{X}_s^i|^k + W_2^l(\mu_s^N, \delta_0) + W_2^l(\bar{\mu}_s, \delta_0) \right)  \right].
\end{align*}
By using Cauchy-Schwarz inequality and the first point of Lemma \ref{l: moments}, we obtain:
\[
\E[|g(X_s^{i,N}, \mu_s^N) - g(\bar{X}_s^i, \bar{\mu}_s)|] \leq c \E[|X_s^{i,N} - \bar{X}_s^i|^2]^{1/2} + c \E[W_2^2(\mu_s^N, \bar{\mu}_s)]^{1/2}.
\]
The proof is then concluded by applying the propagation of chaos result in Theorem \ref{th: poc}.
\end{proof}

As demonstrated in the proof above, the convergence in probability of $\Psi_N$ follows from a fairly standard application of the propagation of chaos. The convergence in \(L^2\) of the numerator, gathered in Proposition \ref{prop: conv num}, relies heavily on \eqref{eq: Meyer}, which is a direct consequence of Meyer’s inequality.

\subsubsection{Proof of Proposition \ref{prop: conv num}}
\begin{proof}
  We show the statement coordinatewise. For $m=1,\dots ,p$ we write
    $$ \frac{1}{N}Z^N_m = \frac{1}{N}\bar{Z}^N_m + \frac{1}{N}(Z^N_m-\bar{Z}^N_m)=:S_1+S_2.$$
    As for $S_1$, we have due to $\bar Z^{i,N}$ being i.i.d.
    \begin{align*}
        \E[S_1^2] = \frac{1}{N}\E [(\bar Z^{1,N}_m)^2],
    \end{align*}
   since Skorokhod integrals are zero-mean. We know from Corollary \ref{cor: L2-bound} that $\E[ (\bar Z^{1,N}_m)^2]<\infty$, and hence, $ \E[S_1^2]\to 0$. For $S_2$ we need to evaluate

    \begin{align*}
        &\E \left(\frac{1}{N}\sum_{i=1}^N (Z^{i,N}_m-\bar Z^{i,N}_m)\right)^2 =\frac{1}{N^2} \sum_{i, j=1}^N \E [(Z^{i,N}_m-\bar Z^{i,N}_m)(Z^{j,N}_m-\bar Z^{j,N}_m)]\\
        &\quad \leq \frac{1}{N^2}\sum_{i, j=1}^N \|Z^{i,N}_m-\bar Z^{i,N}_m\|_{L^2(\Omega)}\|Z^{j,N}_m-\bar Z^{j,N}_m\|_{L^2(\Omega)} = \|Z^{1,N}_m-\bar Z^{1,N}_m\|^2_{L^2(\Omega)},
    \end{align*}
    because, thanks to the symmetry of the definition, all $Z^{i,N}_m-\bar Z^{i,N}_m$ are identically distributed.

    Thus, it suffices to show that
    $$\E \left(\int_0^T [b_m(X_t^{1,N}, \mu_t^N) - b_m(\bar{X}_t^1, \bar{\mu}_t)] \sigma dB^{1,H}_t\right)^2 $$
    tends to zero as $N\to \infty$. The case \( H = \frac{1}{2} \) is straightforward due to the Itô isometry, the Lipschitz continuity of \( b \) as stated in Assumption \ref{as: lip}, and the propagation of chaos from Theorem \ref{th: poc}. Let us now turn to the case where \( H > \frac{1}{2} \).
    For this we can use the inequality \eqref{eq: Meyer} for the one-dimensional integral $\delta^1$ and bound the two summands on its right hand side. We have
    \begin{align*}
        &\lim_{N\to\infty} \E [ \|b_m(X_t^{1,N}, \mu_t^N) - b_m(\bar{X}_t^1, \bar{\mu}_t) \|_{\mathcal H}^2]\\
        &\quad = \lim_{N\to\infty} \int_0^T\int_0^T \E [(b_m(X_u^{1,N}, \mu_u^N) - b_m(\bar{X}_u^1, \bar{\mu}_u))(b_m(X_v^{1,N}, \mu_v^N) - b_m(\bar{X}_v^1, \bar{\mu}_v))] \phi(u,v) du\,dv.
    \end{align*}
    The uniform integrability in the proof of Lemma \ref{l: L2-bound-loclip} allows us to use dominated convergence theorem and swap the limit and the integrals, and Assumption \ref{as: lip} together with the propagation of chaos in Theorem \ref{th: poc} yields convergence to zero. Similarly, for the second summand in the right hand side of \eqref{eq: Meyer} we have
        \begin{align*}
        &\lim_{N\to\infty} \E [ \|D_\cdot(b_m(X_t^{1,N}, \mu_t^N) - b_m(\bar{X}_t^1, \bar{\mu}_t)) \|_{\mathcal H\otimes \mathcal H}^2]\\
        &\quad = \lim_{N\to\infty} \int_0^T\int_0^T \int_0^T\int_0^T \E [D_u(b_m(X_s^{1,N}, \mu_s^N) - b_m(\bar{X}_s^1, \bar{\mu}_s))D_v(b_m(X_t^{1,N}, \mu_t^N) - b_m(\bar{X}_t^1, \bar{\mu}_t))] \\
        &\quad \quad \times \phi(u,v) \phi(t,s)du\,dv\,dt \,ds.
    \end{align*}
   Here we wrote $D$ for $D^1$. We can again use the uniform integrability in the proof of Lemma \ref{l: L2-bound-loclip} to bring the limit inside the integral. {Finally, we can write by Cauchy-Schwarz
   \begin{align*}
      &\E [D_u(b_m(X_s^{1,N}, \mu_s^N) - b_m(\bar{X}_s^1, \bar{\mu}_s))D_v(b_m(X_t^{1,N}, \mu_t^N) - b_m(\bar{X}_t^1, \bar{\mu}_t))] \\
      &\: \leq \sqrt{\E\left[\left|D_u(b_m(X_s^{1,N}, \mu_s^N) - b_m(\bar{X}_s^1, \bar{\mu}_s))\right|^2\right]}\sqrt{\E\left[\left|D_v(b_m(X_t^{1,N}, \mu_t^N) - b_m(\bar{X}_t^1, \bar{\mu}_t))\right|^2\right]}\leq c\frac{1}{\sqrt{N}}.
   \end{align*}
   The last bound, and with it the convergence to zero, follows by Corollary \ref{cor: Df}}.

\end{proof}

The three proofs presented above establish the first property of the fake-stimator \(\tilde{\theta}_N\), namely its consistency. We now turn our attention to proving its asymptotic gaussianity. The first step in this direction is to demonstrate that the numerator is asymptotically gaussian, as indicated by the fluctuation analysis detailed in Theorem \ref{th: fluctuations}.

\subsubsection{Proof of Theorem \ref{th: fluctuations}}
\begin{proof}
By the triangular inequality, we have  
\[
W_1\left(\frac{1}{\sqrt{N}} Z^N, Z\cdot \Sigma\right) \leq W_1\left(\frac{1}{\sqrt{N}} Z^N, \frac{1}{\sqrt{N}} \bar{Z}^N\right) + W_1\left(\frac{1}{\sqrt{N}} \bar{Z}^N, Z\cdot \Sigma\right),
\]  
where we recall that \(Z = \mathcal{N}(0,\operatorname{Id}_p)\). {\change It follows from the multivariate Berry–Esseen theorem for sums of i.i.d. random variables (see, e.g., \cite{KorShe} for a comprehensive historical review of the optimal constant in the Berry–Esseen inequality, and \cite{Rai}, Remark 1.6, for the multivariate extension) that}
\begin{equation}\label{eq: Berry Esseen}
W_1\left(\frac{1}{\sqrt{N}} \bar{Z}^N, Z\cdot \Sigma\right) \leq \frac{c}{\sqrt{N}}.
\end{equation}  
Introducing the notation \(F_N := \frac{1}{\sqrt{N}} Z^N\) and \(\bar{F}_N := \frac{1}{\sqrt{N}} \bar{Z}^N\), we now study the behavior of \(W_1(F_N, \bar{F}_N)\). From the definition of the Wasserstein distance, it is clear that  
\begin{equation}\label{eq: first bound W1}
W_1(F_N, \bar{F}_N) \leq \mathbb{E}[\|F_N - \bar{F}_N\|^2]^{\frac{1}{2}}.  
\end{equation}  
Next,  
\[
\mathbb{E}[\|F_N - \bar{F}_N\|^2] = \mathbb{E}[\|F_N\|^2] - 2 \mathbb{E}[\langle F_N,\, \bar{F}_N\rangle] + \mathbb{E}[\|\bar{F}_N\|^2].
\]  
We aim to show the following claims:  
\begin{equation}\label{eq: claim 1}
\mathbb{E}\bigl[\|F_N\|^2\bigr] \;=\; K \;+\; O\bigl(N^{-\tfrac14}\bigr).
\end{equation}  
\begin{equation}\label{eq: claim 2}
\mathbb{E}[\langle F_N,\, \bar{F}_N\rangle ] = K + O\bigl(N^{-\tfrac14}\bigr),    
\end{equation}  
\begin{equation}\label{eq: claim 3}
\mathbb{E}[\|\bar{F}_N\|^2] = K
\end{equation}  
where $K=\sum_{m=1}^p K_m$ with
\begin{align*}
    &K_m = \sigma ^2\int_0^T \int_0^T \mathbb{E}[b_m(\bar{X}_s^1, \bar{\mu}_s) b_m(\bar{X}_t^1, \bar{\mu}_t)] \phi(s,t) \, ds \, dt\\
    &\quad +\sigma ^2 \int_0^T \int_0^T \int_0^T \int_0^T \mathbb{E}\big[D^1_v b_m(\bar{X}_s^1, \bar{\mu}_s) D^1_u b_m(\bar{X}_t^1, \bar{\mu}_t)\big] \phi(v,s) \phi(u,t) \, dv \, du \, ds \, dt.
\end{align*}

The three claims together imply  
\[
\mathbb{E}[\|F_N - \bar{F}_N\|^2] = O\bigl(N^{-\tfrac14}\bigr).
\]  
Substituting this into \eqref{eq: first bound W1}, and combining it with \eqref{eq: Berry Esseen}, we obtain  
\[
W_1\left(\frac{1}{\sqrt{N}} Z^N, Z\cdot \Sigma\right) \leq c N^{-\frac{1}{8}} + c N^{-\frac{1}{2}} = c N^{-\frac{1}{8}},  
\]  
as desired.  

We now proceed to prove the claims \eqref{eq: claim 1}, \eqref{eq: claim 2}, and \eqref{eq: claim 3}. The approach for all of them is similar and heavily relies on a multidimensional isometry for divergence integrals, as described in Theorem 3.11.1 of \cite{Biagini}. Observe that this theorem is stated in \cite{Biagini} for the fractional Wick–Itô–Skorohod (fWIS) integral, which coincides with our divergence-type integral for \(H > \frac{1}{2}\) (see also Section 3.12 of \cite{Biagini} for further details).

Let us begin by proving \eqref{eq: claim 1}. Using the empirical projection as explained in Remark \ref{rk: projection}, we can write \(Z^N_m\) as \({\change \sigma}\int_0^T B^N_m({X}_s) \, d{B}_s^H\), where  
\begin{equation}{\label{eq: Bmn 40.5}}
B^N_m({x}) = (b^{1,N}_m({x}), \dots, b^{N,N}_m({x}))^T = \left(b_m(x^1, \mu^N), \dots, b_m(x^N, \mu^N)\right)^T, 
\end{equation}
and \({B}^H\) is an \(N\)-dimensional fractional Brownian motion. Consequently, as $\|Z^N\|^2 = \sum_{m = 1}^p \|Z_m^N\|^2$,
\[
\mathbb{E}[\|Z^N\|^2] = \mathbb{E}\left[\sum_{m=1}^p \left\|\int_0^T B_m^N({X}_s) \,  \sigma  d{B}_s^H\right\|^2\right].  
\]  
From the multidimensional isometry in Theorem 3.11.1 of \cite{Biagini}, we obtain  
\[
\mathbb{E}\left[\left\|\int_0^T B^N_m({X}_s) \sigma \, d{B}_s^H\right\|^2\right] =  \sigma ^2 \sum_{i = 1}^N \int_0^T \int_0^T \mathbb{E}[(B^N_m({X}_s))_i (B^N_m({X}_t))_i] \phi(s,t) \, ds \, dt 
\]  
\[
+  \sigma ^2\sum_{i,j = 1}^N \int_0^T \int_0^T \mathbb{E}\left[D_t^{\phi,i}(B^N_m({X}_s))_j \, D_s^{\phi,j}(B^N_m({X}_t))_i\right] ds \, dt,  
\]  
where we use the same notation as in \cite{Biagini}, and according to (3.31) therein,  
\[
D_t^{\phi, i} F := \int_0^T D^i_v F \, \phi(t,v) \, dv.
\]  
{\change By explicitly writing out the components of \( B_m^N \) as in \eqref{eq: Bmn 40.5}, we obtain the following:} 
\[
\mathbb{E}\left[\left\|\int_0^T B^N_m({X}_s)  \sigma \, d{B}_s^H\right\|^2\right] =  \sigma ^2\sum_{i = 1}^N \int_0^T \int_0^T \mathbb{E}[b_m(X_s^{i,N}, \mu_s^N) b_m(X_t^{i,N}, \mu_t^N)] \phi(s,t) \, ds \, dt  
\]  
\[
+  \sigma ^2\sum_{i = 1}^N \int_0^T \int_0^T \int_0^T \int_0^T \mathbb{E}[D^i_v b_m(X_s^{i,N}, \mu_s^N) D^i_u b_m(X_t^{i,N}, \mu_t^N)] \phi(v,s) \phi(u,t) \, dv \, du \, ds \, dt  
\]  
\[
+  \sigma ^2\sum_{i \neq j} \int_0^T \int_0^T \int_0^T \int_0^T \mathbb{E}[D^i_v b_m(X_s^{j,N}, \mu_s^N) D^j_u b_m(X_t^{i,N}, \mu_t^N)] \phi(v,s) \phi(u,t) \, dv \, du \, ds \, dt.  
\]  
We denote these three terms as \(I_1\), \(I_2\), and \(I_3\), respectively.  

Next, we study the convergence of \(\frac{1}{N}(I_1 + I_2 + I_3)\). By leveraging the propagation of chaos, we show that \(I_1\) and \(I_2\) contribute to the limit, while \(I_3\) is negligible. Following the proof of Proposition \ref{prop: denominator}, we first demonstrate that we can transition to independent particles using the propagation of chaos. Then, since the particles are identically distributed, the law of large numbers allows us to conclude the proof. Indeed,  
\begin{align*}
\frac{I_1}{N } &= \frac{1}{N}  \sigma ^2\sum_{i = 1}^N \int_0^T \int_0^T \mathbb{E}[(b_m(X_s^{i,N}, \mu_s^N) - b_m(\bar{X}_s^i, \bar{\mu}_s)) b_m(X_t^{i,N}, \mu_t^N)] \phi(s,t) \, ds \, dt \\
&\quad + \frac{1}{N} \sigma ^2\sum_{i = 1}^N \int_0^T \int_0^T \mathbb{E}[b_m(\bar{X}_s^i, \bar{\mu}_s)(b_m(X_t^{i,N}, \mu_t^N) - b_m(\bar{X}_t^i, \bar{\mu}_t))] \phi(s,t) \, ds \, dt \\
&\quad + \frac{1}{N} \sigma ^2 \sum_{i = 1}^N \int_0^T \int_0^T \mathbb{E}[b_m(\bar{X}_s^i, \bar{\mu}_s) b_m(\bar{X}_t^i, \bar{\mu}_t)] \phi(s,t) \, ds \, dt \\
&=: I_{1,1} + I_{1,2} + I_{1,3}.
\end{align*}
Let us start by analyzing \(I_{1,1}\) and \(I_{1,2}\). Using the Lipschitz continuity and linear growth of \(b\), along with the Cauchy-Schwarz inequality, we obtain the following bounds. By also applying Point 1 of Lemma \ref{l: moments} to control the moments of the processes \((X^{i,N}_t)_{t \in [0,T]}\) and \((\bar{X}^i_t)_{t \in [0,T]}\), we have  
\begin{align*}
|I_{1,1} + I_{1,2}| &\le \frac{c}{N} \sum_{i = 1}^N \int_0^T \int_0^T \big(\mathbb{E}[|X_s^{i,N} - \bar{X}_s^i|^2]^{\frac{1}{2}} + \mathbb{E}[W_2^2(\mu_s^N, \bar{\mu}_s)]^{\frac{1}{2}} \\
&\quad + \mathbb{E}[|X_t^{i,N} - \bar{X}_t^i|^2]^{\frac{1}{2}} + \mathbb{E}[W_2^2(\mu_t^N, \bar{\mu}_t)]^{\frac{1}{2}}\big) \phi(s,t) \, ds \, dt \le c{N}^{-\frac{1}{4}},
\end{align*}
as a direct consequence of the propagation of chaos in Theorem \ref{th: poc} and the integrability of \(\int_0^T \int_0^T \phi(s,t) \, ds \, dt\).  
Furthermore, using the fact that the particles are identically distributed, we obtain  
\[
\frac{I_1}{N} = I_{1,3} + O\left({N}^{-\frac{1}{4}}\right) = \sigma ^2\int_0^T \int_0^T \mathbb{E}[b_m(\bar{X}_s^1, \bar{\mu}_s) b_m(\bar{X}_t^1, \bar{\mu}_t)] \phi(s,t) \, ds \, dt + O\left({N}^{-\frac{1}{4}}\right).
\]

The analysis of \(I_2\) follows a similar approach. Specifically, we write:  
\begin{align*}
\frac{I_2}{N} &= \frac{\sigma ^2}{N} \sum_{i=1}^N \int_0^T \int_0^T \int_0^T \int_0^T \mathbb{E}\big[(D^i_v b_m(X_s^{i,N}, \mu_s^N) - D^i_v b_m(\bar{X}_s^i, \bar{\mu}_s)) D^i_u b_m(X_t^{i,N}, \mu_t^N)\big] \phi(v,s) \phi(u,t) \, dv \, du \, ds \, dt \\
&\quad + \frac{\sigma ^2}{N} \sum_{i=1}^N \int_0^T \int_0^T \int_0^T \int_0^T \mathbb{E}\big[D^i_v b_m(\bar{X}_s^i, \bar{\mu}_s) (D^i_u b_m(X_t^{i,N}, \mu_t^N) - D^i_u b_m(\bar{X}_t^i, \bar{\mu}_t))\big] \phi(v,s) \phi(u,t) \, dv \, du \, ds \, dt \\
&\quad + \frac{\sigma ^2}{N} \sum_{i=1}^N \int_0^T \int_0^T \int_0^T \int_0^T \mathbb{E}\big[D^i_v b_m(\bar{X}_s^i, \bar{\mu}_s) D^i_u b_m(\bar{X}_t^i, \bar{\mu}_t)\big] \phi(v,s) \phi(u,t) \, dv \, du \, ds \, dt \\
&=: I_{2,1} + I_{2,2} + I_{2,3}.
\end{align*}  
From the uniform boundedness of the Malliavin derivatives, as shown in Point 2 of Lemma \ref{l: deriv-bds}, and the propagation of chaos for Malliavin derivatives established in Corollary \ref{cor: Df}, we deduce:  
\[
|I_{2,1} + I_{2,2}| \le c{N}^{-\frac{1}{4}},  
\]  
where we also used the boundedness of the integrals.  
Finally, noting that the independent particles are identically distributed, we obtain:  
\[
\frac{I_2}{N} = \sigma ^2 \int_0^T \int_0^T \int_0^T \int_0^T \mathbb{E}\big[D^1_v b_m(\bar{X}_s^1, \bar{\mu}_s) D^1_u b_m(\bar{X}_t^1, \bar{\mu}_t)\big] \phi(v,s) \phi(u,t) \, dv \, du \, ds \, dt + O\left({N}^{-\frac{1}{4}}\right).
\]
The term \(\frac{I_3}{N}\) vanishes asymptotically. This follows directly from Corollary \ref{cor: Df}, which shows that \(|D^j_u b(X_t^{i,N}, \mu_t^N)| \le \frac{c}{N}\) for \(i \neq j\). Hence,  
\begin{align*}
\frac{I_3}{N} &\le \frac{\sigma ^2}{N} \sum_{i \neq j} \int_0^T \int_0^T \int_0^T \int_0^T \left(\frac{c}{N}\right)^2 \phi(v,s) \phi(u,t) \, dv \, du \, ds \, dt \le \frac{c}{N}.
\end{align*}
It follows that  
\begin{equation}{\label{eq: 34.5 end ZN}}
 \mathbb{E}[\|F_N\|^2] = \frac{1}{N} \mathbb{E}[\|Z^N\|^2] = \sum_{m=1}^p K_m + O\left({N}^{-\frac{1}{4}}\right) + O\left(\frac{1}{N}\right) = \sum_{m=1}^p K_m + O\left({N}^{-\frac{1}{4}}\right),   
\end{equation}
as claimed in \eqref{eq: claim 1}.

To prove \eqref{eq: claim 2}, we follow a similar route. Specifically, we express \(\bar{Z}^N_m\), $m=1,\dots ,p$, using the empirical projection:  
\[
\bar{Z}^N_m = \int_0^T \bar{B}_m^N({\bar{X}_s}) \sigma \, d{B}_s^H,
\]  
where \(\bar{B}^N_m({\bar{X}_s}) = \big(b_m(\bar{X}_s^1, \bar{\mu}_s), \dots, b_m(\bar{X}_s^N, \bar{\mu}_s)\big)^T\).  
Applying Theorem 3.11.1 in \cite{Biagini}, we obtain:  
\begin{align*}
\mathbb{E}[Z^N_m \bar{Z}^N_m] &= \mathbb{E}\Big[\Big\langle\int_0^T B^N_m({X}_s) \sigma \, d{B}_s^H,\, \int_0^T \bar{B}^N_m({\bar{X}_s}) \sigma \, d{B}_s^H\Big\rangle\Big] \\
&= \sigma ^2 \sum_{i=1}^N \int_0^T \int_0^T \mathbb{E}[b_m(X_s^{i,N}, \mu_s^N) b_m(\bar{X}_t^i, \bar{\mu}_t)] \phi(s,t) \, ds \, dt \\
&\quad + \sigma ^2\sum_{i, j=1}^N \int_0^T \int_0^T \int_0^T \int_0^T \mathbb{E}[D^i_v b_m(X_s^{j,N}, \mu_s^N) D^j_u b_m(\bar{X}_t^i, \bar{\mu}_t)] \phi(v,s) \phi(u,t) \, dv \, du \, ds \, dt.
\end{align*}  
Since \(D^j_u b_m(\bar{X}_t^i, \bar{\mu}_t) = 0\) for \(j \neq i\), the double sum simplifies to \(i = j\). Thus,  
\begin{align*}
\mathbb{E}[\langle Z^N,\, \bar{Z}^N\rangle ] &= \sigma ^2\sum_{m=1}^p \Big(\sum_{i=1}^N \int_0^T \int_0^T \mathbb{E}[b_m(X_s^{i,N}, \mu_s^N) b_m(\bar{X}_t^i, \bar{\mu}_t)] \phi(s,t) \, ds \, dt \\
&\quad + \sum_{i=1}^N \int_0^T \int_0^T \int_0^T \int_0^T \mathbb{E}[D^i_v b_m(X_s^{i,N}, \mu_s^N) D^i_u b_m(\bar{X}_t^i, \bar{\mu}_t)] \phi(v,s) \phi(u,t) \, dv \, du \, ds \, dt\Big).
\end{align*}  
By using the propagation of chaos, as outlined in Theorem \ref{th: poc} for the first term, and in Corollary \ref{cor: Df} for the second term, we conclude:  
\[
\mathbb{E}[\langle F_N ,\,\bar{F}_N\rangle ] = \frac{1}{N} \mathbb{E}[\langle Z^N,\, \bar{Z}^N\rangle ] = K + O\left({N}^{-\frac{1}{4}}\right),
\]  
as claimed in \eqref{eq: claim 2}.  

To finish, observe that:  
\[
\mathbb{E}[\|\bar{F}_N\|^2] = \frac{1}{N} \mathbb{E}[\|\bar{Z}^N\|^2] = K,
\]  
where we again applied Theorem 3.11.1 in \cite{Biagini} and \(D^j_u b_m(\bar{X}_t^i, \bar{\mu}_t) = 0\) for \(j \neq i\), as well as the fact that the independent particles are identically distributed.  
This concludes the proof of Theorem \ref{th: fluctuations}.

\end{proof}

The asymptotic gaussianity of the fake-stimator \(\tilde{\theta}_N\) follows directly from the theorem above, as detailed below.

\subsubsection{Proof of Theorem \ref{th: CLT fakestimator}}
\begin{proof}
From \eqref{eq: est for consistency}, we can express  
\[
\sqrt{N}(\tilde{\theta}_N - \theta^0) = (\frac{1}{N}\Psi_N)^{-1}\cdot{\frac{1}{\sqrt{N}} \sum_{i=1}^N \int_0^T b(X_t^{i,N}, \mu_t^N) \sigma \, dB_t^{i,H}}.
\]  
Theorem \ref{th: fluctuations} guarantees that the vector $\frac{1}{\sqrt{N}} \sum_{i=1}^N \int_0^T b(X_t^{i,N}, \mu_t^N) \sigma \, dB_t^{i,H}$ converges in law to \(\mathcal{N}(0, \Sigma^2)\), while Proposition \ref{prop: denominator} establishes that the matrix $\frac{1}{N}\Psi_N$ converges in probability to
$$\Psi = \left(\int_0^T \mathbb{E}[b_i(\bar{X}_t, \bar{\mu}_t)b_j(\bar{X}_t, \bar{\mu}_t)] \, dt\right)_{i,j=1,\dots, p}.$$
By applying Slutsky's theorem, we deduce that \(\sqrt{N}(\tilde{\theta}_N - \theta^0)\) converges in law to \(\mathcal{N}(0, \tilde{\Sigma}^2)\), where  
\[
\tilde{\Sigma}^2 = \Psi^{-2}\Sigma^2,
\]  
as desired.

\end{proof}

\subsection{About the alternative estimators: proofs}
This section is dedicated to showing that the computable estimators provide a reliable approximation of the estim-actor \(\tilde{\theta}_N\), which exhibits several favorable asymptotic properties, as demonstrated in the previous subsection. We begin by establishing a bound on the error introduced when approximating the Malliavin derivatives with the exponential, as described in Proposition \ref{prop: approx Db}.

\subsubsection{Proof of Proposition \ref{prop: approx Db}}
\begin{proof}
To begin, observe that by the definition of \(D^i_s b(X_t^{i,N}, \mu_t^N)\), we have  
\begin{align}{\label{eq: start approx Db}}  
& \big|D^i_s b(X_t^{i,N}, \mu_t^N) - \sigma \partial_x b(X_t^{i,N}, \mu_t^N) \exp\big(\int_s^t \langle \theta^0, \partial_x b(X_r^{i,N}, \mu_r^N) \rangle dr\big)1_{s \leq t} \big| \nonumber \\  
& \le \bigg|\frac{1}{N} \sum_{k = 1}^N (\partial_\mu b)(X^{i,N}_t, \mu_t^N)(X_t^{k,N}) D_s^i X^{k,N}_t \bigg| \nonumber \\
& \qquad \qquad + |\partial_x b(X^{i,N}_t, \mu_t^N)| \big| D_s^i X^{i,N}_t - \sigma \exp\big(\int_s^t \langle \theta^0, \partial_x b(X_r^{i,N}, \mu_r^N) \rangle dr\big)1_{s \leq t} \big| \nonumber \\  
& \le \frac{c}{N} + c \big| D_s^i X^{i,N}_t - \sigma \exp\big(\int_s^t \langle \theta^0, \partial_x b(X_r^{i,N}, \mu_r^N) \rangle dr\big)1_{s \leq t} \big|,  
\end{align}  
where we used the boundedness of the derivatives of \(b\) as stated in Assumption \ref{as: deriv}($b$) and the bound on the Malliavin derivatives from Point 2 of Lemma \ref{l: deriv-bds}. {\rev In particular, we have isolated in the sum the term corresponding to \( k = i \), so that the two bounds 
\( |D_s^i X_r^k| \le \frac{c}{N} \) for \( k \neq i \) and \( |D_s^i X_r^i| \le c \), 
together with the fact that the latter is multiplied by an additional factor \( \frac{1}{N} \), 
yield the bound \( \frac{c}{N} \) appearing above.} To bound the last term on the right-hand side of \eqref{eq: start approx Db}, let us introduce the process  
\begin{equation}{\label{eq: process Zst}}  
Z_{s,t}^i := \sigma  + \int_s^t \langle \theta^0, \partial_x b (X_r^{i,N}, \mu_r^N) \rangle Z_{s,r}^i \, dr.  
\end{equation}  
This equation has an explicit solution given by the exponential:  
\[
Z_{s,t}^i = \sigma \exp\left(\int_s^t \langle \theta^0, \partial_x b (X_r^{i,N}, \mu_r^N) \rangle \, dr \right).  
\]  
Observe that \(|D_s^i X_t^{i,N} - Z_{s,t}^i 1_{s \le t}| = 0\) for \(s > t\). For \(s \le t\), the dynamics of the Malliavin derivatives in \eqref{eq: DsX}, together with \eqref{eq: process Zst}, imply  
\begin{align*}  
|D_s^i X_t^{i,N} - Z_{s,t}^i| &\le \int_s^t |\partial_x b(X_r^{i,N}, \mu_r^N)| |D_s^i X_r^{i,N} - Z_{s,r}^i| \, dr \\  
& \quad + \int_s^t \frac{1}{N} \sum_{k = 1}^N |\partial_\mu b (X_r^{i,N}, \mu_r^N)(X_r^{k,N})| |D_s^i X_r^{k,N}| \, dr \\  
& \le c \int_s^t |D_s^i X_r^{i,N} - Z_{s,r}^i| \, dr + \frac{c}{N},  
\end{align*}  
where we employed the boundedness of the derivatives of \(b\) as given by Assumption \ref{as: deriv}($b$) and Point 2 of Lemma \ref{l: deriv-bds}. By Gronwall's lemma, this yields  
\begin{equation}{\label{eq: Ds - Z}}  
|D_s^i X_t^{i,N} - Z_{s,t}^i| \le \frac{c}{N}.  
\end{equation}  
Substituting this bound into \eqref{eq: start approx Db}, we obtain the desired result.  
\end{proof}

We now leverage the bound established in the proposition above to demonstrate that \(\hat{\theta}_{N, \epsilon}^{(2)}\) performs well asymptotically, as asserted in Theorem \ref{th: estim increments}.

\subsubsection{Proof of Theorem \ref{th: estim increments}}
\begin{proof}
From the definitions of \(\hat{\theta}_{N, \epsilon}^{(2)}\) in \eqref{eq: def estimator increments}, and using the process \(Z_{s,t}^i\) introduced in \eqref{eq: process Zst}, we can write, thanks to the approximation \eqref{eq: Ds - Z},  
\begin{align}{\label{eq: comparaison}}
&\tilde{\theta}_N - \hat{\theta}_{N, \epsilon}^{(2)} = o(\frac{1}{N}) \\
& + (\frac{1}{N}\Psi_N)^{-1}\cdot{\frac{1}{N} \sum_{i=1}^N \int_0^T \int_0^t \partial_x b(X_t^{i,N}, \mu_t^N) \left( \frac{\frac{1}{\epsilon}(X_t^{i, x_0^i + \epsilon} - X_t^{i, x_0^i})}{\frac{1}{\epsilon}(X_s^{i, x_0^i + \epsilon} - X_s^{i, x_0^i}) \lor 1}  \sigma - Z_{s,t}^i \right) \phi(t,s)  \sigma\, ds \, dt}. \nonumber
\end{align}
The factor $(\frac{1}{N}\Psi_N)^{-1}$ is of order \(1\), as it converges in probability to a constant matrix \(\Psi\), according to Proposition \ref{prop: denominator}. Let us now analyze remaining factor in the second term, which is bounded (in absolute value) by  
\[
\frac{c}{N} \sum_{i=1}^N \int_0^T \int_0^t \left| \frac{\frac{1}{\epsilon}(X_t^{i, x_0^i + \epsilon} - X_t^{i, x_0^i})}{\frac{1}{\epsilon}(X_s^{i, x_0^i + \epsilon} - X_s^{i, x_0^i}) \lor 1}  \sigma - Z_{s,t}^i \right| \phi(t,s) \, ds \, dt.
\]  
Define  
\begin{align*}
\left| \frac{\frac{1}{\epsilon}(X_t^{i, x_0^i + \epsilon} - X_t^{i, x_0^i}) \sigma}{\frac{1}{\epsilon}(X_s^{i, x_0^i + \epsilon} - X_s^{i, x_0^i}) \lor 1} - Z_{s,t}^i \right| & \le \left| \frac{\frac{1}{\epsilon}(X_t^{i, x_0^i + \epsilon} - X_t^{i, x_0^i}) \sigma}{\frac{1}{\epsilon}(X_s^{i, x_0^i + \epsilon} - X_s^{i, x_0^i}) \lor 1} - \frac{\partial_{x_0^i}X_t^{i, x_0^i} \sigma}{\partial_{x_0^i}X_s^{i, x_0^i} \lor \frac{1}{2}} \right| + \left| \frac{\partial_{x_0^i}X_t^{i, x_0^i} \sigma}{\partial_{x_0^i}X_s^{i, x_0^i} \lor \frac{1}{2}} - Z_{s,t}^i \right| \\
&=: E_1 + E_2.
\end{align*}
To bound \(E_2\), introduce the process  
\[
Y_t^i := 1 + \int_0^t \langle \theta^0,\,\partial_x b(X_r^{i,N}, \mu_r^N)\rangle Y_r^i \, dr.
\]  
This process has the explicit solution \(Y_t^i = \exp\left( \int_0^t  \langle \theta^0,\,\partial_x b(X_r^{i,N}, \mu_r^N)\rangle \, dr \right)\), so for \(s \le t\),  
\[
 \sigma\frac{Y_t^i}{Y_s^i} =  \sigma\exp\left( \int_s^t \partial_x b(X_r^{i,N}, \mu_r^N) \, dr \right) = Z_{s,t}^i.
\]  
Thus,  
\[
E_2 = \sigma \left| \frac{\partial_{x_0^i}X_t^{i, x_0^i}}{\partial_{x_0^i}X_s^{i, x_0^i} \lor \frac{1}{2}} - \frac{Y_t^i}{Y_s^i} \right| \le \frac{ \sigma\left| \partial_{x_0^i}X_t^{i, x_0^i} - Y_t^i \right|}{Y_s^i} + \frac{ \sigma\left| \partial_{x_0^i}X_t^{i, x_0^i} \right|}{\frac{1}{2} Y_s^i} \left| \partial_{x_0^i}X_s^{i, x_0^i} \lor \frac{1}{2} - Y_s^i \right|.  
\]  
From Assumption \ref{as: lb deriv drift}, we know \(Y_t^i \ge \exp(Mt) \ge 1\) for \(t \in [0, T]\), and from Lemma \ref{l: x0}, \(|\partial_{x_0^i}X_t^{i, x_0^i}| \le c\). Furthermore, acting as we did in order to prove \eqref{eq: Ds - Z}, it is straightforward to obtain
\begin{equation}{\label{eq: partialx0 - Y}}
|\partial_{x_0^i}X_t^{i, x_0^i} - Y_t^i| \le \frac{c}{N}.
\end{equation}
Additionally, \(Y_s^i \le \exp(cT) = c\) for \(s \in [0, T]\). Substituting these bounds, we find  
\[
E_2 \le \frac{c}{N} + c \cdot 1_{\{\partial_{x_0^i} X_s^{i, x_0^i} < \frac{1}{2}\}}.
\]  
Taking the expectation,  
\[
\mathbb{E}[E_2] \le \frac{c}{N} + \mathbb{P}(\partial_{x_0^i} X_s^{i, x_0^i} < \frac{1}{2}).
\]  
Using  
\begin{equation}{\label{eq: proba partialx0}}
\mathbb{P}(\partial_{x_0^i} X_s^{i, x_0^i} < \frac{1}{2}) = \mathbb{P}(Y_s^i - \partial_{x_0^i} X_s^{i, x_0^i} > Y_s^i - \frac{1}{2}) \le \mathbb{P}(Y_s^i - \partial_{x_0^i} X_s^{i, x_0^i} > \frac{1}{2}) \le \frac{c}{N},
\end{equation}  
where we used \(Y_s^i \ge 1\), Markov's inequality, and \eqref{eq: partialx0 - Y}, we conclude  
\[
\mathbb{E}[E_2] \le \frac{c}{N}.
\]  
Let us now proceed to study \(E_1\).
From a Taylor expansion, we directly obtain, for any \( t \in [0, T] \):  
\[
X_t^{i, x_0^i + \epsilon} - X_t^{i, x_0^i} = \epsilon \partial_{x_0^i} X_t^{i, x_0^i} + \epsilon^2 \partial^2_{x_0^i} X_t^{x_0^i + \tau \epsilon},
\]  
for some \( \tau \in [0, 1] \). Thus, using also Point 2 of Lemma \ref{l: x0}, we deduce:  
\[
\left| \frac{1}{\epsilon}(X_t^{i, x_0^i + \epsilon} - X_t^{i, x_0^i}) - \partial_{x_0^i} X_t^{i, x_0^i} \right| \le \epsilon |\partial^2_{x_0^i} X_t^{x_0^i + \tau \epsilon}| \le c \epsilon.
\]  
Now, observe that using the boundedness of \( |\partial_{x_0^i} X_t^{i, x_0^i}| \) from Point 1 of Lemma \ref{l: x0}, we have:  
\[
\begin{aligned}
E_1 & \le \frac{1}{2} \left| \frac{1}{\epsilon}(X_t^{i, x_0^i + \epsilon} - X_t^{i, x_0^i}) - \partial_{x_0^i} X_t^{i, x_0^i} \right|  
+ \frac{c}{2} \left| \frac{1}{\epsilon}(X_t^{i, x_0^i + \epsilon} - X_t^{i, x_0^i}) \lor 1 - \partial_{x_0^i} X_t^{i, x_0^i} \lor \frac{1}{2} \right| \\
& \le 2 c \epsilon + c 1_{\{\partial_{x_0^i} X_s^{i, x_0^i} < \frac{1}{2}\}} + \frac{c}{2} \left| \frac{1}{\epsilon}(X_t^{i, x_0^i + \epsilon} - X_t^{i, x_0^i}) \lor 1 - \partial_{x_0^i} X_t^{i, x_0^i} \right| 1_{\{\partial_{x_0^i} X_s^{i, x_0^i} \ge \frac{1}{2}\}} \\
& \le c \epsilon + c 1_{\{\partial_{x_0^i} X_s^{i, x_0^i} < \frac{1}{2}\}}  
+ c 1_{\{\partial_{x_0^i} X_s^{i, x_0^i} \ge \frac{1}{2}\} \cap \{\frac{1}{\epsilon}(X_t^{i, x_0^i + \epsilon} - X_t^{i, x_0^i}) < 1\}}.
\end{aligned}
\]  
Using Markov's inequality and Point 2 of Lemma \ref{l: x0}, we have:  
\[
\mathbb{P} \big( \{\partial_{x_0^i} X_s^{i, x_0^i} \ge \frac{1}{2}\} \cap \{\frac{1}{\epsilon}(X_t^{i, x_0^i + \epsilon} - X_t^{i, x_0^i}) < 1\} \big)  
\le \mathbb{P}(\epsilon |\partial^2_{x_0^i} X_t^{x_0^i + \tau \epsilon}| > \frac{1}{2}) \le 2 \epsilon \E[|\partial^2_{x_0^i} X_t^{x_0^i + \tau \epsilon}|] \le c \epsilon.
\]  
Together with \eqref{eq: proba partialx0}, this ensures:  
\[
\E[E_1] \le \frac{c}{N} + c \epsilon.
\]  
Combining these results, we conclude that the numerator of \eqref{eq: comparaison} is bounded in \( L^1 \) by \( c \epsilon + \frac{c}{N} \). Thus, requiring \( \epsilon = o(1) \) ensures that the numerator of \eqref{eq: comparaison} converges to \( 0 \) in \( L^1 \) and hence in probability. Since $(\frac{1}{N} \Psi_N)^{-1}$ converges in probability to a constant matrix $\Psi$, we deduce that \( \tilde{\theta}_N - \hat{\theta}_{N, \epsilon}^{(2)} \) converges to \( 0 \) in probability as well.  

Therefore, the consistency of \( \tilde{\theta}_N \), proven in Theorem \ref{th: consistency}, implies the consistency of \( \hat{\theta}_{N, \epsilon}^{(2)} \) as well. Similarly, under the assumption \( \epsilon = o\left(\frac{1}{\sqrt{N}}\right) \), it follows that \( \sqrt{N}(\tilde{\theta}_N - \hat{\theta}_{N, \epsilon}^{(2)}) \) converges to \( 0 \) in probability (and hence in law), which implies the asymptotic Gaussianity of \( \hat{\theta}_{N, \epsilon}^{(2)} \), thanks to Theorem \ref{th: CLT fakestimator}.  
The proof is therefore concluded.  
\end{proof}

Let us now proceed with the proof of the asymptotic properties of the fixed-point estimator, from which we will also derive the asymptotic properties of the iterative estimator.

\subsubsection{Proof of Theorem \ref{th: fp}}
\begin{proof}
Recall that we are considering the case \( p = 1 \) here. For the sake of consistency, we have chosen to retain the notation \( (\Psi_N)^{-1} \) for the denominator.

Observe that, according to \eqref{eq: fake with strato} and the definition of \( F_N \), we can write
\begin{align*}
 & \tilde{\theta}_N = F_N(\theta^0) \\
 &+ \Psi_N^{-1} \sum_{i = 1}^N \int_0^T \int_0^t \left( D^i_s b(X_t^{i,N}, \mu_t^N)- \sigma \partial_x b(X_t^{i,N}, \mu_t^N) \exp\left( \int_s^t \theta^0 \partial_x b(X_r^{i,N}, \mu_r^N)  \, dr \right) 1_{s \leq t} \right) \phi(t,s) \, ds \, dt,   
\end{align*}
which we denote by
\[
\tilde{\theta}_N = F_N(\theta^0) + R_N.
\]
On the other hand, from \eqref{eq: est for consistency} we have
\[
\theta^0 = \tilde{\theta}_N - \Psi_N^{-1} Z^N.
\]
Thus, from the definition of the fixed-point estimator, we get
\begin{equation}{\label{eq: start fp}}
\hat{\theta}_N^{(fp)} - \theta^0 = F_N(\hat{\theta}_N^{(fp)}) - F_N(\theta^0) - R_N + \Psi_N^{-1} Z^N.
\end{equation}
We now wish to analyze each term in the right hand side of the equation above. In particular, we have already thoroughly studied the estim-actor, so we know the asymptotic behavior of the last term $\Psi_N^{-1} Z^N$. Furthermore, Proposition \ref{prop: approx Db} implies that
\begin{equation}{\label{eq: bound RN 48.5}}
 |R_N| \leq \frac{c}{N}.   
\end{equation}
Therefore, to understand the behavior of \( \hat{\theta}_N^{(fp)} - \theta^0 \), we need to examine in detail the asymptotic behavior of \( F_N(\hat{\theta}_N^{(fp)}) - F_N(\theta^0) \).

Observe that  
\begin{equation}\label{eq: increment FN}
F_N(\hat{\theta}_N^{(fp)}) - F_N(\theta^0) = \Psi_N^{-1} \sum_{i = 1}^N \int_0^T \int_0^t \sigma \partial_x b(X_t^{i,N}, \mu_t^N) \big(f(\theta^0) - f(\hat{\theta}_N^{(fp)})\big) \phi(t,s) \, ds \, dt,
\end{equation}  
where \( f(\theta) = \exp\Big(\int_s^t \theta  \partial_x b(X_r^{i,N}, \mu_r^N) \, dr\Big) \).
Applying a Taylor expansion yields  
\[
f(\theta^0) - f(\hat{\theta}_N^{(fp)}) = f'(\theta^0)(\theta^0 - \hat{\theta}_N^{(fp)}) + f''(\bar{\theta})(\theta^0 - \hat{\theta}_N^{(fp)})^2,
\]  
where \(\bar{\theta}\) lies between \(\theta^0\) and \(\hat{\theta}_N^{(fp)}\). From this expansion, two terms emerge, for which we introduce the following notation:  
\begin{equation}\label{eq: FN 2}
F_N(\hat{\theta}_N^{(fp)}) - F_N(\theta^0) =: (\theta^0 - \hat{\theta}_N^{(fp)}) V_N + \tilde{R}_N,
\end{equation}  
where  
\[
V_N := \Psi_N^{-1} \sum_{i = 1}^N \int_0^T \int_0^t \int_s^t \sigma \partial_x b(X_t^{i,N}, \mu_t^N) \partial_x b(X_r^{i,N}, \mu_r^N) \exp\Big(\int_s^t \theta^0 \partial_x b(X_r^{i,N}, \mu_r^N) \, dr\Big) \phi(t,s) \, dr \, ds \, dt.
\] 
Under our hypotheses (in particular, \(\partial_x b \leq 0\) and \(\Theta\) compact in \((\mathbb{R}^+)^p\)), it is evident that \(V_N \geq 0\). Furthermore,  
\begin{align}\label{eq: bound VN}
V_N &\leq \|\partial_x b\|_\infty^2 \bigg(\frac{1}{N} \Psi_N\bigg)^{-1} \frac{\sigma}{N} \sum_{i = 1}^N \int_0^T \int_0^t |t-s| \phi(t,s) \, dt \, ds \nonumber \\
&\leq \frac{\|\partial_x b\|_\infty^2}{l^2 T} \frac{2H-1}{2H+1} T^{2H+1} = C_T,
\end{align}  
where we used the fact that the exponential term is bounded by \(1\) (since its argument is negative under our hypotheses) and that \(\big(\frac{1}{N} \Psi_N\big)^{-1} \leq \frac{1}{l^2 T}\). Recall that \(C_T < 1\) by assumption, ensuring that \(V_N \in [0, 1)\).  

Observe that additionally we can introduce $\tilde{V}_N$ such that \(V_N = : \psi_N^{-1} \tilde{V}_N\), where \(\big(\frac{1}{N}\psi_N\big)^{-1}\) is of order \(1\), as it converges in probability to a constant, {\change and \(\frac{1}{N} \tilde{V}_N\) converges in probability, by Proposition \ref{prop: denominator}, to \(\tilde{V}\) as defined in \eqref{eq: Vtilde 24.5}, which we recall here for convenience }
\[
\tilde{V} = \int_0^T \int_0^t \int_s^t \mathbb{E}\big[\partial_x b (\bar{X}_t, \bar{\mu}_t) \partial_x b (\bar{X}_r, \bar{\mu}_r) \exp\Big(\int_s^t \theta^0 \partial_x b (\bar{X}_r, \bar{\mu}_r) \, dr\Big)\big] \sigma \phi(t,s) \, dr \, ds \, dt,
\]  
as \(N \to \infty\). Therefore,  
\begin{equation}\label{eq: conv VN}
V_N \xrightarrow{\mathbb{P}} \Psi^{-1} \tilde{V} =: V,
\end{equation}  
as \(N \to \infty\). In particular, this implies that \(V_N\) is of order one.

We now replace \eqref{eq: FN 2} into \eqref{eq: start fp}, obtaining  
\[
(\hat{\theta}_N^{(fp)} - \theta^0)(1 - V_N) = \tilde{R}_N - R_N + \Psi_N^{-1} Z^N.
\]  
We claim that  
\begin{equation}\label{eq: claim end fp}
\tilde{R}_N = o_{L^1}\Big(\frac{1}{N}\Big).
\end{equation}  
Once this claim is proven, the proof is completed, as it would follow that  
\[
\hat{\theta}_N^{(fp)} - \theta^0 = o_{L^1}\Big(\frac{1}{N}\Big) + \frac{\Big(\frac{1}{N}\Psi_N\Big)^{-1}\frac{Z^N}{N}}{(1 - V_N)}.
\]  
{\change
The consistency of the estimator follows directly from the convergence of \(\frac{Z^N}{N}\) to zero, as stated in Proposition \ref{prop: conv num}, together with the convergence in probability of  
\[
\frac{\left( \frac{1}{N} \Psi_N \right)^{-1}}{1 - V_N}
\]  
to a constant (see Equation \eqref{eq: conv VN}).

For the asymptotic normality, recall from Theorem \ref{th: fluctuations} that  
\[
\left( \frac{1}{N} \Psi_N \right)^{-1} \frac{Z^N}{\sqrt{N}} \xrightarrow{d} \mathcal{N}(0, \tilde{\Sigma}^2),
\]  
so that Slutsky’s theorem yields
\[
\sqrt{N} \left( \hat{\theta}_N^{(fp)} - \theta^0 \right) \xrightarrow{d} \mathcal{N}\left(0, \left( \frac{\tilde{\Sigma}}{1 - V} \right)^2 \right).
\]
Recall that \(\tilde{\Sigma} = \frac{\Sigma}{\Psi}\), so that
\[
\frac{\tilde{\Sigma}}{1 - V} = \frac{\Sigma}{\Psi (1 - V)}.
\]
Using the identity \(V = \Psi^{-1} \tilde{V}\) from Equation \eqref{eq: conv VN}, we obtain
\[
\Psi (1 - V) = \Psi - \tilde{V},
\]
and thus the asymptotic variance simplifies to
\[
\left( \frac{\Sigma}{\Psi - \tilde{V}} \right)^2,
\]
as claimed.}

In order to conclude the proof, let us prove the claim in \eqref{eq: claim end fp}. To this end, we again employ \eqref{eq: start fp}. However, to bound \(F_N(\hat{\theta}_N^{(fp)}) - F_N(\theta^0)\), instead of using \eqref{eq: FN 2}, we apply the rough bound 
\begin{equation}{\label{eq: lip FN}}
 |F_N(\hat{\theta}_N^{(fp)}) - F_N(\theta^0)| \leq C_T |\theta^0 - \hat{\theta}_N^{(fp)}|,   
\end{equation}
which is derived using similar reasoning as in the proof of \eqref{eq: bound VN}. This implies  
\[
|\hat{\theta}_N^{(fp)} - \theta^0| \leq C_T |\hat{\theta}_N^{(fp)} - \theta^0| + |R_N| + |\Psi_N^{-1} Z^N|.
\]  
Rearranging and using the fact that \(C_T < 1\), we find  
\[
|\hat{\theta}_N^{(fp)} - \theta^0| \leq \frac{|R_N|}{1 - C_T} + \frac{1}{1 - C_T}|\Psi_N^{-1} Z^N|.
\]  
Next, we estimate the expectation of \(|\tilde{R}_N|\):  
\[
\mathbb{E}[|\tilde{R}_N|] \leq c \mathbb{E}[|\hat{\theta}_N^{(fp)} - \theta^0|^2].
\]  
Substituting the bound above, we get  
\[
\mathbb{E}[|\tilde{R}_N|] \leq c \mathbb{E}[|R_N|^2] + c \mathbb{E}[|\Psi_N^{-1} Z^N|^2].
\]  
Under our hypotheses, since \(|N \Psi_N^{-1}|\) is bounded by a constant and \(\mathbb{E}[|Z^N|^2] \leq cN\) (see \eqref{eq: 34.5 end ZN}) it follows, using also \eqref{eq: bound RN 48.5}, that  
\[
\mathbb{E}[|\tilde{R}_N|] \leq \frac{c}{N^2} + \frac{c}{N^2} \mathbb{E}[|Z^N|^2] \leq \frac{c}{N}.
\]  
This establishes the claim in \eqref{eq: claim end fp}, and the proof of the theorem is thus concluded.

\end{proof}

\subsubsection{Proof of Corollary \ref{cor: it}}
\begin{proof}
This result follows easily from Theorem \ref{th: fp}, along with the observation that  
\[
\hat{\theta}^{(it)}_{N,n} - \theta^0 = \hat{\theta}^{(it)}_{N,n} - \hat{\theta}_N^{(fp)} + \hat{\theta}_N^{(fp)} - \theta^0.
\]  
Using the definitions of \(\hat{\theta}^{(it)}_{N,n}\) and \(\hat{\theta}_N^{(fp)}\), and proceeding as in the derivation of \eqref{eq: lip FN}, we have indeed
\[
|\hat{\theta}^{(it)}_{N,n} - \hat{\theta}_N^{(fp)}| = |F_N(\hat{\theta}^{(it)}_{N,n-1}) - F_N(\hat{\theta}_N^{(fp)})| \le C_T |\hat{\theta}^{(it)}_{N,n-1} - \hat{\theta}_N^{(fp)}|.
\]  
By iterating this argument, we obtain  
\[
|\hat{\theta}^{(it)}_{N,n} - \hat{\theta}_N^{(fp)}| \le C_T^n |\hat{\theta}^{(it)}_{N,0} - \hat{\theta}_N^{(fp)}|.
\]  
This error is therefore negligible for establishing consistency, as \(C_T^n \to 0\) for \(n \to \infty\), which always holds since \(C_T < 1\). Additionally, it is negligible for proving asymptotic Gaussianity, provided that \(\sqrt{N} (C_T)^n \to 0\) as \(N, n \to \infty\), which is a condition stated in our theorem. The proof is thus complete.

\end{proof}

\section{Proof auxiliary results}{\label{s: proof technical}}
In this section, we present the proofs of all technical results that have been stated but not yet proven, as they were essential for establishing our main results. We begin with the proof of Lemma \ref{l: moments}, which collects various bounds on the moments of the processes under consideration.

\subsection{Proof of Lemma \ref{l: moments}}
\begin{proof}
Point 1. Let us begin by proving the boundedness of the moments. From the particle dynamics described in \eqref{eq: model}, we have for any \( i=1, \ldots, N \), \( 0 \le t \le T \), and \( q \ge 2 \):

\[
\E[|X_t^{i,N}|^q] = \E\left[ \left| X_0^{i,N} + \int_0^t  \langle \theta^0,\,b(X_s^{i,N}, \mu_s^N)\rangle \, ds + \sigma B_t^{i,H} \right|^q \right].
\]
Using Jensen's inequality, we can bound this expression as follows:
\[
\E[|X_t^{i,N}|^q] \leq c \E[|X_0^{i,N}|^q] + c \|\theta^0\| t^{q-1} \int_0^t \E[\|b(X_s^{i,N}, \mu_s^N)\|^q] \, ds + c \E[|B_t^{i,H}|^q].
\]
By Assumption \ref{as: moments}, we know that \( \E[|X_0^{i,N}|^q] < \infty \). Additionally, Assumptions \ref{as: lip} and \ref{as: bound} ensure that:
\begin{equation}{\label{eq: bound b}}
 \|b(X_s^{i,N}, \mu_s^N)\| \leq c \left( 1 + |X_s^{i,N}| + W_2(\mu_s^N, \delta_0) \right).   
\end{equation}
Applying Jensen’s inequality yields:
\begin{equation}{\label{eq: bound W}}
 \E[W_2^q(\mu_s^N, \delta_0)] \leq \E\left[\left( \frac{1}{N} \sum_{j=1}^N |X_s^{j,N}|^2 \right)^{\frac{q}{2}}\right] \leq \frac{1}{N} \sum_{j=1}^N \E[|X_s^{j,N}|^q] = \E[|X_s^{i,N}|^q],   
\end{equation}
where the last equality follows from the fact that the particles are identically distributed. Moreover, since the fractional Brownian motion \( B^{H,i} \) is an \( H \)-self-similar process {\change (see for example Point 2 of Proposition 2.2 in \cite{Ivan fBm})}, we have for any \( q > 0 \):

\begin{equation}{\label{eq: moments fBm}}
 \E\left[\sup_{t \in [0, T]} |B_t^{i,H}|^q \right] \leq c T^{qH}.   
\end{equation}
Putting everything together, we obtain:
\[
\E[|X_t^{i,N}|^q] \leq c + c \|\theta^0 \| t^{q-1} \int_0^t \E[|X_s^{i,N}|^q] \, ds + c t^{qH}.
\]
Applying Gronwall’s lemma, we conclude that \( \E[|X_t^{i,N}|^q] \leq c \), where \( c \) is a constant depending on \( t \) but uniform in \( N \). This establishes the boundedness of the moments for the interacting particles.

The boundedness of the moments for the associated McKean-Vlasov equation can be derived in a similar manner. Consequently, the proof of Point 1 is completed due to the bound in \eqref{eq: bound W}.\\
\\
Point 2: Let us proceed with the analysis of the increments of the interacting particle system. The proof again relies on their dynamics. For any \( 0 \leq s \leq t \leq T \), we have:
\[
X_t^{i,N} - X_s^{i,N} = \int_s^t  \langle \theta^0,\,b(X_u^{i,N}, \mu_u^N)\rangle \, du + \sigma B_t^{i,H} - \sigma B_s^{i,H}.
\]
By applying Jensen's inequality we directly obtain, for any $q \ge 2$:
\[
\E[|X_t^{i,N} - X_s^{i,N}|^q] \leq c (t-s)^{q-1} \int_s^t \E[\|b(X_u^{i,N}, \mu_u^N)\|^q] \, du + c \E[|B_t^{i,H} - B_s^{i,H}|^q],
\]
which simplifies, by \eqref{eq: bound b}, \eqref{eq: bound W}, and \eqref{eq: moments fBm}, to:
\begin{equation}{\label{eq: increments}}
\E[|X_t^{i,N} - X_s^{i,N}|^q] \leq c (t-s)^q + c (t-s)^{qH},    
\end{equation}
as desired. Furthermore, by Jensen's inequality, we also obtain:
\[
\E[W_2^q(\mu_t^N, \mu_s^N)] \leq \E\left[ \left( \frac{1}{N} \sum_{j=1}^N |X_t^{j,N} - X_s^{j,N}|^2 \right)^{\frac{q}{2}} \right] \leq \frac{1}{N} \sum_{j=1}^N \E[|X_t^{j,N} - X_s^{j,N}|^q] = \E[|X_t^{i,N} - X_s^{i,N}|^q],
\]
due to the identical distribution of the particles.
Thus, inequality \eqref{eq: increments} concludes the proof of Point 2. \\
\\
Point 3 is obtained in the same way as Point 2, by replacing the dynamics of the interacting particle system with those of the limiting independent particles.
    
\end{proof}

\subsection{Proof of Lemma \ref{l: deriv-bds}}
\begin{proof}
Point 1. The explicit computation of the Malliavin derivative for independent particles follows directly from the chain rule. Let us now move to the case of interacting particles. By employing the empirical projection representation of the drift coefficient as introduced in Definition \ref{def: empirical projection}, the system can be rewritten as follows: for \(i = 1, \dots, N\),
\[
X_t^{i,N} = X_0^{i,N} + \int_0^t  \langle \theta^0,\, b^{i,N}(X_r^{1,N}, \dots, X_r^{N, N})\rangle \, dr +  \sigma dB_t^{i,H}.
\]
The Lipschitz assumption on \(b\), provided in Assumption \ref{as: lip}, directly implies that the coefficient \(b^{i,N}\) is also Lipschitz. Applying results from \cite{NuaBook} for classical SDEs (such as Theorem 2.2.1), we immediately obtain that the Malliavin derivatives of \(X^i\) exist, are unique, and square-integrable. Furthermore, for \(s > t\), we have \(D_s^j X_t^{i,N} = 0\) \(\mathbb{P}\)-a.e. for all \(i, j = 1, \dots, N\).
Using the chain rule, the Malliavin derivative for \(0 \leq s \leq t \leq T\) is expressed as
\[
D_s^j X_t^{i,N} =  \sigma 1_{\{ i = j \}} + \int_s^t \langle\theta^0,\,\sum_{k = 1}^N \left(\partial_{x^k} b^{i,N}\right)(X_r^{1,N}, \dots, X_r^{N,N})\rangle D_s^j X_r^{k, N} \, dr.
\]
Next, we apply Proposition \ref{prop: deriv empirical measure} and revert the empirical projection maps to their original form, yielding
\[
D_s^j {X}^{i,N}_t =  \sigma 1_{\{i=j\}} +  \int_s^t \langle\theta^0,\,\Big( \partial_x b({X}^{i,N}_r, {\mu}^N_r) D_s^j {X}^{i,N}_r \, + \frac{1}{N}\sum_{k=1}^N (\partial_{\mu}b)({X}^{i,N}_r, \mu^N_r)(X^{k,N}_r) D_s^j X^{k,N}_r \Big)\rangle \, dr,
\]
noting that the derivative \(\partial_{x^i} b^{i,N}\) generates two terms, as both \({X}^{i,N}_r\) and \({\mu}^N_r\) depend on \(x^i\).
This concludes the proof of Point 1.\\
\\
Point 2. Observe that
    \begin{align*}
        &|D_s^i\bar{X}^i_t | \leq  \sigma  + c\int_s^t \|\partial_x b(\bar{X}^i_r, \bar{\mu}_r)\||D^i_s\bar{X}^i_r| dr \leq  \sigma  + cK \int_s^t  |D^i_s\bar{X}^i_r| dr,
    \end{align*}
    where $K$ is the constant from Assumption \ref{as: deriv}($b$). The Grönwall lemma now yields the desired result. As for the interacting case, using the dynamics of \( D_s^j{X}^{i,N}_t \), we have
\begin{equation}{\label{eq: dynamics DX}}
|D_s^j{X}^{i,N}_t| \leq c 1_{\{i=j\}} + c \int_s^t \left( |D_s^j{X}^{i,N}_r| +  \frac{1}{N} \sum_{k = 1}^N |D_s^j{X}^{k,N}_r| \right) dr,
\end{equation}
where we have used the boundedness of the derivatives of \( b \) and the fact that
\[
\left\|\frac{1}{N} \sum_{k = 1}^N (\partial_{\mu}b)({X}^{i,N}_r, \mu^N_r)(X^{k,N}_r)\right\| | D_s^j{X}^{k,N}_r| \leq \frac{c}{N} \sum_{k = 1}^N |D_s^j{X}^{k,N}_r|.
\]
This implies that, by averaging over \( i \),
\[
\frac{1}{N} \sum_{i = 1}^N |D_s^j{X}^{i,N}_r| \leq \frac{c}{N} + 2c \int_s^t \frac{1}{N} \sum_{i = 1}^N |D_s^j{X}^{i,N}_r| dr.
\]
Using Gronwall's lemma, as in Point 2 of Lemma \ref{l: deriv-bds}, we deduce that
\begin{equation}{\label{eq: bound sum DX}}
\sup_{t \in [s, T]} \frac{1}{N} \sum_{i = 1}^N |D_s^j{X}^{i,N}_r| \leq \frac{c}{N},
\end{equation}
where the constant \( c \) is uniform in \( N \), $t$ and $s$, though not in \( T \), having used $|t -s| \le T$. Substituting this into \eqref{eq: dynamics DX}, and recalling that we are considering the case \( i \neq j \), we get
\[
\sup_{t \in [s, T]}|D_s^j{X}^{i,N}_t| \leq c \int_s^T \sup_{u \in [s, r]} |D_s^j{X}^{i,N}_u| dr + \frac{c}{N}.
\]
A final application of Gronwall's lemma concludes the proof of the bound on $D_s^j X_t^i$ for $j \neq i$. For $j = i$ we clearly have
\begin{align} \label{inter-deriv}
        &|D_s^i{X}^{i,N}_t| \leq  \sigma  + c\int_s^t  |D^i_s{X}^{i,N}_r| dr+\int_s^t \frac{c}{N}\sum_{k=1}^N |D_s^i{X}^{k,N}_r| dr.
    \end{align}
    Plugging \eqref{eq: bound sum DX} into \eqref{inter-deriv} and applying Grönwall's lemma we conclude that 
$$|D_s^i{X}^{i,N}_t| \leq c(1 + \frac{1}{N}) \exp(c |t-s|) \le c,$$
with a constant $c$ that does not depend on $N$, $t$ or $s$ as $N \ge 1$ and $|t-s| \le T$. \\
\\
Point 3. Let us begin by proving the bound on \( D_u^i\bar{X}^i_t - D_v^i\bar{X}^i_t \).  
For \( v \le u \le T \), the dynamics in \eqref{eq: Malliavin indep} gives us the following:
\begin{align}{\label{eq: increment DXbar}}
D_u^i\bar{X}^i_t - D_v^i\bar{X}^i_t = \int_u^t \langle\theta^0,\,\partial_x b(\bar{X}^i_r, \bar{\mu}_r)\rangle (D_u^i\bar{X}^i_r - D_v^i\bar{X}^i_r) \, dr - \int_v^u \langle\theta^0,\,\partial_x b(\bar{X}^i_r, \bar{\mu}_r) \rangle D_v^i\bar{X}^i_r \, dr.
\end{align}
Using the boundedness of \( \partial_x b \) and \( D_v^i\bar{X}^i_r \), as given in Assumption \ref{as: deriv}($b$) and the previously proven Point 2, we can deduce:
\begin{equation}{\label{eq: increment 4}}
|D_u^i\bar{X}^i_t - D_v^i\bar{X}^i_t| \le c \int_u^t |D_u^i\bar{X}^i_r - D_v^i\bar{X}^i_r| \, dr + c |u-v|.
\end{equation}
Applying Gronwall's Lemma then provides the desired result.
Next, let us analyze \( D_u^i\bar{X}^i_t - D_u^i\bar{X}^i_s \). From the dynamics in \eqref{eq: Malliavin indep}, for \( u \le s \le t \le T \), we obtain:
\[
|D_u^i\bar{X}^i_t - D_u^i\bar{X}^i_s| \le c\int_s^t \|\partial_x b(\bar{X}^i_r, \bar{\mu}_r)\| |D_s^i\bar{X}^i_r| \, dr \le c |t-s|,
\]
which follows again from the boundedness of \( \partial_x b \) and \( D_s^i\bar{X}^i_r \).
For the final bound in Point 3, we use \eqref{eq: increment DXbar} to write:
\[
|D_u^i\bar{X}^i_t - D_v^i\bar{X}^i_t - (D_u^i\bar{X}^i_s - D_v^i\bar{X}^i_s)| \le c\int_s^t \|\partial_x b(\bar{X}^i_r, \bar{\mu}_r)\| |D_u^i\bar{X}^i_r - D_v^i\bar{X}^i_r| \, dr.
\]
Thus,
\[
|D_u^i\bar{X}^i_t - D_v^i\bar{X}^i_t - (D_u^i\bar{X}^i_s - D_v^i\bar{X}^i_s)| \le c |u-v| \int_s^t \|\partial_x b(\bar{X}^i_r, \bar{\mu}_r)\| \, dr \le c |u-v| |t-s|,
\]
which completes the proof of Point 3.\\
\\
Point 4. We now apply a similar approach for the interacting particle system. For any \( i, j \in \{1, \dots, N\} \), \( v \le u \), and \( s \le t \), we have:
\begin{align*}
 D_u^j{X}^{i,N}_t - D_v^j{X}^{i,N}_t & = \int_u^t \langle\theta^0,\,\partial_x b({X}^{i,N}_r, {\mu}^N_r)\rangle (D_u^j{X}^{i,N}_r - D_v^j{X}^{i,N}_r) \, dr \\
 & + \int_u^t \frac{1}{N}\sum_{k=1}^N \langle\theta^0,\,(\partial_{\mu}b)({X}^{i,N}_r, \mu^N_r)(X^{k,N}_r)\rangle (D_u^j{X}^{k,N}_r - D_v^j{X}^{k,N}_r) \, dr  \\
 & - \int_v^u \langle\theta^0,\,\partial_x b({X}^{i,N}_r, {\mu}^N_r)\rangle  D_v^j{X}^{i,N}_r \, dr\\
 &- \int_v^u \frac{1}{N}\sum_{k=1}^N \langle\theta^0,\,(\partial_{\mu}b)({X}^{i,N}_r, \mu^N_r)(X^{k,N}_r)\rangle D_v^j{X}^{k,N}_r \, dr.
\end{align*}
This leads to the inequality:
\begin{align}{\label{eq: increment D IPS 5}}
& |D_u^j{X}^{i,N}_t - D_v^j{X}^{i,N}_t| \\
& \le c \int_u^t |D_u^j{X}^{i,N}_r - D_v^j{X}^{i,N}_r| \, dr + c \int_u^t \frac{1}{N}\sum_{k=1}^N |D_u^j{X}^{k,N}_r - D_v^j{X}^{k,N}_r| \, dr + c |u - v|(1_{i = j} + \frac{1}{N}), \nonumber
\end{align}
using \eqref{eq: bound sum DX} for the last term.
Averaging over \( i \), we obtain:
\[
\frac{1}{N}\sum_{i=1}^N |D_u^j{X}^{i,N}_t - D_v^j{X}^{i,N}_t| \le c \int_u^t \frac{1}{N}\sum_{i=1}^N |D_u^j{X}^{i,N}_r - D_v^j{X}^{i,N}_r| \, dr + \frac{c}{N} |u-v|.
\]
Applying Gronwall's Lemma yields:
\begin{equation}{\label{eq: end increment IPS sum}}
\frac{1}{N}\sum_{i=1}^N |D_u^j{X}^{i,N}_t - D_v^j{X}^{i,N}_t| \le \frac{c}{N} |u-v|,
\end{equation}
which can now be substituted into \eqref{eq: increment D IPS 5} to obtain the desired result.
Finally, observe that, from Point 2, we have directly:
\[
|D_u^j{X}^{i,N}_t - D_u^j{X}^{i,N}_s| \le c \int_s^t \left[|D_u^j{X}^{i,N}_r| + \frac{1}{N}\sum_{k=1}^N |D_u^j{X}^{k,N}_r|\right] \, dr \le c \left(1_{i = j} + \frac{1}{N}\right) |t-s|.
\]
Lastly, we can verify that:
\begin{align*}
 &|D_u^j{X}^{i,N}_t - D_v^j{X}^{i,N}_t - (D_u^j{X}^{i,N}_s - D_v^j{X}^{i,N}_s)| \\
 & \le c \int_s^t \left[|D_u^j{X}^{i,N}_r - D_v^j{X}^{i,N}_r| + \frac{1}{N}\sum_{k=1}^N |D_u^j{X}^{k,N}_r - D_v^j{X}^{k,N}_r|\right] \, dr \\
 & \le c |u - v| |t-s|\left(1_{\{i=j\}} + \frac{1}{N}\right),
\end{align*}
using \eqref{eq: increment D IPS 5} and \eqref{eq: end increment IPS sum}. This concludes the proof of the lemma.

\end{proof}

\subsection{Proof of Lemma \ref{l: L2-bound-loclip}}
\begin{proof}
We prove the statement only for interacting particles. The proof for independent particles is completely analogous.\\
   As $H>\frac{1}{2}$, we have
    \begin{align*}
         \E [\|g_\cdot\|_{\mathcal H}^2] &= \E \left[\left(\int_0^T\int_0^T g({X}^{i,N}_s, {\mu}^{N}_s) g({X}^{i,N}_r, {\mu}^{N}_r)\phi(s,r)ds\,dr\right)\right] \\
         &= \int_0^T\int_0^T \E[g({X}^{i,N}_s, {\mu}^{N}_s) g({X}^{i,N}_r, {\mu}^{N}_r)]\phi(s,r)ds\,dr.
    \end{align*}
    {\change By applying the Cauchy-Schwarz inequality and leveraging the polynomial growth of \( g \), which follows from its local Lipschitz continuity and the fact that \( |g(0, \delta_0)| < \infty \), we obtain the bound  
\[
\E[g({X}^{i,N}_s, {\mu}^{N}_s) g({X}^{i,N}_r, {\mu}^{N}_r)] \leq c \sup_{s \in [0, T]} \E\left[(1 + |X_s^{i,N}| + W_2(\mu_s^N, \delta_0))^k\right],
\]  
for some \( k \geq 2 \). This quantity is finite due to the boundedness of the moments of the process, as stated in Lemma \ref{l: moments}, Point 1.}
    Since $\int_0^T\int_0^T |s-r|^{2H-2}ds\,dr$ is finite, we can conclude that $\E [\|g_\cdot\|_{\mathcal H}^2]$ is bounded. For the boundedness of $\E[\|D_\cdot g({X}^{i,N}_\cdot, {\mu}^{N}_\cdot)\|^2_{\mathcal H\otimes \mathcal H}] $ we write
    \begin{align*}
    &\E[\|D_\cdot g({X}^{i,N}_\cdot, {\mu}^{N}_\cdot)\|^2_{\mathcal H\otimes \mathcal H}] \\
    &\quad = \int_0^T\int_0^T \int_0^T\int_0^T \E [D_u g(X_s^{i,N}, \mu_s^N)D_vg(X_t^{i,N}, \mu_t^N) ]\phi(u,v) \phi(t,s)du\,dv\,dt\,ds
    \end{align*}
    and recall that for $u\leq s$
    \[D_u g(X_s^{i,N}, \mu_s^N)= \partial_xg(X_s^{i,N}, \mu_s^N)D_u X_s^{i,N} + \frac{1}{N}\sum_{k=1}^N (\partial_\mu g)(X_s^{i,N}, \mu_s^N)(X_s^{k,N})D_u X^{k,N}_s\]
    is uniformly bounded thanks to the Assumption \ref{as: deriv}($g$) and Lemma \ref{l: deriv-bds}, Point 2.

\end{proof}

\subsection{Proof of Lemma \ref{l: x0}}
\begin{proof}
Point 1.
Starting from the dynamics in \eqref{eq: dx0 X 7}, one can write:  
\begin{equation}\label{eq: bound x0 8}
|\partial_{x_0^j} {X}^{i, x_0^j}_t| \leq 1_{\{i=j\}} + \int_0^t | \partial_{x_0^j} {X}^{i, x_0^j}_r| \, dr + \frac{c}{N} \sum_{k=1}^N \int_0^t | \partial_{x_0^j} X^{k, x_0^j}_r| \, dr,
\end{equation}  
which, upon averaging over \(i\), yields:  
\[
\frac{1}{N} \sum_{i=1}^N |\partial_{x_0^j} {X}^{i, x_0^j}_t| \leq \frac{c}{N} + 2c \int_0^t \frac{1}{N} \sum_{i=1}^N |\partial_{x_0^j} {X}^{i, x_0^j}_r| \, dr.
\]  
Applying Grönwall's lemma shows that the average is bounded by \(\frac{c}{N}\). Substituting this bound back into \eqref{eq: bound x0 8} gives:  
\[
|\partial_{x_0^j} {X}^{i, x_0^j}_t| \leq c \left(1_{\{i=j\}} + \frac{1}{N}\right) + \int_0^t |\partial_{x_0^j} {X}^{i, x_0^j}_r| \, dr.
\]  
A final application of Grönwall's lemma concludes the proof. \\
\\
Point 2. To establish a bound on the second derivative with respect to the initial condition, we begin by formally computing it. Following the approach in Point 1 of Lemma \ref{l: deriv-bds} and transitioning to the empirical projections, it is straightforward to verify that  
\begin{align*}  
\partial^2_{x_0^j} X_t^{i, x_0^j + \tau} &= \int_0^t \bigg[ \partial_x^2 b (X_r^{i, x_0^j + \tau}, \mu_r^N)(\partial_{x_0^j} X_r^{i, x_0^j + \tau})^2 \\  
&\quad + \frac{1}{N} \sum_{k = 1}^N \partial_\mu(\partial_x b (X_r^{i, x_0^j + \tau}, \mu_r^N))(X_r^{k, x_0^j + \tau}) \partial_{x_0^j} X_r^{k, x_0^j + \tau}\partial_{x_0^j} X_r^{i, x_0^j + \tau} \\  
&\quad + \frac{1}{N} \sum_{\tilde{k} = 1}^N \partial_x(\partial_\mu b (X_r^{i, x_0^j + \tau}, \mu_r^N) (X_r^{\tilde{k}, x_0^j + \tau})) \partial_{x_0^j} X_r^{\tilde{k}, x_0^j + \tau}\partial_{x_0^j} X_r^{i, x_0^j + \tau} \\  
&\quad + \frac{1}{N^2} \sum_{k, \tilde{k} = 1}^N \partial_\mu (\partial_\mu b (X_r^{i, x_0^j + \tau}, \mu_r^N) (X_r^{\tilde{k}, x_0^j + \tau})) (X_r^{k, x_0^j + \tau}) \partial_{x_0^j} X_r^{\tilde{k}, x_0^j + \tau}\partial_{x_0^j} X_r^{k, x_0^j + \tau} \\  
&\quad + \partial_x b (X_r^{i, x_0^j + \tau}, \mu_r^N) \partial^2_{x_0^j} X_r^{i, x_0^j + \tau} \\  
&\quad + \frac{1}{N} \sum_{k = 1}^N \partial_\mu b (X_r^{i, x_0^j + \tau}, \mu_r^N) (X_r^{k, x_0^j + \tau}) \partial^2_{x_0^j} X_r^{k, x_0^j + \tau} \bigg] dr.  
\end{align*}  
From Point 1, we know that \( |\partial_{x_0^j} X_r^{k, x_0^j + \tau}| \le c (1_{\{i=j\}} + \frac{1}{N}) \). Consequently, using the boundedness of the first and second derivatives of \( b \) as assumed in Assumptions \ref{as: deriv} and \ref{as: second derivatives}, we obtain  
\begin{align}{\label{eq: second end}}  
|\partial^2_{x_0^j} X_t^{i, x_0^j + \tau}| &\le \int_0^t \bigg[ (1_{\{i=j\}} + \frac{1}{N})^2 + 2 \frac{c}{N} \sum_{k = 1}^N (1_{\{k=j\}} + \frac{1}{N})(1_{\{i=j\}} + \frac{1}{N}) \nonumber\\  
&\quad + \frac{c}{N^2} \sum_{k,\tilde{k}= 1}^N (1_{\{k=j\}} + \frac{1}{N})(1_{\{\tilde{k} = j\}} + \frac{1}{N}) + c |\partial^2_{x_0^j} X_r^{i, x_0^j + \tau}| + \frac{c}{N} \sum_{k = 1}^N |\partial^2_{x_0^j} X_r^{k, x_0^j + \tau}| \bigg] dr \nonumber\\  
&\le c(1_{\{i=j\}} + \frac{1}{N})^2 + \frac{c}{N}(1_{i = j} + \frac{1}{N}) + \frac{c}{N^2} \nonumber\\  
&\quad + c \int_0^t |\partial^2_{x_0^j} X_r^{i, x_0^j + \tau}| dr + \int_0^t \frac{c}{N} \sum_{k = 1}^N |\partial^2_{x_0^j} X_r^{k, x_0^j + \tau}| dr.  
\end{align}  
Averaging over \( i \), we find  
\[
\frac{1}{N} \sum_{i = 1}^N |\partial^2_{x_0^j} X_t^{i, x_0^j + \tau}| \le \frac{c}{N} + \int_0^t \frac{c}{N} \sum_{k = 1}^N |\partial^2_{x_0^j} X_r^{k, x_0^j + \tau}| dr,  
\]  
where the remaining terms are negligible. By applying Grönwall's lemma, we conclude that this quantity is bounded by \( \frac{c}{N} \).  
Substituting this result into the bound in \eqref{eq: second end}, we obtain  
\[
|\partial^2_{x_0^j} X_t^{i, x_0^j + \tau}| \le c(1_{\{i=j\}} + \frac{1}{N}) + c \int_0^t |\partial^2_{x_0^j} X_r^{i, x_0^j + \tau}| dr.  
\]  
Finally, another application of Grönwall's lemma yields the desired result.

\end{proof}

\end{document}